\documentclass[reqno, a4paper]{amsart}
\usepackage{amsmath, amssymb, enumerate, esint, tikz, xcolor}
\usepackage{bbm}
\usepackage{bm}
\usepackage[mathscr]{eucal}
\usepackage[top=2.5cm, bottom=2.5cm, left=2.5cm, right=2.5cm]{geometry}

\usepackage{mathtools}

\usepackage[foot]{amsaddr} 

\allowdisplaybreaks

\usepackage{comment}

\usepackage[colorlinks=true,linkcolor=teal, citecolor=teal]{hyperref}

\usepackage[nameinlink,capitalize]{cleveref}
\crefname{equation}{}{}
\crefname{theo}{Theorem}{Theorems}
\crefname{coro}{Corollary}{Corollaries}
\crefname{prop}{Proposition}{Propositions}
\crefname{lemm}{Lemma}{Lemmas}
\crefname{exam}{Example}{Examples}
\crefname{assum}{Assumption}{Assumptions}

\newtheorem{theo}{Theorem}[section]

\newtheorem{lemm}[theo]{Lemma}
\newtheorem{prop}[theo]{Proposition}
\newtheorem{prob}[theo]{Problem}

\theoremstyle{definition}
\newtheorem{defi}[theo]{Definition}
\newtheorem{assum}[theo]{Assumption}
\newtheorem{exam}[theo]{Example}
\newtheorem{rema}[theo]{Remark}

\numberwithin{equation}{section}

\newcommand{\bR}{\mathbb R}
\newcommand{\bN}{\mathbb N}

\newcommand{\bE}{\mathbb E}

\newcommand{\bP}{\mathbb P}
\newcommand{\bF}{\mathbb F}
\newcommand{\1}{\mathbbm{1}}
\newcommand{\bbm}{\mathbbm{m}}
\newcommand{\bI}{\mathbb I}

\newcommand{\bS}{\mathbb S}
\newcommand{\bV}{\mathbb V}

\newcommand{\bfD}{\mathbf D}

\newcommand{\bfL}{\mathbf L}

\newcommand{\bmZ}{\mathcal Z}

\newcommand{\bmC}{\bm{C}}
\newcommand{\bme}{\bm{\mathrm{e}}}

\newcommand{\cA}{\mathcal A}
\newcommand{\cB}{\mathcal B}

\newcommand{\cE}{\mathcal E}
\newcommand{\cF}{\mathcal F}

\newcommand{\cG}{\mathcal G}

\newcommand{\cL}{\mathcal L}

\newcommand{\cN}{\mathcal N}
\newcommand{\cS}{\mathcal S}
\newcommand{\cP}{\mathcal P}

\newcommand{\cO}{\mathcal O}
\newcommand{\cW}{\mathcal W}

\newcommand{\scrD}{\nabla}
\newcommand{\scrM}{\mathscr M}
\newcommand{\scrS}{\mathscr S}

\newcommand{\rmG}{\mathrm{G}}

\newcommand{\diag}{\mathrm{diag}}
\newcommand{\Tr}{\mathbf{tr}}

\newcommand{\vect}{\mathrm{vec}}
\newcommand{\ent}{\mathsf{Ent}}
\newcommand{\opt}{\mathrm{opt}}
\newcommand{\RO}{E}

\newcommand{\wt}{\widetilde}

\newcommand{\pr}{\otimes}
\newcommand{\tran}{\mathsf{T}}

\newcommand{\ep}{\varepsilon}
\newcommand{\e}{\mathrm{e}}

\newcommand{\supp}{\mathrm{supp}}
\newcommand{\Leb}{\bm{\lambda}}

\newcommand{\pd}{\partial}

\newcommand{\im}{\mathrm{i}}
\newcommand{\od}{\mathrm{d}}

\newcommand{\nq}{\!\!}

\usepackage{soul}

\def\b{\big}
\def\B{\Big}
\def\bb{\bigg}
\def\BB{\Bigg}

\def\<{\left<}
\def\>{\right>}

\begin{document}

\title[Entropy-Regularized Portfolio Optimization with Jumps]{Entropy-Regularized Mean-Variance Portfolio Optimization with Jumps}

\author[Christian Bender]{Christian Bender$^{1}$}
\address{$^1$\rm{Department of Mathematics, Saarland University, Germany}}
\email{bender@math.uni-saarland.de}

\author[Nguyen Tran Thuan]{Nguyen Tran Thuan$^{1,2}$}
\email{nguyen@math.uni-saarland.de}
\address{$^2$\rm{Department of Mathematics, Vinh University, 182 Le Duan, Vinh, Nghe An, Viet Nam}}
\email{thuan.tr.nguyen@gmail.com}

\thanks{\textit{The second author dedicates this article to Professor Tr\!\`\!\!\^an L\^{\!\!\!\d{o}}c H\`ung for his inspiring life story.} \\[0.2cm] The first preprint version of this manuscript has been made available in December 2023. To the best of our knowledge, it is the first paper discussing the exploratory control problem for reinforcement learning in continuous time with controlled jumps. In the meantime different aspects of reinforcement learning for jump diffusions have been addressed in \cite{GLZ24} and in our recent work \cite{BT24}.}

\date{\today}

\begin{abstract}
Motivated by the trade-off between exploitation and exploration in reinforcement learning, we study a continuous-time entropy-regularized mean-variance portfolio selection problem in the presence of jumps. We propose an exploratory SDE for the wealth process associated with multiple risky assets which exhibit L\'evy jumps. In contrast to the existing literature, we study the limiting behavior of the natural discrete-time formulation of the wealth process associated with a randomized control in order to derive the continuous-time dynamics.  We then show that an optimal distributional control of the continuous-time entropy-regularized exploratory mean-variance problem is still Gaussian despite being in jump models. Moreover, the respective optimal wealth process solves a linear SDE whose representation is explicitly obtained.

\bigskip


\medskip

\end{abstract}

\maketitle


\section{Introduction}

\subsection{The problem} The mean-variance (MV) portfolio optimization problem pioneered by Markowitz \cite{Ma52} is one of the most popular  criteria in the portfolio selection theory due to its simple and natural formulation in dealing with the two  important aspects of investment, namely, risk and return. In the MV model, investors aim to minimize the variance, which quantifies the risk, of the terminal wealth of their portfolios while targeting a prespecified expected value of the terminal wealth. This criterion therefore effectively reflects a trade-off between the risk and expected return in an intuitive way. After Markowitz's foundational works, the MV approach has attracted considerable attention with numerous extensions and applications. For example, among other works in the continuous-time setting when the financial market is driven by a multidimensional Brownian motion, Zhou and Li \cite{ZL00} investigate the MV problem in terms of stochastic linear-quadratic (LQ) optimization using an embedding method. After that, Li et al. \cite{LZL02} introduce the Lagrange multiplier method to transform the MV problem to an unconstrained stochastic LQ control problem so that standard techniques are applicable. As the literature on the MV criterion is vast, we refer the reader to \cite{ZLG18} for a review on this topic.

The classical model-based MV problem, where model assumptions are predescribed, has been fairly well investigated and quite completely solved in various settings with analytical solutions. To apply these results in practice, one usually needs to estimate model parameters based on historical data of the underlying asset prices accumulated during trading. Nevertheless, it is widely acknowledged that it is difficult to estimate those parameters with an applicable accuracy, and furthermore, classical optimal MV strategies frequently exhibit high sensitivity to those parameters which then might become practically sub-optimal due to estimation error.

In recent years, reinforcement learning (RL) methods, which have increasingly attracted more attention in quantitative finance, become a promising approach to overcome those practical difficulties.  By and large, RL algorithms iteratively execute randomized controls for some period (or, episode) and apply the data which has been collected over the previous periods to update the unknown model parameters and the randomized control, see, e.g.,  \cite{JZ22, JZ23,STZ23} for RL algorithms in a continuous-time stochastic control setting. The randomization of the controls reflects the trade-off between  exploration (learning the unknown investment environment) and exploitation (optimizing adaptively to the updated model parameters). Thus, RL algorithms can produce (nearly) optimal solutions without the need of statistically estimating the model parameters beforehand. The reader is referred to \cite{HXY23} for an overview to recent developments and applications of RL in finance.

The iterative construction of the randomized controls in the algorithms mentioned above relies on an entropy-regularized formulation of the stochastic control problem. Here, the entropy regularization rewards exploration and leads to the optimality of distribution-valued (or, relaxed) controls. Recently, Wang and Zhou \cite{WZ20} introduced such an entropy-regularized exploratory SDE framework for the MV problem in a Black--Scholes environment.  To be more precise and for easier explanation, let us introduce some notations. Let $T>0$ be a fixed finite time horizon and $W = (W_t)_{t \in [0, T]}$ a standard $1$-dimensional Brownian motion. The exploratory SDE for the wealth process $X^\pi = (X^\pi_t)_{t \in [0, T]}$ under an admissible control $\pi = (\pi_t)_{t \in [0, T]}$, which is a \textit{distribution-valued} stochastic process and where $\pi_t$ is the probability density function of the exploration law at time $t$, is heuristically derived and has the following form
\begin{align}\label{eq:intro:exploratory-SDE-WZ20}
	\od X^{\pi}_t =\mu_t b \od t + \sqrt{\mu_t^2 + \sigma_t^2} \, a\od W_t.
\end{align}
Here, the drift $b \in \bR$ and volatility $a>0$ are unknown constants, $\mu_t: = \int_{\bR} u \pi_t(u) \od u$ represents the mean and $\sigma_t^2: = \int_{\bR} u^2 \pi_t(u) \od u - \mu_t^2$ the variance of the distribution of exploration at time $t$. We refer to \cite{WZ20, WZZ20} for the motivation and derivation of \eqref{eq:intro:exploratory-SDE-WZ20}. To encourage and quantify the exploration process, Wang and Zhou \cite{WZ20} incorporate a \textit{differential entropy} term to the objective function and the classical MV problem then becomes an \textit{entropy-regularized exploratory MV problem}. The authors then prove that the optimal feedback distributional control is Gaussian with time-decaying variance. Moreover, via a simulation study it is also illustrated in \cite{WZ20} that the RL approach for solving the MV problem significantly improves some other methods such as the traditional maximum likelihood estimate (MLE) and the deep deterministic policy gradient (DDPG). The approach  as in \cite{WZ20, WZZ20} has been extended in various contexts, see, e.g., \cite{DDJ23, GXZ22, WL24}.

It is, however, widely acknowledged that models with jumps are more appropriate to describe the fluctuation of asset prices, see, e.g., \cite{AJ14, CT03}. Following this direction, many researchers have extensively studied the classical MV problem and its variants in several jump models, see, e.g., \cite{JMSS12, Li05, OS19} and the references therein. Then a question naturally arises: \textit{How would the continuous-time entropy-regularized exploratory MV problem and its solutions be like if the asset prices exhibit jumps?} To address this question, one first needs to describe the exploratory SDE with jumps for the respective wealth process.

In contrast to the models built upon the Brownian framework by Wang and Zhou \cite{WZ20} and by Wang et al. \cite{WZZ20}, where the exploratory SDE for the wealth/controlled process can be heuristically inferred from knowing its first two conditional moments only, models with jumps are quite involved as, in general, one has to test against various other functions rather than the linear and quadratic functions to detect the distributional behavior of jumps. In fact, these test functions essentially depend on the jump activities of the underlying asset price process. Hence, the derivation for the exploratory SDE based on first two moments in \cite{WZ20, WZZ20} is seemingly not applicable for jump models, at least in a straightforward way. To deal with this problem, we exploit the linear dependence on controls of the wealth process and propose a different argument to derive the exploratory SDE.

\subsection{Our contributions and discussions}

Let $D \in \bN$ and assume that the log-price process of $D$ risky assets is a weak solution of an SDE driven by a $D$-dimensional L\'evy process $L$. Here $L$ includes, but not necessarily simultaneously, a Brownian motion $W$ and an independent pure-jump L\'evy process $J$, both are $D$-dimensional. Except the square integrability, there are no additional assumptions imposed on the L\'evy measure.

\subsubsection{Continuous-time exploratory SDE with L\'evy jumps} To derive an exploratory SDE for the wealth process, we begin with a discrete-time dynamic of the wealth under an exploration procedure, see \cref{sec:exploration-procedure}. In \cite{WZ20, WZZ20}, the authors  first average out realizations of distributional controls on each discrete-time sub-interval using a law of large numbers, and then combine them all together to infer the dynamic on entire $[0, T]$. Here, unlike the argument in \cite{WZ20, WZZ20}, we first explicitly model randomized controls on discrete-time partitions  of  $[0, T]$ and identify a family of discrete-time integrators which incorporate the additional  ``exploration noise''.
 To do that, we need to handle the additional randomness caused by exploration differently for the Brownian and for the jump component which can be roughly described as follows:

\begin{itemize}
	\itemsep2pt
	\item For the Brownian part, thanks to the linear structure with respect to the control, one can (partially) separate the original randomness caused by the asset prices and the  randomness caused by exploration in an appropriate way, see \cref{discrete-BM-part}. 
	
	\item For the jump component, we employ a suitable $D^2$-dimensional random measure to simultaneously capture both sources of randomness, see \cref{discrete-jump-part}.
\end{itemize}
 Then, by refining the discrete time points, we show in \cref{thm:limit-distribution-integrator} below that the stochastic \textit{integrators} of our discrete-time scheme converge in distribution to a multidimensional L\'evy process. This limit theorem gives rise to a natural continuous-time formulation of the exploratory control problem with entropy regularization. Note that randomized controls on discrete-time grids have recently been considered in \cite{DJ24, STZ23} for diffusion models. However, \cite[Theorem 2.2]{STZ23} investigates the limiting behavior of the cost of such controls and \cite[Lemma 4]{DJ24} describes the convergence of the optimal control density, while we apply this discretization to infer the structure of the continuous-time ``exploration noise''.  
 
 We also remark that the heuristic passage to the limit in the existing literature \cite{WZ20, WZZ20} only yields information about the conditional mean and covariance of the continuous-time controlled system. It, thus, allows for many different SDE representations, even in the case of no jumps, as discussed below. In contrast, our derivation identifies a specific SDE formulation, which we consider a natural choice for modeling exploration in the continuous-time framework. Indeed, as discussed in \cref{sec:sample_state_process} below, our formulation of the exploratory SDE, which is derived from discrete-time randomized controls, is closely related to the \emph{sample state process}, which is a key object for the design of learning algorithms in the  recent continuous-time RL literature \cite{JZ22,JZ23}. It can even be interpreted as a mathematically rigorous reformulation of this sample state process which avoids the use of an independent identically distributed sampling mechanism indexed by continuous time.

\subsubsection{Problem formulation in multidimensional setting}
We consider $D$ risky assets with jumps and derive the continuous-time dynamics of the wealth process with exploration, see SDE \eqref{eq:exploratory-SDE-jump} and \cref{rema:on-exploratory-SDE} for further discussion.

Let us compare our exploratory SDE with other works in the case of no jumps. Since we use a different argument, our exploratory SDE unsurprisingly takes a different form from \eqref{eq:intro:exploratory-SDE-WZ20} in \cite{WZ20}. If $D = 1$, then the dynamic of wealth under a distributional control $\pi$ in our setting particularly  becomes
\begin{align}\label{eq:intro:exploratory-SDE-BN23}
	\od X^{\pi}_t = \mu_t b\od t + \mu_t a\od W_t +  \sigma_t a\od \cW_t,
\end{align}
where $\cW$ is a $1$-dimensional Brownian motion independent of $W$. We notice that $X^\pi$ in \eqref{eq:intro:exploratory-SDE-WZ20} and in \eqref{eq:intro:exploratory-SDE-BN23} have the same distribution. However, differently from \eqref{eq:intro:exploratory-SDE-WZ20}, in our SDE \eqref{eq:intro:exploratory-SDE-BN23} the exploration randomness represented by $\cW$ is separated from the noise $W$ caused by asset prices. One also remarks that the SDE in form of \eqref{eq:intro:exploratory-SDE-BN23} has been recently considered in \cite{DDJ23, WL24}. Nevertheless, when $D >1$, the authors in \cite{DDJ23, WL24} use an additional $1$-dimensional Brownian motion to model the exploration (i.e. $\cW$ is $1$-dimensional), while, according to our analysis, it suggests to use a $D^2$-dimensional Brownian motion (i.e. $\cW$ is $D^2$-dimensional).

\subsubsection{Optimal distributional control and wealth process} Following \cite{WZ20}, we first use the Lagrange multiplier method to transform the exploratory MV problem to an entropy-regularized quadratic-loss control problem and then apply the dynamic programing principle to find its solutions. 

We show in \cref{thm:verification} that, despite the presence of jumps, among admissible distributional controls which are not necessarily in the feedback form,  an optimal Gaussian control in feedback form can be obtained. As a feature of our approach, the respective optimal wealth process satisfies a \textit{linear} SDE (see \eqref{eq:solution-wealth}) which allows us to find its expression in a closed-form (see \eqref{eq:explicit-solution-optimal-SDE} and \cite{BT23S}). As a consequence, the Lagrange multiplier is also explicitly obtained (see \eqref{eq:explicit-Lagrange-multiplier} and \cite{BT23S}). Moreover, the value function has a quadratic form with respect to the wealth variable whose coefficients are solutions to a system of partial integro-differential equations (PIDEs). In the special case of no jumps and $D = 1$ and with constant coefficients, our value function coincides, of course, with that in \cite{WZ20}, see \cref{exam:constant-coeff-D=1}.

\subsection{Structure of the article} 
In \cref{sec:preliminaries}, we introduce the notation and recall the classical MV problem. The derivation of the continuous-time exploratory SDE with L\'evy jumps is presented in \cref{sec:Exploratory-SDE-jump}. In \cref{sec:entropy-exploratory-MV}, we study the entropy-regularized exploratory MV problem, investigate its closed-form solutions, and discuss the Lagrange multipliers. \cref{sec:proof-weak-convergence} is devoted to present the proof of weak convergence of the discrete-time integrators (\cref{thm:limit-distribution-integrator}).

\section{Preliminaries}
\label{sec:preliminaries}
\subsection{Notations}

Let $D \in \bN: = \{1, 2,\ldots\}$. For $a, b \in \bR$, we use the usual notations  $a\wedge b : = \min\{a, b\}$ and $a\vee b: = \max\{a, b\}$. For $a<b$, let $\int_a^b : = \int_{(a, b]}$. Notation $\log$ indicates the natural logarithm. Sub-indexing a symbol by a label means the place where that symbol appears. We also use the conventions $\inf \emptyset : = \infty$ and $\sum_{i \in \emptyset} = \int_{\emptyset} : = 0$.

\subsubsection{Vector spaces and matrices}
Let $\|\cdot\|$ be the usual Euclidean norm and $(\bme_d)_{d=1}^D$ the natural basis in $\bR^D$. For $r>0$, we set $B_D(r): = \{x \in \bR^D : \|x\| < r\}$  and  $B_D^c(r): = \bR^D\backslash B_D(r)$.

All vectors are written in the column form. For a vector $x$ we use the
notation $x^{(i)}$ or $[x]^{(i)}$ to denote its $i$-th component. For a matrix $A$,
\begin{itemize}
	\item $A^{(i, j)}$ or $[A]^{(i, j)}$ is the element in the $i$-th row and $j$-th column of $A$;
	
	\item if $A$ is a $D\times D$ matrix, then $\Tr[A]$, $\det(A)$, $A^{-1}$ denote the trace, determinant and inverse of $A$ respectively. Let $\diag(A): = \diag(A^{(1,1)}, \ldots, A^{(D, D)})$ denote the diagonal matrix with diagonal entries $A^{(1,1)}, \ldots, A^{(D, D)}$;
	
	\item the usual Euclidean/Frobenius norm of $A$ is also denoted by $\|A\|$, i.e. $\|A\|  := \sqrt{\Tr[A^\tran A]}$.
\end{itemize}
 Notation $I_D$ means the $D\times D$ identity matrix. We also use the following classes of matrices:
\begin{itemize}
	\item $\bR^{D \times D'}$ denotes the family of all real matrices with size $D\times D'$;
	
	\item $\bS^D$ (resp. $\bS^D_+$, $\bS^D_{++}$) is the family of all symmetric (resp. positive semidefinite, positive definite) $A \in \bR^{D\times D}$. For $A \in \bS^D_+$, denote by $A^{\frac{1}{2}} \in \bS^D_+$ the unique square root of $A$, i.e. $A^{\frac{1}{2}} A^{\frac{1}{2}} = A$. If $A \in \bS^D_{++}$, then we let $A^{-\frac{1}{2}}: = (A^{\frac{1}{2}})^{-1}$.
	
	\item $\cO_D$ consists of all orthonormal $O \in \bR^{D\times D}$, i.e. $O^\tran O = I_D$.
\end{itemize}
For $A \in \bR^{D\times D'}$, denote by $\vect(A)$ the vectorization of $A$ defined  as an element of $\bR^{DD'}$ by  stacking the columns of $A$ on top of one another, i.e.
$$\vect(A) : =  (A^{(1,1)}, \ldots, A^{(D, 1)}, A^{(1,2)}, \ldots, A^{(D, 2)}, \ldots, A^{(1, D')}, \ldots, A^{(D, D')})^\tran.$$
For (column) vectors $x_1, \ldots, x_n$ with possibly different sizes, $\vect(x_1, \ldots, x_n)$ means the vector obtained by stacking $x_i$ on top of $x_{i+1}$, $1 \le i \le n-1$. 
To shorten notation at some places we also use the Kronecker product $\pr \colon \bR^D \times \bR^{D'} \to \bR^{DD'}$ defined by
$$x\pr y : = \vect(x^{(1)} y, \ldots, x^{(D)}y).$$
One notices that the operator $\pr$ is bilinear and $\|x \pr y\| = \|x\| \|y\|$.

\subsubsection{Function spaces}

For a function $f \colon \bR^D \to \bR$, we use the following notations:
\begin{itemize}
	\item $\|f\|_\infty: = \sup_{x \in \bR^D} |f(x)|$;
	
	\item $\pd f$ and $\pd^2 f$ denote usual partial derivatives of $f$ with respect to scalar components;
		
	\item $\scrD f$ and $\scrD^2 f$ denote the gradient and the Hessian of $f$ respectively, and $\|\scrD f\|_\infty^2 : = \sum_{d=1}^D \|\pd_d f\|_\infty^2$, $\|\scrD^2 f\|_\infty^2 : = \sum_{d, d' =1}^D  \|\pd^2_{d, d'} f\|_\infty^2$, where partial derivatives $\pd_d f: = \pd_{x^{(d)}} f$ and $\pd^2_{d, d'} f: = \pd^2_{x^{(d)} x^{(d')}} f$; 
	
	\item When $f$ has several (multivariate) components, we use $\scrD_y f$ and $\scrD^2_{yy} f$ to indicate the gradient and Hessian of $f$ with respect to component $y$. If $x$ is a scalar component and $y$ is a multivariate component, then we write $\scrD^2_{xy} : = (\pd^2_{xy^{(1)}}, \ldots, \pd^2_{xy^{(D)}})^\tran$.

	\item $\supp(f)$ stands for the support of $f$, i.e. the closure of  $\{x \in \bR^D : f(x) \neq 0\}$.
\end{itemize}
For $k =1, 2,\ldots$, denote by $C^k(\bR^D)$ the family of all $k$ times continuously differentiable functions on $\bR^D$. $C^k_b(\bR^D)$ consists of all bounded $f \in C^k(\bR^D)$ with bounded derivatives (up to the $k$-th order) and $C^\infty_b(\bR^D) : = \cap_{k \ge 1} C^k_b(\bR^D)$.  $C^k_c(\bR^D)$ denotes the family of all $f \in C^k(\bR^D)$ with compact support. We let $f \in C^{1,2}([0, T] \times \bR^D)$ if $f$ is (resp. twice) continuously differentiable with respect to $t \in [0, T]$ (resp. to $y \in \bR^D$) and its partial derivatives are jointly continuous.

\subsection{Stochastic basis} Let us fix a time horizon $T \in (0, \infty)$. Assume that $(\Omega, \cF, \bP; (\cF_{t})_{t \in [0, T]})$ satisfies the usual conditions, which means that $(\Omega, \cF, \bP)$ is a compete probability space, the filtration $(\cF_t)_{t \in [0, T]}$ is right-continuous and $\cF_0$ contains all $\bP$-null sets. This allows us to assume that every $(\cF_t)_{t \in [0, T]}$-adapted local martingale has \textit{c\`adl\`ag} (right-continuous with finite left limits) paths. For a random variable $\xi$, the expectation, variance, and conditional expectation given a sub-$\sigma$-algebra $\cG \subseteq \cF$, if it exists under $\bP$, is respectively denoted by $\bE[\xi]$, $\bV[\xi]$, and $\bE[\xi|\cG]$. We also use $\bfL_p(\bP): = \bfL_p(\Omega, \cF, \bP)$.

For a  c\`adl\`ag process $X = (X_t)_{t \in [0, T]}$, we denote $\Delta X_t : = X_t - X_{t-}$ for $t \in [0, T]$, where $X_{0-} : = X_0$ and $X_{t-} : = \lim_{t>s \uparrow t} X_s$ for $t \in (0, T]$. For a time index set $\bI \subseteq [0, \infty)$ and for processes $X = (X_t)_{t \in \bI}$, $Y = (Y_t)_{t \in \bI}$, we write $X = Y$ to indicate that $X_t = Y_t$ for
all $t \in \bI$ a.s., and the same meaning applied when the relation ``='' is replaced by some other standard
relations such as ``$\le$'', ``$>$'', etc.

 We refer to \cite{Pr05} for unexplained notions such as semimartingales, (optional) quadratic covariation $[X, Y]$ and conditional quadratic covariation $\<X, Y\>$ of semimartingales $X$, $Y$.

\subsection{Multidimensional L\'evy process}
An $\bR^D$-valued process $L = (L_t)_{t \in [0, T]}$ is called a  L\'evy process if it has independent and stationary increments, has c\`agl\`ag paths with $L_0 = 0$ a.s. The distributional property of $L$ is characterized by the L\'evy--Khintchine formula (see, e.g., \cite[Theorem 1.2.14]{Ap09}), for $t \in [0, T]$ and $u \in \bR^D$,
\begin{align*}
	\bE[\e^{\im u^\tran L_t}] = \e^{- t \kappa(u)}
\end{align*}
where the \textit{characteristic exponent} $\kappa$ is given, for $u \in \bR^D$, by
\begin{align*}
	\kappa(u) = - \im u^\tran b + \frac{u^\tran Au}{2} - \int_{z\neq 0} (\e^{\im u^\tran z} - 1 - \im u^\tran z \1_{\{\|z\|\le 1\}}) \nu(\od z).
\end{align*}
The characteristic triplet $(b, A, \nu)$ associated with the canonical truncation function $h(z): = z\1_{\{\|z\| \le 1\}}$ is deterministic and consists of the \textit{drift coefficient} $b \in \bR^D$, the \textit{Gaussian covariance matrix}  $A \in \bS^D_+$, and the \textit{L\'evy measure} $\nu$, i.e. a measure on $\cB(\bR^D\backslash\{0\})$ with $\int_{z \neq 0} (\|z\|^2 \wedge 1) \nu(\od z) <\infty$.  We call $L$ a \textit{Gaussian L\'evy process} if $\nu \equiv 0$, and call $L$ a \textit{purely non-Gaussian L\'evy process} if $A =0$.

\subsection{Classical continuous-time MV portfolio selection}\label{sec:classical-MV} Assume that the price process of a risk-less asset $S^{(0)} = (S^{(0)}_t)_{t \in [0, T]}$ and $D$ risky assets $S = (S_t)_{t \in [0, T]}$ are governed by the following SDEs
\begin{align*}
	\od S^{(0)}_t & = r S^{(0)}_t \od t, \quad S^{(0)}_0 = 1,\\
	\od S_t^{(d)} & = S_{t-}^{(d)} \od R_t^{(d)}, \quad S_0^{(d)}: = s_0^{(d)} >0, \quad d = 1, \ldots, D.
\end{align*} 
where the interest rate $r \ge 0$ is given and the (stochastic) log-price process $R$ is described by \eqref{SDE-classical-MV} and \eqref{eq:generator-Y} below. In the context of stock price modeling, it is natural to  assume the condition $\Delta R^{(d)} > -1$ on the jump sizes, which  ensures that the stock prices $S^{(d)}$ stay strictly positive. This condition is, however, not required to obtain the main results, and so we do not impose it.

An investment strategy in $D$ risky assets is expressed by a predictable $\bR^D$-valued process $H$ where $H^{(d)}_t$ is the  \textit{discounted dollar amount} invested in the $d$-th risky asset at time $t-$, i.e instantly before time $t$. The resulting discounted wealth process $X^H = (X^H_t)_{t \in [0, T]}$ associated with $H$ can be written as
\begin{align}\label{SDE-classical-MV}
	\od X^H_t = \sum_{d = 1}^D H_t^{(d)} (\od R_t^{(d)} - r\od t) =: H_t^\tran \od Y_t,
\end{align} 
where $X^H_0 = x_0 \in \bR$ is the given initial wealth. From now we will work with the driving process $Y$ and the discounted wealth $X^H$ as in \eqref{SDE-classical-MV}.

Assume that the log-price of $D$ underlying (discounted) risky assets is represented by a c\`adl\`ag and adapted process $Y = (Y_t)_{t \in [0, T]}$ which is Markovian whose infinitesimal generator is given, for sufficiently smooth $f$, by
\begin{align}\label{eq:generator-Y}
	(\cL_Y f)(y)  & = b(y)^\tran \scrD f(y) + \frac{1}{2} \Tr[A(y) \scrD^2 f(y)]  + \int_{z \neq 0} \B(f(y + \gamma(y) z) - f(y) - \scrD f(y)^\tran \gamma(y) z \B) \nu(\od z).
\end{align}
Here $\nu$ is a square integrable L\'evy measure and the coefficients $b\colon \bR^D \to \bR^D$, $A \in \bS^D_+$, and $\gamma\colon \bR^D \to \bR^{D\times D}$ satisfy standard assumptions which will be specified later in \cref{sec:setting}.

The classical Markowitz MV portfolio selection problem, parameterized by $\hat{z} \in \bR$, is then formulated as
\begin{align}\label{classical-MV-constraint}
\begin{cases}
	\min_H \bV[X^H_T]\\
	\mbox{subject to } X^H \mbox{ given in } \eqref{SDE-classical-MV} \mbox{ and }\bE[X^H_T] = \hat{z},
\end{cases}
\end{align}
where the minimum is taken over admissible $H$ which will be specified in our setting later. To deal with the constraint $\bE[X^H_T] = \hat{z}$ in \eqref{classical-MV-constraint}, we follow \cite{ZL00, WZ20} to consider the objective function parameterized by $w \in \bR$,
\begin{align*}
	\bV[X^H_T] - 2 w (\bE[X^H_T] - \hat{z}),
\end{align*} 
which is equal to 
$$\bE[(X^H_T - w)^2] - (\hat{z} - w)^2.$$
Then, to solve \eqref{classical-MV-constraint}, we consider the following \textit{unconstrained} quadratic-loss minimization problem parameterized by $w$,
\begin{align}\label{classical-MV-non-constraint}
		\begin{cases}
		\min_{H} \bE[(X^H_T - w)^2]\\
		\mbox{subject to } X^H \mbox{ given in } \eqref{SDE-classical-MV}.
	\end{cases}\end{align}
Once \eqref{classical-MV-non-constraint} is solved with a minimizer $H^*(w)$, which  depends on $w$, we let $\hat{w}$ be the value such that the constraint $\bE \b[X^{H^*(\hat{w})}_T\b] = \hat{z}$ is satisfied. Then such an $H^*(\hat{w})$ solves the original problem \eqref{classical-MV-constraint}, and $\hat{w}$ is called the \textit{Lagrange multiplier}\footnote{The Lagrange multiplier actually is $2\hat{w}$, but we use $\hat{w}$ to slightly simplify the presentation.}.

\section{Exploratory SDE with L\'evy jumps}
\label{sec:Exploratory-SDE-jump}
\subsection{Setting}\label{sec:setting}

Let us fix $D \in \bN$ and set $\RO: =  \bR^D\backslash\{0\}$. Let $\varphi_D$ be a probability density of $\xi \sim \cN(0, I_D)$ where $\cN(0, I_D)$ is the $D$-dimensional Gaussian distribution with zero mean and covariance $I_D$.

For $b, A, \gamma$ and $\nu$ appearing in \eqref{eq:generator-Y} we assume throughout this article the following:

\begin{assum}\label{assum:SDE-coefficients}
	The L\'evy measure $\nu$ and coefficients $b\colon \bR^D \to \bR^D$, $a, \gamma \colon \bR^D \to \bR^{D \times D}$ ,  $A: = aa^\tran \in \bS^D_+$ satisfy:
	\begin{enumerate}[\quad (a)]
		\itemsep3pt
		\item (Square integrability) $\nu$ is square integrable on $\RO$, i.e. $\int_{\RO} \|e\|^2 \nu(\od e) <\infty$;
		
		\item (Growth condition) $\|b(x)\| + \|a(x)\| + \|\gamma(x)\| \le C_1 (1+\|x\|)$ for all $x \in \bR^D$;
		
		\item (Lipschitz condition) $\|b(x) - b(y)\| + \|a(x) - a(y)\| + \|\gamma(x)- \gamma(y)\| \le C_2 \|x -y\|$ for all $x, y \in \bR^D$;
		
		\item (Non-degeneration) $\Sigma(y) : =  A(y)  + \gamma(y) \int_{\RO} e e^\tran \nu(\od e) \gamma(y)^\tran \in \bS^D_{++}$ for all $y \in \bR^D$.
	\end{enumerate}
\end{assum}

\subsection{Continuous-time dynamic of the wealth process with exploration: A heuristic approach}
Let $W = (W_t)_{t \in [0, T]}$ be a $D$-dimensional standard Brownian motion, and $J = (J_t)_{t \in [0, T]} \subseteq \bfL_2(\bP)$ a purely non-Gaussian L\'evy process which is independent of $W$ and has the following L\'evy--It\^o decomposition (see, e.g., \cite[Theorem 2.4.26]{Ap09})
\begin{align*}
	J_t : = \int_0^t \nq \int_{\RO} e \wt N(\od s, \od e). 
\end{align*} 
Here $\wt N$ is the compensated Poisson random measure of $J$  associated with L\'evy measure $\nu$. Since $b, a, \gamma$ in \cref{assum:SDE-coefficients} are sufficiently regular, it is known that the following SDE has a unique (strong) solution in $\bfL_2(\bP)$ (see, e.g., \cite[Theorem 3.1]{Ku04}),
\begin{align*}
	\od Y_t = b(Y_{t-})\od t  + a(Y_{t-}) \od W_t +  \gamma(Y_{t-}) \od J_t, \quad Y_0 = y_0 \in \bR^D,
\end{align*}
which admits $\cL_Y$ provided in \eqref{eq:generator-Y} as the Markov generator. 

Let $\{\Pi_n\}_{n \ge 1}$ be a sequence of partitions of $[0, T]$, where $\Pi_n = \{0=:t^n_0 < t^n_1 < \ldots < t^n_n:=T\}$. Denote $\Delta t^n_i: = t^n_i - t^n_{i-1}$  and assume that $|\Pi_n| : = \max_{1\le i \le n}\Delta t^n_i\to 0$ as $n \to \infty$. To shorten the presentation at some places, for a process $(P_{t})_{t \in [0, T]}$, we also use the notations
\begin{align*}
	P_{n, i}: = P_{t^n_i} \quad \mbox{and} \quad \Delta_{n, i} P : = P_{n, i} - P_{n, i-1}.
\end{align*}
With the convention $\sup \emptyset : = 0$, we define
\begin{align*}
	\sigma^n_t : = \sup\{i \ge 1: t^n_i \le t\}, \quad t\in [0, T].
\end{align*}
For each $n$, we obtain a process $Y^n$, which approximates $Y$ along the partition $\Pi_n$, given by
\begin{align*}
	Y_t^{n} & := Y_0 + \sum_{i=1}^{\sigma^n_t} \B(b(Y_{n, {i-1}}) \Delta t^n_i  + a(Y_{n, {i-1}}) \Delta_{n, i} W +  \gamma(Y_{n, {i-1}}) \Delta_{n, i} J\B), \quad t \in [0, T].
\end{align*}

\subsubsection{Exploration procedure}
\label{sec:exploration-procedure} Our main idea is as follows:  For $i =1, \ldots, n$, we draw the control at time $t^n_{i-1}$ from some distribution, which is chosen with the accumulative information available at time $t^n_{i-1}$. Once the distribution is fixed, the realization is independent of the rest. In addition, since any distribution on $\bR^D$ can be represented as $F(\xi)$ for some measurable $F \colon \bR^D \to \bR^D$ and $\xi \sim \cN(0, I_D)$, determining a distribution boils down to find such an $F$.

Let us specify this idea.
\begin{enumerate}[(i)]
	\item Let $\Xi: = \{\xi_{n, i}\}_{n \ge 1, 1\le i \le n}$ be a collection of i.i.d. random vectors in $\bR^D$ with common distribution $\cN(0, I_D)$ and probability density $\varphi_D$. Assume that $\Xi$ is independent of $(W, J)$. Family $\Xi$ represents a new source of randomness caused from the exploration along with the randomness generated by $(W, J)$. To capture the information flow, we define the filtration $\bF^{\Pi_n} = (\cF_{n,i})_{i = 0}^n$ as follows	
	\begin{align*}
		\cF_{n, i} : = \sigma\{(W_s, J_s) : 0 \le s \le t^n_i\} \vee \cG_{n,i}, \quad \mbox{where} \quad \cG_{n, i}: = \sigma\b\{\xi_{n, j} : j \le i\b\}, \cG_{n, 0} : = \{\emptyset, \Omega\}.
	\end{align*}
	
	\item $H \colon \Pi_n \times \Omega \times \bR^D \to \bR^D$ is admissible in the following sense (here $H_{n, i-1}$ stands for $H_{t^n_{i-1}}$):
	\begin{enumerate}[(a)]
		\itemsep2pt
		
		\item For each $i = 1, \ldots, n$, the map $(\omega, u) \mapsto  H_{n, i-1}(\omega; u)$ is $\cF_{n, i-1} \otimes \cB(\bR^D)$-measurable;
		
		\item One has $\bE\b[\int_{\bR^D} \|H_{n, i-1}(u)\|^2 \varphi_D(u) \od u\b] <\infty$;
		
		\item As proposed in \cite{WZ20}, the exploration cost  can be represented in terms of differential entropy which is assumed to be finite to encourage the exploration. Following this idea, we in addition assume that for each $i = 1, \ldots, n$ and $\omega \in \Omega$,  $H_{n, i-1}(\omega; \zeta)$ has a probability density $p^H_{n, i-1}(\omega; \cdot)$, where $\zeta \sim \cN(0, I_D)$ is independent of $\cF_{n, i-1}$, such that $\int_{\bR^D} p^H_{n, i-1}(u) \log p^H_{n, i-1}(u) \od u$ is an integrable random variable. Then the expected accumulative differential entropy
		\begin{align*}
		\bE\bb[ -\sum_{i = 1}^n  (t^n_i - t^n_{i-1}) \int_{\bR^D} p^H_{n, i-1}(u) \log p^H_{n, i-1}(u) \od u \bb]
		\end{align*}
		is finite.
	\end{enumerate}
	\item The controlled wealth $X^H = (X^H_t)_{t \in [0, T]}$ associated with $H$ along time points of $\Pi_n$ is
	\begin{align*}
		X^{H}_{n,{i}} = X^{H}_{n, {i-1}} + H_{n, {i-1}}(\xi_{n,i})^{\tran} \Delta_{n, i} Y^n, \quad i = 1, \ldots, n, \quad X^H_0 = x_0 \in \bR.
	\end{align*}
\end{enumerate}

\begin{prop}\label{prop:decomposition-H}
	For $n \ge 1$, $1 \le i \le n$, there exist (uniquely up to a $\bP$-null set) a  random vector $\mu^H_{n, i-1}$ and a random matrix $\vartheta^H_{n, i-1} \in \bS^D_{++}$, both are $\cF_{n, i-1}$-measurable and square integrable, and a square integrable random vector $\eta^H_{n, i}$ with 
	\begin{align}\label{property-eta}
		\bE[\eta^H_{n, i}|\cF_{n, i-1}] = 0 \quad \mbox{and} \quad \bE[\eta^H_{n, i} (\eta^H_{n, i})^\tran|\cF_{n, i-1}] = I_D \quad \mbox{a.s.}
	\end{align}
	such that
	\begin{align}\label{eq:decomposition-H}
		H_{n, i-1}(\xi_{n,i}) & = \mu^H_{n,{i-1}} + \vartheta^H_{n, {i-1}} \eta^H_{n,i} \quad \mbox{a.s.}
	\end{align}
\end{prop}

\begin{proof} See \cref{app:proof:prop:decomposition-H}.
\end{proof}

We decompose the process $X^H$ with the initial wealth $X^H_0 = x_0 \in \bR$ as
\begin{align}\label{eq:decomposition-wealth-exploration}
	X^{H}_{t}  = x_0 & +  \sum_{i=1}^{\sigma^n_t} H_{n, i-1}(\xi_{n,i})^\tran b(Y_{n,{i-1}}) \Delta t^n_{i} \notag \\
	&  + \sum_{i=1}^{\sigma^n_t} H_{n, i-1}(\xi_{n,i})^\tran a(Y_{n,{i-1}}) \Delta_{n,i} W  + \sum_{i=1}^{\sigma^n_t} H_{n, i-1}(\xi_{n,i})^\tran \gamma(Y_{n,{i-1}}) \Delta_{n,i} J \notag\\
	=: x_0 & + I_{\eqref{eq:decomposition-wealth-exploration}} + I\!I_{\eqref{eq:decomposition-wealth-exploration}} + I\!I\!I_{\eqref{eq:decomposition-wealth-exploration}}.
\end{align}

\subsubsection{The drift part $I_{\eqref{eq:decomposition-wealth-exploration}}$} \label{sec:dift-part} According to the decomposition \eqref{eq:decomposition-H}, we express, a.s., 
\begin{align*}
	I_{\eqref{eq:decomposition-wealth-exploration}} & = \sum_{i=1}^{\sigma_t^n} (\mu^H_{n, i-1})^\tran b(Y_{n,{i-1}}) \Delta t^n_i +  \sum_{i=1}^{\sigma_t^n} (\vartheta^H_{n,i-1} \eta^H_{n,i})^\tran b(Y_{n,{i-1}}) \Delta t^n_i\\
	& = \sum_{i=1}^{\sigma_t^n} (\mu^H_{n, i-1})^\tran b(Y_{n,{i-1}}) \Delta t^n_i +  \sum_{d=1}^D \sum_{i=1}^{\sigma_t^n} \BB[\sum_{k=1}^D \vartheta^{H, (k, d)}_{n, i-1} b^{(k)}(Y_{n, i-1})\BB] \B[\eta^{H, (d)}_{n, i} \Delta t^n_i\B].
\end{align*}
For the discrete-time integrator in the second term, we have the following law of large numbers
\begin{align*}
	\sum_{i=1}^{\sigma^n_t} \eta^{H, (d)}_{n, i} \Delta t^n_i \xrightarrow{\bfL_2(\bP)} 0 \quad \mbox{as } n \to \infty
\end{align*}
for all $d = 1, \ldots, D$. Indeed, due to the orthogonality and $\bE[|\eta^{H, (d)}_{n,i}|^2] =1$ it holds that
\begin{align*}
	\bE\bb[\bb|\sum_{i=1}^{\sigma^n_t} \eta^{H, (d)}_{n, i} \Delta t^n_i\bb|^2\bb]=  \sum_{i=1}^{\sigma^n_t} |\Delta t^n_i|^2 \le t \max_{1\le i \le n} \Delta t^n_i \to 0.
\end{align*}

\subsubsection{The Brownian part $I\!I_{\eqref{eq:decomposition-wealth-exploration}}$} \label{discrete-BM-part}
By the same arguments as for the drift part, we decompose $I\!I_{\eqref{eq:decomposition-wealth-exploration}}$ as
\begin{align*}
	I\!I_{\eqref{eq:decomposition-wealth-exploration}}
	& = \sum_{i=1}^{\sigma_t^n} (\mu^H_{n, i-1})^{\tran} a(Y_{n,{i-1}}) \Delta_{n,i} W +    \sum_{d, d' =1, \ldots, D} \; \sum_{i=1}^{\sigma_t^n} \BB[\sum_{k=1}^D\vartheta_{n, i-1}^{H, (k, d)} a^{(k, d')}(Y_{n,{i-1}})\BB]  \B[\eta_{n,i}^{H,(d)}  \Delta_{n,i} W^{(d')}\B].
\end{align*}
Define the interpolated process $W^n = (W^n_t)_{t \in [0, T]}$ and the $\bR^{D^2}$-valued process $M^{n} = (M^n_t)_{t \in [0, T]}$ by
\begin{align*}
		W^n_t & : = \sum_{i =1}^{\sigma^n_t} \Delta_{n, i} W,\\
		M_t^{n, (d, d')} & := \sum_{i=1}^{\sigma^n_t} \eta_{n,i}^{H, (d)} \Delta_{n,i} W^{(d')}, \quad d, d' = 1, \ldots, D,\\
		M^n_t &= (M_t^{n, (1, 1)}, \ldots, M_t^{n, (1, D)}, M_t^{n, (2, 1)}, \ldots, M_t^{n, (2, D)}, \ldots, M_t^{n, (D, 1)}, \ldots, M_t^{n, (D, D)})^{\tran},
	\end{align*}
so that $\Delta_{n, i} W^n = \Delta_{n, i}W$ and $M^n_t = \sum_{i=1}^{\sigma^n_t} \eta^H_{n,i} \pr \Delta_{n, i} W$. Then we get
\begin{align*}
	I\!I_{\eqref{eq:decomposition-wealth-exploration}} =  \sum_{i=1}^{\sigma_t^n} (\mu^H_{n, i-1})^\tran a(Y_{n,{i-1}}) \Delta_{n,i} W^n + \sum_{d, d' =1, \ldots, D}\; \sum_{i=1}^{\sigma_t^n} [\vartheta^H_{n, i-1} a(Y_{n, i-1})]^{(d,d')}  \Delta_{n, i} M^{n, (d, d')}.
	\end{align*}
Here $W^n, M^n$ can be respectively regarded as a discrete-time integrator of the first and the second term in the decomposition of $I\!I_{\eqref{eq:decomposition-wealth-exploration}}$.

\subsubsection{The jump part $I\!I\!I_{\eqref{eq:decomposition-wealth-exploration}}$}\label{discrete-jump-part}
\
For the jump part in \eqref{eq:decomposition-wealth-exploration}, a first try is to rewrite
\begin{align*}
	I\!I\!I_{\eqref{eq:decomposition-wealth-exploration}}
	&= \int_{(0, t] \times \RO \times \bR^D} \sum_{i=1}^{n} \bb[H_{n, i-1} (u)^{\tran} \gamma(Y_{n, {i-1}}) e\bb] \1_{(t^n_{i-1}, t^n_i]}(s) \bbm_n (\od s, \od e, \od u),
\end{align*}
for the random measure 
\begin{align}\label{eq:jump_measure_finite_MC}
	\bbm_n(\od t, \od e, \od u)& :  = \sum_{i =1}^n \delta_{(t^n_{i}, \, \Delta_{n, i} J, \,  \xi_{n,i})}(\od t, \od e, \od u),
\end{align}
on $\cB([0, T] \times E \times \bR^D)$. Here,  $\delta$ denotes the Dirac measure. So, we move the Gaussian random variables $\xi_{n,i}$ for the control randomization from the integrand to the random measure $\bbm_n$, that acts as a new integrator. It is, however, intuitively clear that the limit random measure (in a weak sense) should be
\begin{equation}  \label{eq:jump_measure_finite}
\bbm(\od t, \od e, \od u)=\sum_{j} \delta_{(\tau_j, \, \Delta J_{\tau_j}, \,  \xi_j)}(\od t, \od e, \od u), 
\end{equation}
where $(\tau_j)_{j\in \mathbb{N}}$ are the jump times of $J$ and $(\xi_j)_{j\in \mathbb{N}}$ is a sequence of independent standard Gaussians (independent of $J$), i.e.,  in the limit we would like to create independent Gaussian jumps at each jump time of $L$ as additional source of noise. If the original L\'evy process has infinite activity, this random measure does not induce a L\'evy process, because the squared Gaussian jumps
$
\sum_{j;\tau_j\leq t} \xi_j^2
$ 
do not converge. As a way out, we re-scale the additional Gaussian jumps depending on the jump sizes of the original L\'evy process. 

To this end, let us fix a $\psi \in C^2(\bR^D)$ which satisfies
$$\|\scrD \psi\|_\infty + \|\scrD^2\psi\|_\infty <\infty, \quad \psi \ge 0 \quad \mbox{and} \quad  \psi(x) = 0 \Leftrightarrow x = 0.$$
A prototype example in our context is that, for a given constant $c >0$,
$$\psi(x) = \sqrt{\|x\|^2 + c^2} - c.$$
Define the random measure $\bbm^\psi_n$ on $\cB([0, T] \times E \times \bR^D)$ by setting
\begin{align*}
	\bbm^\psi_n(\od t, \od e, \od u)& :  = \sum_{i =1}^n \delta_{(t^n_{i}, \, \Delta_{n, i} J, \, \psi(\Delta_{n, i} J) \xi_{n,i})}(\od t, \od e, \od u).
\end{align*}
 Then the third term $I\!I\!I_{\eqref{eq:decomposition-wealth-exploration}}$ is finally expressed as
\begin{align*}
	I\!I\!I_{\eqref{eq:decomposition-wealth-exploration}}
	&= \int_{(0, t] \times \RO \times \bR^D} \sum_{i=1}^{n} \bb[H_{n, i-1} \bb(\frac{u}{\psi(e)}\bb)^{\tran} \gamma(Y_{n, {i-1}}) e\bb] \1_{(t^n_{i-1}, t^n_i]}(s) \bbm^\psi_n (\od s, \od e, \od u).
\end{align*}
Note that the random measure $\bbm^\psi_n$ is characterized by the induced martingale $L^{n, \psi} = (L^{n, \psi}_t)_{t \in [0, T]}$ with $L^{n, \psi}_0 = 0$ and
\begin{align*}
	L^{n, \psi}_t: = \int_{(0, t] \times \RO \times \bR^D} (e, u)^\tran \bbm^\psi_n(\od s, \od e, \od u) = \sum_{i=1}^{\sigma^n_t}  (\Delta_{n, i} J, \psi(\Delta_{n, i} J) \xi_{n,i})^\tran= \sum_{i=1}^{\sigma^n_t} \Delta_{n, i} L^{n, \psi}.
\end{align*}
As explained above, the ``damping factor'' $\psi(\Delta_{n, i} J)$ in front of $\xi_{n, i}$ in the third coordinate of $\bbm^\psi_n$  is introduced  to ensure that $	L^{n, \psi}$ can converge to a L\'evy process in the infinite activity case.   The smoothness condition for $\psi$ is merely convenient for applying It\^o's formula later on.

\subsubsection{Distributional limit of discrete-time integrators}\label{sec:distributional-limit}

Set $\bfD: = D^2 + 3D$. We collect all discrete-time integrators of the Brownian and the jump parts to obtain the triangular array of $\bfD$-dimensional random vectors $\bmZ^n = (\bmZ^n_t)_{t \in [0, T]}$ with
\begin{align*}
	\bmZ^n: = \vect(W^n, M^n,	L^{n, \psi}).
\end{align*}
Our purpose is to investigate the distributional limit of $(\bmZ^n)_{n\ge 1}$. To this end, we introduce the Borel measure $\nu_{L}^\psi$ defined on $\bR^{2D}$ by setting
\begin{align*}
	\nu_{L}^\psi(\od e, \od u) : = \1_{\{\|e\| >0\}} \varphi_D\bb(\frac{u}{\psi(e)}\bb) \frac{\od u}{\psi(e)^D} \nu(\od e), \quad e, u \in \bR^D.
\end{align*}
Then, by a change of variables, one has
\begin{align*}
	\int_{\bR^{2D}} f(e, u) \nu^\psi_L(\od e, \od u) = \int_{\RO \times \bR^D} f(e, \psi(e)u) \nu(\od e) \varphi_D(u) \od u
\end{align*}
provided that $f\ge 0$ or  $\int_{\bR^{2D}}|f(e,u)| \nu^\psi_L(\od e, \od u) <\infty$. In particular, choosing $f(e,u) = \|e\|^2 + \|u\|^2$ we find that $\nu_{L}^\psi$ is a square integrable L\'evy measure on $\bR^{2D}\backslash\{0\}$ with $\nu_L^\psi(\{0\}\times \bR^D) = 0$ as
\begin{align*}
	&\int_{\bR^{2D}\backslash\{0\}} (\|e\|^2 + \|u\|^2) \nu_{L}^\psi(\od e, \od u) = \int_{\RO \times \bR^D}(\|e\|^2 + \psi(e)^2 \|u\|^2)\nu(\od e) \varphi_D(u) \od u \\
	& = \int_{\RO} \|e\|^2 \nu(\od e) + \int_{\RO}\psi(e)^2 \nu(\od e) \int_{\bR^D}\|u\|^2 \varphi_D(u) \od u  \le (1+ D\|\scrD \psi\|_\infty^2)\int_{\RO} \|e\|^2 \nu(\od e)  < \infty. 
\end{align*}

We need the following condition to obtain the desired weak convergence.

\begin{assum}\label{assumption-eta}
	$\{\wt H_{n, i-1}(\xi_{n, i})^\tran (\Theta^H_{n, i-1})^{-1} \wt H_{n, i-1}(\xi_{n, i})\}_{1\le i \le n, n\ge 1}$ is uniformly integrable. 
\end{assum}

\begin{rema} Let us briefly comment on \cref{assumption-eta}.
	\begin{enumerate}
		\itemsep2pt
		\item By the construction of $\eta^H_{n, i}$ in the proof of \cref{prop:decomposition-H}, one has, a.s.,
		\begin{align*}
			&\wt H_{n, i-1}(\xi_{n, i})^\tran (\Theta^H_{n, i-1})^{-1} \wt H_{n, i-1}(\xi_{n, i}) \\
			& = \Tr[(\Theta^H_{n, i-1})^{-\frac{1}{2}}\wt H_{n, i-1}(\xi_{n, i})((\Theta^H_{n, i-1})^{-\frac{1}{2}}\wt H_{n, i-1}(\xi_{n, i}))^\tran] \\
			& = \|\eta^H_{n, i}\|^2,
		\end{align*}
		i.e., \cref{assumption-eta} is equivalent to the uniform integrability of $\{\|\eta^H_{n, i}\|^2\}_{1\le i \le n, n\ge 1}$.
		
		\item Assume, for all $n$, that $H\colon \Pi_n \times \Omega \times \bR^D \to \bR^D$ has the form
		\begin{align}\label{remark-linear-control-discrete}
		H_{n, i-1}(\omega; u) = \mathsf{m}_{n, i-1}(\omega) + \mathsf{v}_{n, i-1}(\omega) u, \quad i = 1, \dots, n,
		\end{align}
		where $\mathsf{m}_{n, i-1}$ and $\mathsf{v}_{n, i-1}$ are respectively $\bR^D$-valued and $\bS^D_{++}$-valued random variables, both are $\cF_{n, i-1}$-measurable and square integrable with $\log(\det(\mathsf{v}_{n, i-1})) \in \bfL_1(\bP)$.
	 Then $H$ is linear with respect to the exploration variable and is admissible in the sense given in \cref{sec:exploration-procedure}. Moreover, in the notation of \cref{prop:decomposition-H}, one has $\mu^H_{n, i-1} = \mathsf{m}_{n, i-1}$, $\vartheta^H_{n, i-1} = \mathsf{v}_{n, i-1}$, and $\eta^H_{n, i} = \xi_{n, i}$, which obviously implies that \cref{assumption-eta} holds.
	 
	 \item We will see in \cref{thm:verification} below that the time discretization of the optimal control process for the associated continuous-time control problem has the form  \eqref{remark-linear-control-discrete}. 
	\end{enumerate}
\end{rema}

Under the setting of \Cref{sec:exploration-procedure}, we have the following result whose proof is postponed to \Cref{sec:proof-weak-convergence}.

\begin{theo}\label{thm:limit-distribution-integrator} Assume that $\cW$ is a $D^2$-dimensional standard Brownian motion independent of $W$, and that $L^\psi$ is a square integrable martingale null at $0$ which is a $2D$-dimensional purely non-Gaussian L\'evy process with L\'evy measure $\nu^\psi_L$. Assume that processes $W, \cW, L^\psi$ are defined on the same probability space. Then,  $L^\psi$ is independent of $(W, \cW)$, and under \cref{assumption-eta}, the sequence $(\bmZ^n)_{n \ge 1}$ converges weakly to $\vect(W, \cW, L^\psi)$ as $n \to \infty$ in the Skorokhod topology on the space of c\`adl\`ag functions $f \colon [0, T] \to \bR^{D^2 + 3D}$.
\end{theo}

By rearranging components of $\cW$, we may consider $\cW$ as an $\bR^{D\times D}$-valued process. Then \cref{thm:limit-distribution-integrator} suggests that the exploratory SDE in the continuous-time setting for the controlled wealth process $X^H$ with an admissible $H$ is as follows
\begin{align}\label{eq:exploratory-SDE-jump}
	\od X^H_t  = (\mu^H_{t})^\tran b(Y_{t-}) \od t & + (\mu^H_{t})^\tran a(Y_{t-}) \od W_t + \Tr[(\Theta^H_t)^{\frac{1}{2}} a(Y_{t-}) \od \cW_t^\tran] \notag \\
	& + \int_{\RO \times \bR^D} H_{t} \bb(\frac{u}{\psi(e)}\bb)^\tran \gamma(Y_{t-}) e\, \wt N_L^\psi(\od t, \od e, \od u), \quad X^H_0 = x_0 \in \bR,
\end{align}
where $\wt N^\psi_L$ is the compensated Poisson random measure of $L^\psi$ and the underlying process $Y$ is given by
\begin{align*}
	\od Y_t = b(Y_{t-})\od t  + a(Y_{t-}) \od W_t +  \gamma(Y_{t-}) \int_{\RO \times \bR^D} e \wt N^\psi_L(\od t, \od e, \od u), \quad Y_0 = y_0 \in \bR^D.
\end{align*}
One notices that such a $Y$ also admits $\cL_Y$ in \eqref{eq:generator-Y} as the generator.

\begin{rema}\label{rema:on-exploratory-SDE} Let us briefly comment on SDE \eqref{eq:exploratory-SDE-jump}. For the Brownian component, the noise caused by exploration, i.e. $\cW$, is completely separated from the original noise, i.e. $W$. While for the jump part, both noises are simultaneously captured by the Poisson random measure generated by a $D^2$-dimensional L\'evy process. Interestingly, for the optimal control $H$ obtained in \eqref{eq:optimal-stochastic-control}, it turns out that one can completely separate these two sources of randomness due to the linearity with respect to the exploration variable.
\end{rema}

\begin{rema}
	If the control enters in the drift part only, then the reasoning in \Cref{sec:dift-part} shows that there is no extra exploration noise in the continuous-time formulation. This is the case, e.g., in \cite{GHZ23} where the authors add jumps as uncontrolled L\'evy noise. 
\end{rema}

\begin{rema}
	We briefly explain the relation of the jump part in \eqref{eq:exploratory-SDE-jump} to the notion of a \emph{relaxed Poisson measure}, which has been introduced in the context of relaxed controls by  \cite{Ku00}. We will restrict our discussion to the finite activity case (as assumed in \cite{Ku00}), in which we do not need to re-scale the Gaussian jumps and we can formally set $\psi\equiv 1$. Then, 
	$$
	\widetilde N^\psi_L(\od t, \od e, \od u)=\bbm(\od t, \od e, \od u)-\nu(\od e) \varphi_D(u)\od u \od t
	$$
	for the random measure $\bbm$ defined in \eqref{eq:jump_measure_finite}. Hence, $\bbm$ is a relaxed Poisson measure for the relaxed control  $\varphi_D(u)\od u \od t$ in the sense of \cite[p. 190]{Ku00}. The approximation $\bbm_n$, which we consider in \eqref{eq:jump_measure_finite_MC}, arises from the modeling of control randomization in RL. Numerically it can be interpreted as a Monte Carlo approximation of the relaxed control $\varphi_D(u)\od u \od t$ on the time grid $t_0^n,\ldots, t_n^n$. This Monte Carlo  approximation is conceptually different to the numerical approximation considered by Kushner in \cite{Ku00}. Roughly speaking, Kushner's approach would approximate the relaxed control by replacing the multivariate Gaussian distribution by a discrete distribution supported on $N$ points $\alpha_1,\ldots, \alpha_N$ and successively calling these $N$ points over a refined time grid. 
\end{rema}

\section{Entropy-regularized exploratory MV problem with L\'evy jumps}\label{sec:entropy-exploratory-MV}


We work on a fixed complete probability space $(\Omega, \cF, \bP)$ carrying the triplet $(W, \cW, L^\psi)$ aforementioned in \Cref{thm:limit-distribution-integrator}. Let $N^\psi_L$ denote the associated Poisson random measure of $L^\psi$ with the compensation $\wt N^\psi_L : = N^\psi_L - \Leb_1 \otimes \nu_L^\psi$, where $\Leb_1$ is the $1$-dimensional Lebesgue measure. For $0 \le t \le s \le T$, we denote $\cF_s^t = \sigma\{W_r - W_t, \cW_r-\cW_t, L^\psi_r - L^\psi_t: t\le r\le s\}$ augmented by all $\bP$-null sets. Set $\cF_s: = \cF^0_s$.

\smallskip

For $\xi \sim \cN(0, I_D)$ we define the family of \textit{deterministic} admissible functions as
\begin{align*}
	\cA : = \bb\{ F \;\bb|\; F \colon \bR^D \to \bR^D  \; \mbox{ Borel, } \int_{\bR^D} \|F(u)\|^2 \varphi_D(u)\od u < \infty, \; F(\xi) \mbox{ has a probability density } p^F\bb\}.
\end{align*}

Admissible controls in the discrete-time setting are adapted to the continuous-time setting as follows.

\begin{defi}[Admissible control]\label{defi:admissible-control}
	For $(t, y) \in [0, T) \times \bR^D$, denote by $\cA(t, y)$ the family of all admissible controls $H$ for which the following conditions hold:
	\begin{enumerate}[\quad (H1)]
		\itemsep2pt
		\item  \label{adm-control-predictability} (Admissibility) $H\colon [t, T] \times \Omega \times \bR^D \to \bR^D$ satisfies that
		\begin{enumerate}[(a)]
			\itemsep2pt
			\item $H$ is $\cP([t, T]) \otimes \cB(\bR^D)$-measurable, where $\cP([t, T])$ is the predictable $\sigma$-algebra on $[t, T] \times \Omega$;
			
			\item $H_s(\cdot): = H_s(\omega; \cdot) \in \cA$ for all $(s, \omega) \in [t, T] \times \Omega$.
		\end{enumerate} 
		
		\item \label{adm-control-integrability}  (Integrability)  It holds that 
		\begin{align}\label{adm-control-omegawise-integrability}
			\bP\bb(\int_t^T \nq \int_{\bR^D} \|H_s(u)\|^2 \varphi_D(u) \od u < \infty\bb) =1,
		\end{align}
		and that processes $\mu^H = (\mu^H_s)_{s \in [t, T]}$, $\Theta^H = (\Theta^H_s)_{s \in [t, T]}$ defined on $[t, T] \times 
		\Omega$ by
		\begin{align*}
			\mu^H_s  : = \int_{\bR^D} H_s(u) \varphi_D(u) \od u, \quad \wt H_s(u): = H_s(u) - \mu^H_s, \quad	\Theta^H_s  : = \int_{\bR^D} \wt H_s(u)  \wt H_s(u)^\tran \varphi_D(u) \od u,
		\end{align*}
		satisfy that
		\begin{align}\label{admissible:integrability}
			 \bE\bb[ \int_t^T \bb((\mu^H_s)^{\tran} A(Y^{t,y}_{s-}) \mu^H_s + \Tr[A(Y^{t,y}_{s-}) \Theta^H_s]  & + \int_{\RO \times \bR^D} |H_s(u)^\tran \gamma(Y^{t, y}_{s-}) e|^2 \nu(\od e) \varphi_D(u)\od u\bb) \od s\bb]  \notag \\
			&  + \bE\bb[\bb|\int_t^T |(\mu^H_s)^\tran b(Y^{t, y}_{s-})| \od s\bb|^2 \bb] < \infty,
		\end{align}
		where $Y^{t,y} = (Y^{t, y}_s)_{s \in [t, T]}$ is a (unique) strong solution to the following SDE on $[t, T]$
		\begin{align}\label{eq:SDE-Y-at-t}
			\od Y^{t,y}_s &  = b(Y^{t,y}_{s-})\od s  + a(Y^{t,y}_{s-}) \od W_s +  \gamma(Y^{t,y}_{s-}) \int_{\RO \times \bR^D} e \wt N^\psi_L(\od s, \od e, \od u),\quad  Y^{t,y}_t=y.
		\end{align}
		
		\item\label{admissible:entropy} (Finite accumulative differential entropy) There is a kernel $p^H\colon [t, T] \times \Omega \times \bR^D \to \bR$ such that $p^H_s(\omega; \cdot)$ is a probability density function of $H_s(\omega; \zeta)$ for any $(s, \omega) \in [t, T] \times \Omega$, where $\zeta \sim \cN(0, I_D)$ is independent of $\cF^t_s$, and that $(s, \omega) \mapsto \int_{\bR^D} p^H_s(u) \log p^H_s(u) \od u$ is $(\cF^t_s)_{s \in [t, T]}$-predictable with
		\begin{align}\label{eq:admissible:entropy}
			\bE\bb[ \int_t^T \bb| \int_{\bR^D} p^H_s(u) \log p^H_s(u)\od u \bb| \od s\bb] <\infty.
		\end{align}
	\end{enumerate}
\end{defi}

For a given control $H \in \cA(t, y)$ and $x \in \bR$, the dynamic of the controlled wealth process $X^{t, x, y; H} = (X^{t, x, y; H}_s)_{s \in [t, T]}$, which is assumed to has c\`adl\`ag paths, is described by the exploratory SDE on $[t, T]$ as
\begin{align}\label{eq:couple-wealth-price}
		\od X^{t, x, y; H}_s  & =  (\mu^H_s)^\tran b(Y^{t, y}_{s-}) \od s + (\mu^H_s)^\tran a(Y^{t, y}_{s-}) \od W_s + \Tr[(\Theta^H_s)^{\frac{1}{2}} a(Y^{t, y}_{s-}) \od \cW_s^\tran] \notag \\
		& \quad + \int_{\RO \times \bR^D} H_s \bb(\dfrac{u}{\psi(e)}\bb)^\tran \gamma(Y^{t, y}_{s-}) e\, \wt N_L^\psi(\od s, \od e, \od u), \quad
		X^{t, x, y; H}_t  = x,
\end{align}
where $Y^{t,y}$ solves the SDE \eqref{eq:SDE-Y-at-t}.

\begin{rema}\label{rema:admissible-control}
	\begin{enumerate}
		\itemsep2pt

		\item Processes $\mu^H$ and $\Theta^H$ are predictable by  (H\ref{adm-control-predictability}) and Fubini's theorem.
		
		\item As a consequence of \cite[Theorem X.1.1]{Bh97}, there exists a $c_D>0$ such that	$\|A^{\frac{1}{2}} - B^{\frac{1}{2}}\| \le c_D \|A - B\|^{\frac{1}{2}}$ for any $A, B \in \bS^D_+$. Hence $\bS^D_{+} \ni A \mapsto A^{\frac{1}{2}}$ is (H\"older) continuous which then  ensures that $(\Theta^H)^{\frac{1}{2}}$ is also a predictable $\bS^D_+$-valued process.
		
		\item Due to \eqref{admissible:integrability}, $X^{t, x, y; H}$ in \eqref{eq:couple-wealth-price} is a square integrable process satisfying 
		\begin{align}\label{eq:integrability-admissible-H}
			\bE\bb[\sup_{t \le s \le T} |X^{t, x, y; H}_s|^2\bb] <\infty.
		\end{align}
	\end{enumerate}
\end{rema}

\subsection{Problem formulation}\label{subsec-exploratory-MV-new}

We are now in a position to formulate the entropy-regularized exploratory MV problem. Remark that, due to the time inconsistency of the MV problem, we just examine solutions among \textit{precommitted} strategies which are optimal at $t = 0$ only.

Let us fix a $\hat z \in \bR$ which represents the targeted expected terminal wealth. For an initial wealth $x_0 \in \bR$ and $y_0 \in \bR^D$, we consider the problem
 \begin{equation}\label{exploratory-MV-constraint}
	\left\{ \begin{aligned} 
		&\min_{H \in \cA(0, y_0)} \bE\bb[\B(X^{0, x_0, y_0;H}_T - \bE \B[X^{0, x_0, y_0;H}_T\B] \B)^2 + \lambda \int_0^T \nq \int_{\bR^D} p^H_s(u) \log p^H_s(u) \od u \od s\bb] \\
		& \mbox{subject to } X^{0, x_0, y_0; H} \mbox{ given in } \eqref{eq:couple-wealth-price} \mbox{ and }\bE\B[X^{0, x_0, y_0;H}_T\B] = \hat{z}.
	\end{aligned} \right.
\end{equation}
Here the \textit{exploration weight} $\lambda\ge 0$, which is fixed from now on, describes the trade-off between exploitation and exploration and it is also known as the \textit{temperature parameter} in the RL literature.

We follow \cite{WZ20} to apply the Lagrange multiplier method to solve  \eqref{exploratory-MV-constraint} (see \cref{sec:classical-MV} for a similar argument in the setting without exploration). In the first step, we examine the following entropy-regularized quadratic-loss minimization problem, parameterized by $\hat{w} \in \bR$,
\begin{equation}\label{exploratory-MV-unconstraint}
	\left\{ \begin{aligned} 
		&\min_{H \in \cA(0, y_0)} \bE\bb[\B(X^{0, x_0, y_0;H}_T - \hat{w} \B)^2 + \lambda \int_0^T \nq \int_{\bR^D} p^H_s(u) \log p^H_s(u) \od u \od s\bb] \\
		& \mbox{subject to } X^{0, x_0, y_0; H} \mbox{ given in } \eqref{eq:couple-wealth-price}.
	\end{aligned} \right.
\end{equation}
 We solve \eqref{exploratory-MV-unconstraint} to obtain a solution $H^*: = H^*(\hat{w})$ depending on $\hat{w}$. This task is presented in \cref{entropy-regularized quadratic-loss}. In the next step, we find the Lagrange multiplier $\hat{w}$  by using the constraint $\bE[X_T^{H^*}] = \hat{z}$. Then $H^*(\hat{w})$ is a solution to problem \eqref{exploratory-MV-constraint} where $\hat{w}$ is the obtained Lagrange multiplier. The latter task is done in \cref{sec:explicit-solutions}.

\subsection{The entropy-regularized quadratic-loss optimization problem}\label{entropy-regularized quadratic-loss}
Let us fix $\hat{w} \in \bR$.  Problem \eqref{exploratory-MV-unconstraint} is an unconstrained control problem and we will find its solutions via the dynamic programing approach. Define the function $V^H(\cdot|\hat{w})$ associated with a control $H \in \cA(t, y)$ and $x\in \bR$ by setting 
\begin{align*}
	V^H(t, x, y|\hat{w}): = \bE\bb[ \B(X^{t, x, y;H}_T - \hat{w}\B)^2 + \lambda \int_t^T \nq \int_{\bR^D} p^H_s(u) \log p^H_s(u) \od u \od s \bb].
\end{align*}
We consider the following system of problems which particularly yields to \eqref{exploratory-MV-unconstraint} when $(t, x, y) = (0, x_0, y_0)$.
\begin{prob}\label{prob:MV-entropy}
	For given $(t, x, y) \in [0, T) \times \bR \times \bR^D$, find an $H^* \in \cA(t, y)$ such that 
	\begin{align}\label{eq:MV-problem-formulation}
	V^*(t, x, y|\hat{w}):=	V^{H^*}(t, x, y|\hat{w}) = \min_{H \in \cA(t, y)} V^{H}(t, x, y|\hat{w})
	\end{align}
	subject to the state equation \eqref{eq:couple-wealth-price}.
\end{prob}

\begin{defi} For a given initial triple $(t, x, y)$, any $H^* \in \cA(t, y)$ satisfying \eqref{eq:MV-problem-formulation} is called an \textit{optimal control}, the corresponding controlled state process $X^{t, x, y;*}: = X^{t, x, y;H^*}$ is called an \textit{optimal state/wealth process}, and $V^*(\cdot|\hat{w})$ satisfying the terminal condition $V^*(T, x, y|\hat{w}) = (x-\hat{w})^2$ is called the \textit{value function}.
\end{defi}

\subsubsection{Entropy-regularized Hamilton--Jacobi--Bellman (HJB) equation}
As we use the dynamic programming approach to solve \cref{prob:MV-entropy}, it is useful to consider the associated HJB equation. Let us first introduce some notations. For $F \in \cA$, we define $m^F \in \bR^D$ and $\theta^F \in \bS^D_+$ by
\begin{align*}
	m^F &: = \int_{\bR^D} F(u) \varphi_D(u) \od u,\\
	\theta^F &:= \int_{\bR^D} (F(u)-m^F) (F(u)-m^F)^\tran \varphi_D(u) \od u = \int_{\bR^D} F(u) F(u)^\tran \varphi_D(u) \od u - m^F (m^F)^\tran,
\end{align*}
and the \textit{differential entropy} of $F$ is denoted by
\begin{align*}
	\ent(F)  := -\int_{\bR^D} p^F(u) \log p^F(u) \od u.
\end{align*}
Using the classical Bellman's principle of optimality and a standard verification argument (see the proof of \cref{thm:verification} below) we find that the HJB type formula in our setting is stated in form of a (possibly degenerate) second-order PIDE as follows:
\begin{align}\label{eq:HJB}
	0 &= \pd_t v(t, x, y) + b(y)^\tran \scrD_y v(t, x, y) + \frac{1}{2} \Tr[A(y) \scrD^2_{yy} v(t, x, y) ] \notag \\
		 &\quad + \min_{F \in \cA} \bb\{  \frac{1}{2} \pd^2_{xx} v(t, x, y)\B((m^F)^\tran A(y) m^F + \Tr[A(y)\theta^F]\B) \notag \\
		 & \quad \qquad \qquad + (m^F)^\tran \B( A(y) \scrD^2_{xy} v(t, x, y) +  \pd_x v(t, x, y) b(y) \B) \notag\\
		  &\quad \qquad \qquad + \int_{\RO \times \bR^D} \B(v(t, x + F(u)^\tran \gamma(y) e, y + \gamma(y) e) - v(t, x, y) \notag \\
		 &\hspace{120pt}  - \pd_x v(t, x, y) F(u)^\tran \gamma(y) e  - \scrD_y v(t, x, y)^\tran \gamma(y)e \B) \nu(\od e) \varphi_D(u) \od u \notag \\
		 & \quad \qquad  \qquad  -  \lambda \ent(F)\bb\}, \quad (t, x, y) \in [0, T) \times \bR \times 
		  \bR^D,
\end{align}
with the terminal condition $v(T, x, y) = (x-\hat{w})^2$ for $(x, y) \in \bR \times \bR^D$.

\begin{rema}\label{rema:min-entropy} By \cite[Theorem 8.6.5]{CT06}, one has $\ent(F) = - \infty$ if $\det(\theta^F) = 0$. Hence, it suffices to consider the above minimization over $F \in \cA$ with $\det(\theta^F)>0$, i.e. over $F \in \cA$ with $\theta^F \in \bS^D_{++}$.
\end{rema}

\begin{rema}
	 In the case of no jumps, i.e. $\nu = 0$, \eqref{eq:HJB}  simplifies to
	 \begin{align*}
	 	0 &= \pd_t v(t, x, y) + b(y)^\tran \scrD_y v(t, x, y) + \frac{1}{2} \Tr[A(y) \scrD^2_{yy} v(t, x, y) ] \notag \\
	 	&\quad + \min_{m \in \bR^D, \theta \in \bS^D_{++}} \bb\{  \frac{1}{2} \pd^2_{xx} v(t, x, y)\B(m^\tran A(y) m + \Tr[A(y)\theta]\B) \notag \\
	 	& \hspace{100pt} + m^\tran \B( A(y) \scrD^2_{xy} v(t, x, y) +  \pd_x v(t, x, y) b(y) \B)  \\
	 	& \hspace{100pt}  -\lambda\quad  \max_{ F\in \mathcal{A}; \;m^F=m \textnormal{ and } \theta^F=\theta} \quad \ent(F)\bb\}, \quad (t, x, y) \in [0, T) \times \bR \times 
	 	\bR^D.
	 \end{align*}
	 By the entropy-maximizing property of the Gaussian distribution, the maximum over $F$ is achieved at the linear function $F(u)=m+\theta^{\frac{1}{2}} u$ and the HJB-equation becomes a  
	 second-order PDE for the unknown value function $v$. This PDE can the be solved by a quadratic ansatz as done e.g. in \cite{WZ20} for the case $D = 1$  and constant coefficients. In the presence of jumps, this ``separation argument'' (solving first for $F$ given its first two moments) a priori does not work anymore, because $F$ explicitly enters the integral term of the HJB-PIDE \eqref{eq:HJB}. As we will show in the next section,  \eqref{eq:HJB} can nonetheless be solved by a quadratic ansatz, and, then, the separation step can be performed a posteriori, leading to the Gaussianity of the optimal control law. 
\end{rema}

\subsubsection{Quadratic ansatz}\label{sec:ansatz}  

We first introduce the following function classes in relation to the coefficient $\gamma$ and L\'evy measure $\nu$.

\begin{defi}\label{defi:Upsilon-class} For a Borel function $g \colon [0, T] \times \bR^D \to \bR$ we let
	$g \in \Upsilon(0)$ (resp. $g \in \Upsilon(1)$, $g \in \Upsilon(2)$) if there exists a (jointly) continuous function $\Upsilon^{(0)}_g (\mbox{resp. } \Upsilon^{(1)}_g, \Upsilon^{(2)}_g) \colon [0, T] \times \bR^D \to [0, \infty)$ such that 
	\begin{align*}
		&\int_{\RO} |g(t, y + \gamma(y) e)|\, \|e\|^2 \nu(\od e) \le \Upsilon^{(0)}_g (t, y),\\
		\mbox{resp.} \quad &\int_{\RO} |g(t, y + \gamma(y) e) - g(t, y)| \,\|e\|\nu(\od e) \le \Upsilon^{(1)}_g(t, y),\\
		\mbox{resp.} \quad &\int_{\RO} \B|g(t, y + \gamma(y) e) - g(t, y) - \scrD_y g(t, y)^\tran \gamma(y) e\B|  \nu(\od e) \le \Upsilon^{(2)}_g(t, y),
	\end{align*}
for all $(t,y) \in [0, T]\times \bR^D$, where we additionally assume that $\scrD_y g$ exists and measurable for $g \in \Upsilon(2)$. Then $\Upsilon^{(k)}_g$ is called an $\Upsilon$-\textit{dominating function} of $g \in \Upsilon(k)$.
\end{defi}

\begin{rema}
A standard calculation shows that $g \in \Upsilon(0) \cap \Upsilon(1) \cap \Upsilon(2)$ if $\int_{\RO} \|e\|^2 \nu(\od e) <\infty$ and one of the following holds: 
		\begin{enumerate}[\quad (a)]
			\item $g$ is twice continuously differentiable with respect to $y$ with $$\sup_{(t, y) \in [0, T] \times \bR^D}(|g(t, y)| + \|\scrD_y g(t, y)\| + \|\scrD^2_{yy} g(t, y)\|) < \infty.$$
			
			\item $\sup_{(t, y) \in [0, T] \times \bR^D} |g(t, y)| <\infty$, $\scrD_y g$ is jointly continuous on $[0, T] \times \bR^D$, and $\nu(E)<\infty$.
		\end{enumerate}		
\end{rema}

For $\alpha \in \Upsilon(0) \cap \Upsilon(1)$, $\alpha >0$ on $[0, T]\times \bR^D$, and $\scrD_y \alpha$ exists, we define the functions $\scrM_\alpha\colon [0, T] \times \bR^D \to \bR^D$ and $\scrS_\alpha \colon [0, T] \times \bR^D \to \bS^D_{++}$ as
\begin{align}
	\scrM_\alpha(t, y) &: = \alpha(t, y) b(y) +  A(y) \scrD_y \alpha(t, y) + \gamma(y)  \int_{\RO} (\alpha(t, y + \gamma(y)e) - \alpha(t, y)) e \,\nu(\od e),\label{eq:defi:optimal-mean}\\
	\scrS_\alpha(t, y) & : = \alpha(t, y) A(y) + \gamma(y) \bb(\int_{\RO} \alpha(t, y + \gamma(y) e) e e^\tran  \nu(\od e)\bb) \gamma(y)^\tran\label{eq:defi:optimal-variance}.
\end{align}
In particular, if $\alpha \equiv 1$ on $[0, T] \times \bR^D$ then $\scrM_\alpha = b$ and $\scrS_\alpha = \Sigma$. One also remarks that the mapping $\scrS_\alpha$ is well-defined. Indeed, for any $(t, y) \in [0, T] \times \bR^D$ and $u \in \bR^D \backslash\{0\}$, one has $u^\tran \scrS_\alpha(t,y) u>0$ because of
 $\alpha>0$ and the non-degenerate condition (see \Cref{sec:setting}). As a consequence, the inverse $\scrS_\alpha^{-1}(t,y)$ exists and also belongs to $\bS^D_{++}$ which can be easily derived from the spectral decomposition of $\scrS_\alpha(t,y)$.

\begin{prop}[Quadratic value function]\label{prop:optimal-value-function}
	Let $\alpha, \beta \in C^{1,2}([0, T] \times \bR^D) \cap \Upsilon(2)$. Assume that $\alpha \in \Upsilon(0) \cap \Upsilon(1)$ and $\alpha>0$, and that $\alpha, \beta$ solve  the following system of PIDEs pointwise on $[0, T) \times \bR^D$,
	\begin{equation}\label{system:PDEs}
		\left\{ \begin{aligned} 
			&\pd_t \alpha(t, y) + \cL_Y \alpha(t,y)  - (\scrM_\alpha^\tran \scrS_\alpha^{-1} \scrM_\alpha)(t, y) =0, \\
			& \pd_t \beta(t, y) + \cL_Y \beta(t,y) - \frac{\lambda}{2} \log\bb( \frac{(\lambda \pi)^D}{\det(\scrS_\alpha(t, y))}\bb) = 0,\\
			&\alpha(T, \cdot) \equiv 1 \quad \mbox{and}\quad \beta(T, \cdot) \equiv 0,
		\end{aligned} \right.
	\end{equation}
where $\cL_Y \phi(t,y): = (\cL_Y  \phi(t, \cdot))(y)$ for $\phi \in \{\alpha, \beta\}$. Then, for $(t, x, y) \in [0, T] \times \bR \times \bR^D$, 
\begin{align}\label{eq:ansatz-value-function}
	v^{\opt}(t, x, y) := \alpha(t, y) (x- \hat{w})^2 + \beta(t, y)
\end{align}
solves the HJB equation \eqref{eq:HJB}. Moreover, a minimizer $F^\opt = F_\alpha^\opt(t, x, y; \lambda; \cdot) \in \cA$ is
	\begin{align}\label{eq:optimal-control-ansatz-F}
		F_\alpha^{\opt}(t, x, y; \lambda; u) = m^{\opt}_{\alpha}(t,x,y) + \theta^{\opt}_\alpha(t,y;\lambda)^{\frac{1}{2}} u,
	\end{align}
	where
	\begin{align}\label{minimizer-mean-variance}
		m^\opt_\alpha(t, x, y) := - (x- \hat{w}) (\scrS_\alpha^{-1} \scrM_\alpha)(t, y) \quad \mbox{and} \quad
		\theta^\opt_\alpha(t, y;\lambda) : = \frac{\lambda}{2} \scrS_\alpha^{-1}(t, y).
	\end{align}
\end{prop}

\begin{proof} One first notices that $\cL_Y \alpha(t, \cdot)$ and $\cL_Y \beta(t, \cdot)$ are well-defined functions for $t \in [0, T]$. 
To simplify the presentation, we omit the argument $y$ of coefficient functions $b, A, \gamma$, and for fixed $(t, y)$, we formally use the following notations for $\alpha$ (and analogously for $\beta$),
\begin{align*}
	\alpha : = \alpha(t, y),\;\; \tilde \alpha(e): = \alpha(t, y + \gamma(y) e),\;\; \cL_Y \alpha: = \cL_Y \alpha(t, y).
\end{align*}
Plugging the ansatz \eqref{eq:ansatz-value-function} into the HJB equation \eqref{eq:HJB} and rearranging terms we get the following which holds pointwise on $[0, T) \times \bR \times \bR^D$,
\begin{align}\label{eq:HJB-ansatz}
0&= (x-\hat{w})^2 (\pd_t \alpha + \cL_Y \alpha)  + (\pd_t \beta + \cL_Y\beta) \notag\\
&\quad +  \min_{F \in \cA,\, \theta^F \in \bS^D_{++}}\bb\{ \alpha \B((m^F)^\tran A m^F + \Tr[A\theta^F]\B) + 2 (x-w)  (m^F)^\tran (A \scrD_y \alpha  + \alpha b)\notag \\
& \quad  + \int_{\RO} \B( \tilde\alpha(e) e^\tran \gamma^\tran (\theta^F + m^F (m^F)^\tran) \gamma e + 2 (x-w) (\tilde\alpha(e) - \alpha) (m^F)^\tran \gamma e\B) \nu(\od e) - \lambda \ent(F)\bb\},
\end{align}
where the minimization is taken over $F\in \cA$ with $\theta^F \in \bS^D_{++}$ due to \cref{rema:min-entropy}. Remark that given any $m \in \bR^D$, $\theta \in \bS^D_{++}$, there always exists an $F \in \cA$ such that $m^F = m$ and $\theta^F = \theta$, for example, one might take $F(u) = m + \theta^{\frac{1}{2}} u$. Then the minimum over $F \in \cA$ with $\theta^F \in \bS^D_{++}$ in \eqref{eq:HJB-ansatz} can be separated into two individual minimization problems, one is over $m^F \in \bR^D$ and the other is over $\theta^F \in \bS^D_{++}$. Specifically, let $\Psi_{\eqref{eq:HJB-ansatz}}^F$ denote the expression inside the minimum in \eqref{eq:HJB-ansatz}, then one has
\begin{align}\label{eq:min-F-separate}
	\min_{F \in \cA,\, \theta^F \in \bS^D_{++}} \Psi_{\eqref{eq:HJB-ansatz}}^F & = \min_{m \in \bR^D} \bb\{\alpha m^\tran A m  + 2(x-w)   m^\tran  (A \scrD_y \alpha + \alpha b) \notag\\
	& \qquad \qquad + \int_{\RO} \B(\tilde\alpha(e)  (m^\tran \gamma e)^2 + 2(x-w)(\tilde\alpha(e) - \alpha)m^\tran \gamma e \B) \nu(\od e)\bb\} \notag \\
	& \quad + \min_{\theta \in \bS^D_{++}} \bb\{ \alpha \Tr[A \theta] + \int_{\RO} \tilde\alpha(e) e^\tran \gamma^\tran \theta \gamma e \,\nu(\od e) - \lambda \max_{F \in \cA,\, \theta^F = \theta}\ent(F) \bb\} \notag\\
	& =: \min_{m \in \bR^D} f_{\eqref{eq:min-F-separate}}(m) + \min_{\theta \in \bS^D_{++}} g_{\eqref{eq:min-F-separate}}(\theta).
\end{align}
 It is known that the differential entropy  is translation invariant and it is maximized over all distributions with a given covariance matrix by Gaussian distribution, see, e.g., \cite[Theorem 8.6.5]{CT06}. Hence,  $g_{\eqref{eq:min-F-separate}}$ can be expressed as
 \begin{align*}
 	g_{\eqref{eq:min-F-separate}}(\theta) = \alpha \Tr[A \theta] +  \int_{\RO} \tilde\alpha(e) e^\tran \gamma^\tran \theta \gamma e \, \nu(\od e)  - \frac{\lambda}{2} \log(\det(\theta)) - \frac{\lambda D}{2}\log(2\pi \e).
 \end{align*}
 Combining \eqref{eq:HJB-ansatz} with \eqref{eq:min-F-separate} yields the equation
\begin{align}\label{eq:HJB-particular-f-g}
	 (x-\hat{w})^2 (\pd_t \alpha + \cL_Y \alpha)  + (\pd_t \beta + \cL_Y\beta) + \min_{m \in \bR^D} f_{\eqref{eq:min-F-separate}}(m) + \min_{\theta \in \bS^D_{++}} g_{\eqref{eq:min-F-separate}}(\theta) =0.
\end{align}
We first consider the minimization problem
\begin{align*}
	\min_{\theta \in \bS^D_{++}} g_{\eqref{eq:min-F-separate}}(\theta).
\end{align*}
By vectorization, $\bS^D_{++}$ can be regarded as an open subset of  $\bR^{D(D+1)/2}$, where the openness (under the Euclidean norm) can be inferred from Sylvester's criterion, so  that $g_{\eqref{eq:min-F-separate}}$ becomes a function defined on $\bS^D_{++} \subset \bR^{D(D+1)/2}$.  Since $\theta \mapsto -\log(\det(\theta))$ is a convex and differentiable function on $\bS^D_{++}$, it implies that $g_{\eqref{eq:min-F-separate}}$ is also convex and differentiable. Hence, solutions of $\scrD g_{\eqref{eq:min-F-separate}}$ globally minimize $g_{\eqref{eq:min-F-separate}}$ on $\bS^D_{++}$. To find its solutions, we represent  $\theta = (\theta^{(1,1)}, \ldots, \theta^{(D,1)}, \theta^{(2,2)}, \ldots  \theta^{(D,2)},\ldots, \theta^{(D, D)})^\tran \in \bR^{D(D+1)/2}$. Then, for $1\le j \le i \le D$, according to \cite[p.311, Eq. (8.12)]{Ha97} one has
\begin{align*}
	\frac{\pd \log \det(\theta)}{\pd \theta^{(i, j)}} = [2\theta^{-1} - \diag(\theta^{-1})]^{(i, j)}
\end{align*}
so that the partial derivatives of $g_{\eqref{eq:min-F-separate}}$ are computed by
\begin{align*}
	&\frac{\pd g_{\eqref{eq:min-F-separate}}}{\pd \theta^{(i, j)}}(\theta) \\
	& = \alpha [2A - \diag(A)]^{(i, j)} + \int_{\RO} \tilde\alpha(e) [2\gamma e e^\tran \gamma^\tran - \diag(\gamma e e^\tran \gamma^\tran)]^{(i, j)}\nu(\od e) -\frac{\lambda}{2} [2 \theta^{-1} - \diag(\theta^{-1}) ]^{(i, j)}.
\end{align*}
Solving $\scrD g_{\eqref{eq:min-F-separate}}(\theta) = 0$ we get the solution $\theta = \theta^\opt_\alpha(t, y;\lambda)$ as provided in \eqref{minimizer-mean-variance}. Hence, $\theta^\opt_\alpha(t, y;\lambda)$ is a global minimizer of $g_{\eqref{eq:min-F-separate}}$ on $\bS^D_{++}$. We next investigation the problem
\begin{align*}
	\min_{m \in \bR^D} f_{\eqref{eq:min-F-separate}}(m).
\end{align*}
Solving $\scrD f_{\eqref{eq:min-F-separate}}(m) = 0$ yields the solution $m = m^\opt_\alpha(t, x, y)$ which is provided in \eqref{minimizer-mean-variance}. Moreover, since
\begin{align*}
	\scrD^2 f_{\eqref{eq:min-F-separate}} = 2 \alpha A  + 2 \gamma \bb(\int_{\RO} \tilde \alpha(e) e e^\tran \nu(\od e)\bb)\gamma^\tran =  2 \scrS_\alpha
\end{align*}
and $\scrS_\alpha \in \bS^D_{++}$ as claimed above, we infer that $m^\opt_\alpha(t, x, y)$ is a global minimizer of $f_{\eqref{eq:min-F-separate}}$ on $\bR^D$. Plugging these minimizers back into \eqref{eq:HJB-particular-f-g} and noticing that 
\begin{align*}
	\Tr[\alpha A \scrS_\alpha^{-1}] = \Tr\bb[I_D - \int_{\RO} \tilde \alpha(e) \gamma e e^\tran \gamma^\tran \scrS_\alpha^{-1} \nu(\od e) \bb] = D - \int_{\RO} \tilde \alpha(e) e^\tran \gamma^\tran \scrS_\alpha^{-1} \gamma e \, \nu(\od e)
\end{align*}
we eventually arrive at the equation
\begin{align*}
	0 & =  (x-\hat{w})^2 \B(\pd_t \alpha + \cL_Y \alpha  - \scrM_\alpha^\tran \scrS_\alpha^{-1} \scrM_\alpha  \B)  + \bb(\pd_t \beta + \cL_Y \beta - \frac{\lambda}{2} \log\bb( \frac{(\lambda \pi)^D}{\det(\scrS_\alpha)}\bb) \bb)
\end{align*}
which holds true according to assumption \eqref{system:PDEs}. As a consequence, the function provided in \eqref{eq:optimal-control-ansatz-F} is an optimal solution of \eqref{eq:HJB-ansatz}.
\end{proof}

\subsubsection{Verification argument}

In the following result, the coefficients $K \in \{b, a, \gamma, A, \Sigma\}$ are conveniently extended to be defined on $[0, T] \times \bR^D$ by setting $K(t, y): = K(y)$. We recall $\scrM_\alpha$ and $\scrS_\alpha$ from \eqref{eq:defi:optimal-mean} and \eqref{eq:defi:optimal-variance} respectively.

\begin{theo}\label{thm:verification}
	Let $\alpha, \beta$ satisfy the assumptions of \cref{prop:optimal-value-function}.  Let $(t, x, y) \in [0, T) \times \bR \times \bR^D$ and recall $Y^{t, y}$ in \eqref{eq:SDE-Y-at-t}. Assume furthermore that $\{\beta(\tau, Y^{t,y}_{\tau}) \,|\, \tau \colon \Omega \to [t, T] \mbox{ is a stopping time}\}$ is uniformly integrable and that $\alpha$ is bounded on $[0, T]\times \bR^D$ and satisfies
	\begin{align}
		& \int_t^T |(\scrM_\alpha^\tran \scrS^{-1}_\alpha b)(s, Y^{t,y}_{s-})|^2 \od s + \sup_{s \in (t, T)} |(\scrM_\alpha^\tran \scrS_\alpha^{-1}  \Sigma \scrS^{-1}_\alpha \scrM_\alpha)(s,Y^{t,y}_{s-})  | \le c_{\eqref{eq:thm:verification-condition-alpha-1}} \quad \mbox{a.s.}, \label{eq:thm:verification-condition-alpha-1} \\\
		&\bE\bb[\int_t^T  \Tr[(\Sigma \scrS_\alpha^{-1})(s, Y^{t,y}_{s-})] \od s\bb]  + \bE\bb[\int_t^T\B|\log\B(\det(\scrS_\alpha(s, Y^{t,y}_{s-}))\B)\B| \od s\bb] <\infty, \label{eq:thm:verification-condition-alpha-2}
	\end{align}
	for some non-random constant $c_{\eqref{eq:thm:verification-condition-alpha-1}}>0$. Then a solution for \cref{prob:MV-entropy}  is 
	\begin{align}\label{eq:optimal-stochastic-control}
		H^{t, x, y;*}_s(u) = -(X^{t, x, y;*}_{s-} - \hat{w})(\scrS_\alpha^{-1} \scrM_\alpha)(s, Y^{t,y}_{s-}) + \sqrt{\frac{\lambda}{2}} \scrS_\alpha^{-\frac{1}{2}}(s, Y^{t,y}_{s-}) u, \quad s\in (t, T], u \in \bR^D,
	\end{align}
	with $H^{t, x, y;*}_t(u) := -(x - \hat{w})(\scrS_\alpha^{-1} \scrM_\alpha)(t, y) + \sqrt{\frac{\lambda}{2}} \scrS_\alpha^{-\frac{1}{2}}(t, y) u$, and the corresponding optimal wealth process $X^{t, x, y;*} = (X^{t, x, y;*}_s)_{s \in [t, T]}$ is a unique c\`adl\`ag (strong) solution to the SDE on $[t, T]$,
	\begin{equation}\label{eq:solution-wealth}
		\od X^{t, x, y;*}_s  = -(X^{t, x, y;*}_{s-} - \hat{w}) \od Z^{t,y}_s + \sqrt{\frac{\lambda}{2}} \od M^{t,y}_s, \quad X^{t, x, y;*}_t = x.
	\end{equation}
	Here $Z^{t,y} = (Z^{t,y}_s)_{s \in [t, T]}$, $M^{t,y} = (M^{t,y}_s)_{s \in [t, T]}$ are c\`adl\`ag with $Z^{t,y}_t = 0, M^{t,y}_t = 0$ given by
	\begin{equation*}
		\left\{ \begin{aligned} 
		& \od Z^{t,y}_s = (\scrM_\alpha^\tran \scrS_\alpha^{-1} )(s, Y^{t,y}_{s-}) \od Y^{t,y}_s,\\
		& \od M^{t,y}_s = \Tr[ (\scrS_\alpha^{-\frac{1}{2}} a)(s, Y^{t,y}_{s-}) \od \cW_s^\tran] + \int_{\RO \times \bR^D} \bb(  \scrS_\alpha^{-\frac{1}{2}}(s, Y^{t,y}_{s-}) \frac{u}{\psi(e)}\bb)^\tran \gamma(Y^{t,y}_{s-}) e \, \wt N^\psi_L(\od s, \od e, \od u).
	\end{aligned} \right.
	\end{equation*}
	The value function is $V^*(\cdot|\hat{w}) = v^\opt$, where $v^\opt$ is provided in \eqref{eq:ansatz-value-function}.
\end{theo}

\begin{rema}
The RL algorithms developed in \cite{JZ22,JZ23} learn the value function and the optimal measure-valued control in parametric classes of functions and probability measures (which have to be chosen beforehand). The structural results on the optimal value and the optimal control obtained in  \cref{prop:optimal-value-function} and \cref{thm:verification} facilitate such a parametrization. Indeed, \cref{prop:optimal-value-function} shows that the optimal value function is quadratic in the portfolio wealth with coefficients which can be computed in terms of the functions $\alpha$ and $\beta$. Moreover, the optimal control law is Gaussian   
 with mean $-(X^{t, x, y;*}_{-} - \hat{w})(\scrS_\alpha^{-1} \scrM_\alpha)(\cdot, Y^{t,y}_{-})$ and covariance matrix $\frac{\lambda}{2} \scrS_\alpha^{-1}(\cdot, Y^{t,y}_{-})$ by \cref{thm:verification} (and, hence, mean and covariance matrix do not depend on $\beta$). Under additional structural assumptions on the stock price model, see \cref{exam:propotional-coefficients} below, the PIDE for $\alpha$ can be solved in closed form, leading to an explicit parametrization for the optimal control law.
  
We also remark that, in general, the mean of the optimal control linearly depends on the associated portfolio wealth, while its covariance matrix is independent of the portfolio wealth. 
\end{rema}

\begin{proof}[Proof of \cref{thm:verification}] Let us fix $(t, x, y) \in [0, T) \times \bR \times \bR^D$. For the sake of notational simplicity, in the presentation below we omit the super-scripts $(t, y)$ and $(t, x, y)$ in relevant processes such as  $Y^{t, y}$ in \eqref{eq:SDE-Y-at-t}, $X^{t, x, y; H}$ in \eqref{eq:couple-wealth-price}, and $M^{t,y}$, $Z^{t,y}$. Since $\bE\b[\int_t^T \Tr[(\Sigma \scrS_\alpha^{-1})(s, Y_{s-})] \od s\b] <\infty$ by  \eqref{eq:thm:verification-condition-alpha-2}, it implies that $M$ is a uniformly square integrable martingale with $\bE[\max_{t \le s \le T}|M_s|^2] <\infty$ due to Doob's maximal inequality. By assumption \eqref{eq:thm:verification-condition-alpha-1}, we apply \cref{lem:solution-SDE} to infer that the SDE \eqref{eq:solution-wealth} has a unique c\`adl\`ag solution $X^*$ with 
	\begin{align}\label{eq:integrability-optimal-SDE}
		\bE\bb[\sup_{t \le s \le T} |X^*_s|^2\bb] < \infty.
	\end{align}
	
	\smallskip
	
	\textbf{\textit{Step 1.}} Take $H \in \cA(t, y)$ arbitrarily.  For $v^\opt$ given in \eqref{eq:ansatz-value-function}, one has
	\begin{align*}
		V^H(t, x, y|\hat{w}) = \bE\bb[v^\opt(T, X^H_T, Y_T) + \lambda \int_t^T \nq \int_{\bR} p^H_s(u) \log p^H_s(u) \od u \od s\bb].
	\end{align*}
Applying It\^o's formula (see, e.g., \cite[Theorem 2.5]{Ku04}) for $v^\opt \in C^{1,2}([0, T] \times \bR^{1+D})$ and $X^H$, $Y$ we obtain, a.s., for $t <r \le T$,
 \begin{align}\label{eq:Ito-value-function}
 	& v^\opt(r, X^{H}_{r}, Y_{
 		r}) - v^\opt(t, x, y)  =  \int_t^{r} \pd_t v^\opt (s, X^H_{s-}, Y_{s-}) \od s \notag \\
 		& \quad + \int_t^{r} \pd_x v^\opt (s, X^H_{s-}, Y_{s-}) (\mu^H_s)^\tran b(Y_{s-}) \od s +  \int_t^{r} \scrD_y v^\opt (s, X^H_{s-}, Y_{s-})^\tran b(Y_{s-}) \od s \notag  \\
 	& \quad + \int_t^r \pd_x v^\opt (s, X^H_{s-}, Y_{s-}) \B((\mu^H_s)^\tran a(Y_{s-}) \od W_s + \Tr[(\Theta^H_s)^{\frac{1}{2}} a(Y_{s-}) \od \cW_s^\tran]\B) \notag\\
 	& \quad  + \int_t^{r} \scrD_y v^\opt (s, X^H_{s-}, Y_{s-})^\tran a(Y_{s-}) \od W_s \notag\\
 	& \quad + \frac{1}{2} \int_t^{r} \pd^2_{xx} v^\opt (s, X^H_{s-}, Y_{s-}) \B((\mu^H_{s})^\tran A(Y_{s-}) \mu^H_{s}  + \Tr[A(Y_{s-}) \Theta^H_s]\B) \od s \notag \\
 	& \quad  + \int_t^{r} \B((\mu^H_{s})^\tran A(Y_{s-}) \scrD^2_{xy} v^\opt (s, X^H_{s-}, Y_{s-})  + \frac{1}{2} \Tr[ \scrD^2_{yy} v^\opt (s, X^H_{s-}, Y_{s-}) A(Y_{s-})]\B) \od s \notag  \\
 	& \quad + \int_{(t, r] \times \RO \times \bR^D} \bb[v^\opt \bb(s, X^H_{s-} + H_{s}\bb(\frac{u}{\psi(e)}\bb)^\tran \gamma(Y_{s-}) e, Y_{s-} + \gamma(Y_{s-}) e\bb) \notag\\
 	& \hspace{250pt} - v^\opt(s, X^H_{s-}, Y_{s-})\bb] \wt N^\psi_L(\od s, \od e, \od u) \notag  \\
 	& \quad + \int_{(t, r] \times \RO \times \bR^D} \bb[ v^\opt \bb(s, X^H_{s-} + H_{s}\bb(\frac{u}{\psi(e)}\bb)^\tran \gamma(Y_{s-}) e, Y_{s-} + \gamma(Y_{s-}) e\bb) - v^\opt(s, X^H_{s-}, Y_{s-}) \notag
 	\\
 	& \qquad - \pd_x v^\opt(s, X^H_{s-}, Y_{s-}) H_{s}\bb(\frac{u}{\psi(e)}\bb)^\tran \gamma(Y_{s-}) e - \scrD_y v^\opt(s, X^H_{s-}, Y_{s-})^\tran \gamma(Y_{s-}) e \bb] \nu^\psi_L(\od e, \od u) \od s.
 \end{align}
We let $z: = \frac{u}{\psi(e)}$ and denote by $P(s, e, z)$ the integrand against $\nu^\psi_L(\od e, \od u) \od s$ in \eqref{eq:Ito-value-function}. It follows from the explicit form of $v^\opt$ that
\begin{align*}
	P(s, e, z) & =  \alpha(s, Y_{s-} + \gamma(Y_{s-}) e) (X^H_{s-} + H_s(z)^\tran \gamma(Y_{s-}) e - \hat{w})^2 -  \alpha(s, Y_{s-}) (X^H_{s-} - \hat{w})^2 \\
	& \quad + \beta(s, Y_{s-} + \gamma(Y_{s-}) e) - \beta(s, Y_{s-})  - 2 \alpha(s, Y_{s-})(X^H_{s-} - \hat{w}) H_s(z)^\tran \gamma(Y_{s-}) e\\
	& \quad - \scrD_y \alpha(s, Y_{s-})^\tran  (X^H_{s-} - \hat{w})^2 \gamma(Y_{s-})e  - \scrD_y \beta(s, Y_{s-})^\tran \gamma(Y_{s-})e\\
	& = (X^H_{s-} - \hat{w})^2\B[\alpha(s, Y_{s-} + \gamma(Y_{s-}) e) - \alpha(s, Y_{s-}) - \scrD_y \alpha(s, Y_{s-})^\tran \gamma(Y_{s-}) e \B]\\
	& \quad + 2(X^H_{s-} - \hat{w})\B[ \alpha(s, Y_{s-} + \gamma(Y_{s-})e) - \alpha(s, Y_{s-}) \B] H_s(z)^\tran \gamma(Y_{s-}) e\\
	& \quad + \alpha(s, Y_{s-} + \gamma(Y_{s-}) e) (H_s(z)^\tran \gamma(Y_{s-}) e)^2 \\
	& \quad + \beta(s, Y_{s-} + \gamma(Y_{s-}) e) - \beta(s, Y_{s-}) - \scrD_y \beta(s, Y_{s-})^\tran \gamma(Y_{s-})e.
\end{align*}
Let $\Upsilon^{(0)}_\alpha, \Upsilon^{(1)}_\alpha, \Upsilon^{(2)}_\alpha$ and $\Upsilon^{(2)}_\beta$ respectively be (continuous) $\Upsilon$-dominating functions of $\alpha$ and $\beta$ in the sense of \cref{defi:Upsilon-class}. Then, for some constant $c_D>0$ depending only on $D$, we get, a.s, 
\begin{align*}
	&\int_t^T \nq \int_{\RO\times \bR^D} |P(s, e, z)| \nu^\psi_L(\od e, \od u) \od s  = 	\int_t^T \nq \int_{\RO\times \bR^D} |P(s, e, u)| \nu(\od e) \varphi_D(u) \od u \od s \\
	& \le \int_t^T (X^H_{s-} - \hat{w})^2 \Upsilon^{(2)}_\alpha(s, Y_{s-}) \od s  \\
	& \quad + 2 c_D \int_t^T |X^H_{s-} - \hat{w}| \bb(\int_{\bR^D} \|H_s(u)\| \varphi_D(u) \od u\bb) \|\gamma(Y_{s-})\| \Upsilon^{(1)}_\alpha(s, Y_{s-}) \od s\\
	& \quad +  c_D \int_t^T  \|\gamma(Y_{s-})\|^2 \Upsilon^{(0)}_\alpha(s, Y_{s-})  \bb( \int_{\bR^D} \|H_s(u)\|^2 \varphi_D(u) \od u\bb) \od s + \int_t^T \Upsilon^{(2)}_\beta(s, Y_{s-}) \od s\\
	&  <\infty,
\end{align*}
where we use the c\`adl\`ag property of $X^H, Y$ and assumption \eqref{adm-control-omegawise-integrability} to deduce the finiteness.\\
Let $Q(s, e, z)$ denote the integrand against $\wt N^\psi_L$ in \eqref{eq:Ito-value-function} and define
\begin{align*}
R(s,e,z) & : = \B[2\alpha(s, Y_{s-})(X^H_{s-} - \hat{w}) H_s(z) + \scrD_y \alpha(s, Y_{s-})  (X^H_{s-} - \hat{w})^2  + \scrD_y \beta(s, Y_{s-})\B]^\tran  \gamma(Y_{s-})e\\
& =: \wt R(s, z)^\tran \gamma(Y_{s-})e
\end{align*}
so that
\begin{align*}
	Q(s, e, z) = P(s, e, z) + R (s, e, z).
\end{align*}
Then, there is a constant $c'_D>0$ such that, a.s.,
\begin{align*}
&\int_t^T \nq \int_{\RO \times \bR^D} |R(s, e, z)|^2 \nu^\psi_L(\od e, \od u) \od s = \int_t^T \nq \int_{\RO \times \bR^D} |R(s, e, u)|^2 \nu(\od e) \varphi_D(u) \od u \od s\\
& \le c'_D \bb(\int_{\RO} \|e\|^2 \nu(\od e) \bb) \int_t^T \nq \int_{\bR^D} \|\wt R(s, u)\|^2 \|\gamma(Y_{s-})\|^2 \varphi_D(u) \od u \od s\\
& <\infty.
\end{align*}
On the other hand, by rearranging terms we get a predictable process $\phi^H$ and a local martingale $U^H$ null at $t$ such that
\begin{align*}
	 v^\opt(r, X^{H}_{r}, Y_{
	 	r}) - v^\opt(t, x, y) = \int_t^{r} \phi_s^H \od s + U^H_{r}.
\end{align*}
Since $v^\opt$ solve the HJB equation \eqref{eq:HJB-ansatz} and any $H \in \cA(t, y)$ is sub-optimal in general, we arrive at, a.s, 
\begin{align}\label{eq:Ito-formula-decomposition}
	v^\opt(r, X^{H}_{r}, Y_{
		r}) - v^\opt(t, x, y) \ge -\lambda \int_t^{r} \nq \int_{\bR^D} p^H_s(u) \log p^H_s(u) \od u \od s + U^H_{r}.
\end{align}
To deal with $U^H$, we define the localizing sequence $(\tau_n)_{n \ge 1}$ as follows
\begin{align*}
	\tau_n : = T \wedge \inf\bb\{r \in (t, T] : & \int_t^r\bb(  \int_{\RO \times \bR^D} (|P(s, e, u)| + |R(s, e, u)|^2) \nu(\od e) \varphi_D(u) \od u \\
	&  + \scrD_y v^\opt (s, X^H_{s-}, Y_{s-})^\tran A(Y_{s-}) \scrD_y v^\opt (s, X^H_{s-}, Y_{s-}) \\
	&  + |\pd_x v^\opt (s, X^H_{s-}, Y_{s-})|^2\B((\mu^H_s)^\tran A(Y_{s-}) \mu^H_s + \Tr[A(Y_{s-}) \Theta^H_s]\B) \bb)  \od s \ge n \bb\}.
\end{align*}
Since the integrand against $\od s$ in the definition of $\tau_n$ is integrable on $[t, T]$ a.s., the integral $\int_t^r (\cdots) \od s$ is finite and non-decreasing in $r$ a.s., and hence $(\tau_n)_{n \ge 1}$ is a non-decreasing sequence of stopping times converging a.s. to $T$ as $n \to \infty$. We note that the local martingale $U^H$ on the right-hand side of \eqref{eq:Ito-formula-decomposition} is an integrable martingale null at $t$ when stopping at $\tau_n$, and hence, vanishes when taking the expectation. Therefore,
\begin{align*}
	v^\opt(t, x, y) & \le  \bE\bb[v^\opt(\tau_n, X^H_{\tau_n}, Y_{\tau_n}) + \lambda \int_t^{\tau_n} \nq \int_{\bR^D} p^H_s(u) \log p^H_s(u) \od u\bb]\\
	& = \bE\bb[ \alpha(\tau_n, Y_{\tau_n}) (X^H_{\tau_n} - \hat{w})^2 + \beta(\tau_n, Y_{\tau_n}) + \lambda \int_t^{\tau_n} \nq \int_{\bR^D} p^H_s(u) \log p^H_s(u) \od u\bb].
\end{align*}
By assumption, $\alpha$ is continuous and bounded, $\beta$ is continuous and $\{\beta(\tau_n, Y_{\tau_n})\}_{n \ge 1}$ is uniformly integrable, and the entropy term is also uniform integrable for $H \in \cA(t, y)$, we exploit \eqref{eq:integrability-admissible-H} and use the dominated convergence theorem with keeping in mind that $(\tau_n)_{n\ge 1}$ is a.s. eventually constant $T$ to get
\begin{align*}
	v^\opt(t, x, y) \le  \bE\bb[v^\opt(T, X^H_{T}, Y_{T}) + \lambda \int_t^{T} \nq \int_{\bR^D} p^H_s(u) \log p^H_s(u) \od u\bb] = V^H(t, x, y|\hat{w}).
\end{align*}
Since $H \in \cA(t, y)$ is arbitrary, it implies that $v^\opt(t, x, y) \le V^*(t, x, y|\hat{w})$.

\smallskip

\textbf{\textit{Step 2.}} As suggested by \eqref{eq:optimal-control-ansatz-F}, $H^*$ provided in \eqref{eq:optimal-stochastic-control}  is a candidate for optimal controls. If $H^*$  is admissible, then we can apply the arguments in \textbf{\textit{Step 1}} for $H^*$, where inequality \eqref{eq:Ito-formula-decomposition} becomes an equality, to obtain
\begin{align*}
	v^{\opt}(t, x, y) = V^{H^*}(t, x, y|\hat{w}).
\end{align*}
Hence $v^{\opt}(t, x, y) = V^{*}(t, x, y|\hat{w})$. It remains to show that $H^*$ is admissible by verifying the requirements in \cref{defi:admissible-control}. Condition (H\ref{adm-control-predictability}) is obvious from the definition of $H^*$. For (H\ref{adm-control-integrability}), one has
\begin{align*}
	\mu^{H^*}_s = -(X^*_{s-} - \hat{w}) (\scrS_\alpha^{-1} \scrM_\alpha)(s, Y_{s-})  \quad \mbox{and} \quad \Theta^{H^*}_s = \frac{\lambda}{2} \scrS_\alpha^{-1}(s, Y_{s-}).
\end{align*}
Condition \eqref{adm-control-omegawise-integrability}
is straightforward due to the c\`adl\`ag property of $X^*, Y$ and the continuity of $\scrM_\alpha$, $\scrS_\alpha^{-1}$ on $[0, T] \times \bR^D$. For \eqref{admissible:integrability}, expanding the square and using $\int_{\bR^D} u \varphi_D(u) \od u = 0$ and $\int_{\bR^D} u u^\tran \varphi_D(u) \od u = I_D$ in the jump part we get
\begin{align*}
	& (\mu^{H^*}_s)^{\tran} A(Y_{s-}) \mu^{H^*}_s + \Tr[A(Y_{s-}) \Theta^{H^*}_s]  + \int_{\RO \times \bR^D} |H^*_s(u)^\tran \gamma(Y_{s-}) e|^2 \nu(\od e) \varphi_D(u)\od u \\
	& = (X^*_{s-}-\hat{w})^2 (\scrM_\alpha^\tran \scrS_\alpha^{-1} A \scrS_\alpha^{-1} \scrM_\alpha)(s, Y_{s-}) + \frac{\lambda}{2} \Tr[(A \scrS_\alpha^{-1})(s, Y_{s-})] \\
	&  + (X^*_{s-}-\hat{w})^2 \bb(\scrM_\alpha^\tran \scrS_\alpha^{-1} \gamma \int_{\RO}    e e^\tran \nu(\od e) \gamma^\tran \scrS_\alpha^{-1} \scrM_\alpha\bb)(s, Y_{s-}) + \frac{\lambda}{2} \Tr\bb[\bb(\gamma \int_{\RO}    e e^\tran \nu(\od e) \gamma^\tran \scrS_\alpha^{-1}\bb) (s, Y_{s-})\bb]\\
	& =(X^*_{s-}-\hat{w})^2 (\scrM_\alpha^\tran \scrS_\alpha^{-1}  \Sigma \scrS_\alpha^{-1} \scrM_\alpha)(s, Y_{s-}) + \frac{\lambda}{2} \Tr[(\Sigma \scrS_\alpha^{-1})(s, Y_{s-})].
\end{align*}
In addition, using H\"older's inequality yields 
\begin{align*}
	\bb|\int_t^T|(\mu^{H^*}_s)^\tran b(s, Y_{s-})| \od s\bb|^2
	 & \le \bb(\int_t^T (X^*_{s-} - \hat{w})^2 \od s \bb) \bb(\int_t^T |(\scrM_\alpha^\tran \scrS_\alpha^{-1}  b)(s, Y_{s-})|^2 \od s\bb).
\end{align*}
Hence \eqref{admissible:integrability} is satisfied by using \eqref{eq:thm:verification-condition-alpha-1}, \eqref{eq:thm:verification-condition-alpha-2} and \eqref{eq:integrability-optimal-SDE}. To verify (H\ref{admissible:entropy}), we might take $p^{H^*}_s(\cdot)$ to be the continuous density function of the Gaussian distribution $\cN(\mu^{H^*}_s, \Theta^{H^*}_s)$ with mean $\mu^{H^*}_s$ and covariance matrix $\Theta^{H^*}_s$, and then \eqref{eq:admissible:entropy} follows from \eqref{eq:thm:verification-condition-alpha-2}.
\end{proof}

\subsection{Explicit solutions of optimal exploratory SDEs and Lagrange multipliers} \label{sec:explicit-solutions} 
As an advantage of our approach, the optimal exploratory dynamic \eqref{eq:solution-wealth} is a linear SDE with jumps which enables us to find its solutions in a closed-form. As a consequence, we can also explicitly determine the Lagrange multiplier $\hat{w}$ using the constraint $\bE[X^{H^*}_T] = \hat{z}$, where $H^*$ is given in \eqref{eq:optimal-stochastic-control}. 

We consider \cref{prob:MV-entropy} and assume the assumptions of \cref{thm:verification} for $(t, x, y)  = (0, x_0, y_0)$, and omit super-scripts $(0, y_0)$ and $(0, x_0, y_0)$ in relevant processes.

\begin{prop}\label{prop:explicit-SDE-multiplier} Under the assumptions of \cref{thm:verification}, if $\Delta Z \neq 1$ on $[0, T]$ then the optimal wealth process $X^* = (X^*_r)_{r \in [0, T]}$ in \eqref{eq:solution-wealth} is given by
	\begin{align}\label{eq:explicit-solution-optimal-SDE}
		X^*_r & =  \hat{w} + \bb[x_0 - \hat{w} + \sqrt{\frac{\lambda}{2}}\bb(\int_0^r \frac{\od M_s}{\cE(-Z)_{s-}} + \int_0^r \frac{\od [M, Z]_s}{\cE(-Z)_{s-}}\bb)\bb] \cE(-Z)_r, \quad r \in [0, T],
	\end{align}
	where $\cE(-Z) = (\cE(-Z)_r)_{r \in [0, T]}$ denotes the Dol\'eans--Dade exponential\footnote{See, e.g., \cite[Section II.8]{Pr05}.} of $-Z$, i.e.
	\begin{align*}
		\cE(-Z)_0 = 1, \quad \cE(-Z)_r = \exp\bb(-Z_r - \frac{1}{2} \int_0^r \od [Z, Z]^c_s\bb) \prod_{0 < s \le r} (1- \Delta Z_s) \e^{\Delta Z_s}, \quad r \in (0, T].
	\end{align*}
	Here the quadratic covariation terms are explicitly expressed as follows
	\begin{align*}
		\od [M, Z]_s & = \int_{\RO \times \bR^D} \frac{1}{\psi(e)} e^\tran \B[\gamma^\tran \scrS_\alpha^{-1} \scrM_\alpha (\scrS_\alpha^{-\frac{1}{2}} u)^\tran  \gamma\B](s, Y_{s-}) e\, N^\psi_L (\od s, \od e, \od u),\\
		\od [Z, Z]_s^c & =  ( \scrM_\alpha^\tran \scrS^{-1}_\alpha A \scrS^{-1}_\alpha \scrM_\alpha)(s, Y_{s-}) \od s.
	\end{align*}
Moreover, if $\bE[\cE(-Z)_T] \neq 1$ then the Lagrange multiplier $\hat{w}$ (such that $\bE[X^*_T] = \hat{z}$) is given by
\begin{align}\label{eq:explicit-Lagrange-multiplier}
	\hat{w} = \frac{1}{1- \bE[\cE(-Z)_T]} \bb(\hat{z} - \sqrt{\frac{\lambda}{2}} \bE\bb[\bb(\int_0^T \frac{\od M_s}{\cE(-Z)_{s-}} + \int_0^T \frac{\od [M, Z]_s}{\cE(-Z)_{s-}}\bb) \cE(-Z)_T\bb] - x_0 \bE[\cE(-Z)_T]\bb).
\end{align}
\end{prop}

\begin{proof}
	For $Z$ given in \cref{thm:verification}, we write
	$$- (X^*_{s-} - \hat{w}) \od Z_s = (X^*_{s-} - \hat{w}) \od (- Z_s).$$
	Since $\Delta Z \neq 1$ by assumption, it implies that $\inf\{s \in (0, T] : 1 - \Delta Z_s =0\} = \infty$ a.s. We then apply \cite[Exercise V.27]{Pr05} to obtain the explicit representation for $X^*$ as in \eqref{eq:explicit-solution-optimal-SDE}.
	
	For the Lagrange multiplier $\hat{w}$, we first notice that $\cE(-Z)$ satisfies the following SDE on $[0, T]$
	\begin{align*}
		\cE(-Z)_r = 1 + \int_0^r \cE(-Z)_{s-} \od (-Z_s).
	\end{align*}
	Since the \textit{conditional} quadratic variation\footnote{See, e.g., \cite[Chapter III, p.124]{Pr05}.} of the integrator $-Z$ is
	\begin{align*}
		\<-Z, -Z\>_T = \int_0^T (\scrM_\alpha^\tran \scrS_\alpha^{-1} \Sigma \scrS_\alpha^{-1} \scrM_\alpha)(s, Y_{s-}) \od s \le c_{\eqref{eq:thm:verification-condition-alpha-1}} T,
	\end{align*}
	it follows from condition \eqref{eq:thm:verification-condition-alpha-1} and \cref{lem:solution-SDE} that $\cE(-Z)$ is square integrable. Since $X^*$ is also square integrable by \eqref{eq:integrability-optimal-SDE}, letting $r = T$ and taking the expectation both sides of \eqref{eq:explicit-solution-optimal-SDE} we rearrange terms and use the constraint $\bE[X^*_T] = \hat{z}$ to obtain \eqref{eq:explicit-Lagrange-multiplier}.
\end{proof}

\begin{rema}\label{remark:jump-size-Z}
	\begin{enumerate}[\quad (1)]
		\itemsep2pt
		\item \label{remark:jump-size-Z:AC} If $\nu$ is absolutely continuous with respect to the $D$-dimensional Lebesgue measure $\Leb_D$, then  $\Delta Z \neq 1$ on $[0, T]$. Indeed, by letting $J_t : = \int_0^t \int_{\RO \times \bR^D} e \wt N^\psi_L(\od s, \od e, \od u)$ so that $J$ admits $\nu$ as its L\'evy measure we get
		\begin{align*}
			\bE[\#\{s \in (0, T] : \Delta Z_s = 1\}] & = \bE[\#\{s \in (0, T] : (\scrM_\alpha^\tran \scrS_\alpha^{-1} \gamma)(s, Y_{s-}) \Delta J_s = 1\}]\\
			& = \bE\bb[\int_0^T \nq \int_{\RO} \1_{\{(\scrM_\alpha^\tran \scrS_\alpha^{-1} \gamma)(s, Y_{s-}) e = 1\}} \nu(\od e) \od s \bb]\\
			& = 0,
		\end{align*}
		where we combine Fubini's theorem with the fact that  hyperplanes have Lebesgue measure zero to obtain the last equality. Hence, $\#\{s \in (0, T] : \Delta Z_s = 1\} = 0$ a.s. 
		
		\item If the condition ``$\Delta Z \neq 1$ on $[0, T]$'' in \cref{prop:explicit-SDE-multiplier} is not satisfied, then we can still obtain explicit representations of $X^*$ and $\hat{w}$. However, as these expressions are rather technical, we refer the interested readers to \cite{BT23S} for more details.
	\end{enumerate}
\end{rema}

\subsection{Illustrative examples}

 Let us consider some situations in which assumptions of \cref{prop:optimal-value-function} and \cref{thm:verification} are validated. For matrices $P, Q \in \bS^D$ we write $P \preceq Q$ or $Q \succeq P$ if $Q - P \in \bS^D_+$.

\begin{exam}[Proportional coefficients]\label{exam:propotional-coefficients}
Let $b, a, \gamma$ in \cref{sec:setting} satisfy $b^{(i)}, a^{(i,j)}, \gamma^{(i, j)} \in C^\infty_b(\bR^D)$ for all $i, j = 1, \ldots, D$. Assume that there are constants $K>0$ and $\ep>0$ such that, for all $y \in \bR^D$,
\begin{equation}\label{eq:exam:dependent-coefficients}
	\left\{ \begin{aligned} 
		& A(y) \succeq \ep I_D,\\
		&(b^\tran \Sigma^{-1}b)(y) = K.
		\end{aligned} \right.
\end{equation}
For example, if there exist $U\colon \bR^D \to \bS^D_{++}$ with $U^{(i,j)} \in C^\infty_b(\bR^D)$, $i, j = 1, \ldots, D$, and a constant $\delta>0$ such that $U(y) \succeq \delta I_D$ for all $y \in \bR^D$ and that, for some constant $\tilde b \in \bR^D$, $\tilde a, \tilde \gamma \in \bR^{D\times D}$ with $\tilde b \neq 0$ and $\det(\tilde a) \neq 0$,
\begin{align*}
	b(y) = U(y) \tilde b, \quad a(y) = U(y) \tilde a,\quad  \gamma(y) = U(y) \tilde\gamma,
\end{align*} 
then condition \eqref{eq:exam:dependent-coefficients} holds true with $K = \tilde b^\tran(\tilde a \tilde a^\tran + \tilde \gamma \int_{\RO} e e^\tran \nu(\od e) \tilde \gamma^\tran)^{-1} \tilde b$ and $\ep = \delta^2 \tilde \ep$, where $\tilde \ep>0$ is sufficiently small such that $\tilde a \tilde a^\tran \succeq \tilde \ep I_D$.

Now, under \eqref{eq:exam:dependent-coefficients}, \cref{assum:SDE-coefficients} is obviously satisfied. Moreover, since $\Sigma(y) \succeq  A(y)$, it follows from the ellipticity condition $A(y) \succeq \ep I_D$ that $\Sigma(y) \succeq \ep I_D$. Hence, \cref{lem:matrix-inequality} gives
\begin{align*}
	\Sigma^{-1}(y) \preceq \frac{1}{\ep} I_D, \quad \forall y \in \bR^D.
\end{align*}
Consequently, one has $\sup_{y \in \bR^D}\|\Sigma^{-1}(y)\| <\infty$.

We first find solution $\alpha$ of the PIDEs \eqref{system:PDEs} which does not depend on $y$. For $\alpha(t, \cdot) = \alpha(t)$, we get $$\scrS_\alpha(t, y) = \alpha(t) \Sigma(y), \quad \scrM_\alpha(t, y) = \alpha(t) b(y)$$
so that $(\scrM^\tran_\alpha \scrS^{-1}_\alpha \scrM_\alpha)(t, y) = \alpha(t) K$.
 Then the PIDE for $\alpha$ in \eqref{system:PDEs} boils down to the following ordinary differential equation (ODE)
\begin{equation*}
	\left\{ \begin{aligned} 
		&\alpha'(t)   - \alpha(t) K = 0, \quad t \in [0, T), \\
		&\alpha(T) = 1,
	\end{aligned} \right.
\end{equation*}
whose solution is given by
\begin{align*}
	\alpha(t) = \e^{-(T-t)K}, \quad t \in [0, T].
\end{align*}
It is easy to check that the assumptions of \cref{prop:optimal-value-function} and \cref{thm:verification} are satisfied for $\alpha$. Next, plugging this $\alpha$ into the PIDE for $\beta$ in  \eqref{system:PDEs} we obtain
		\begin{equation}\label{exam:PIDE-beta}
			\left\{ \begin{aligned} 
				& \pd_t \beta (t,y) + \cL_Y \beta(t,y) - \frac{\lambda}{2} \log\bb( \frac{(\lambda \pi)^D}{\det(\alpha(t) \Sigma(y))}\bb) = 0, \quad t \in [0, T),\\
				&\beta(T,\cdot) = 0.
			\end{aligned} \right.
		\end{equation}
	We apply \cite[Theorem 1]{MYZ10} to conclude that the PIDE \eqref{exam:PIDE-beta} has a unique classical solution $\beta \in C^{1,2}([0, T] \times \bR^D)$. Moreover, $\beta$ and its partial derivatives $\pd_t \beta, \scrD_y \beta, \scrD^2_{yy} \beta$ are  uniformly bounded on $[0, T] \times \bR^D$. Then $\beta$ also satisfies the assumptions of \cref{prop:optimal-value-function} and \cref{thm:verification}. The Feynman--Kac representation for $\beta$ (see, e.g., \cite[Theorems 3.4 and 3.5]{BBP97}) is 
	\begin{align*}
		\beta(t, y) =  \bE\bb[\int_t^T f(s, Y^{t, y}_s) \od s \bb], \quad (t, y) \in [0, T] \times \bR^d,
	\end{align*}	
	where $f(t,y): = - \frac{\lambda}{2} \log
	( \frac{(\lambda \pi)^D}{\det(\alpha(t) \Sigma(y))})$ and $Y^{t, y}$ is given in \eqref{eq:SDE-Y-at-t}. The value function is
\begin{align*}
	V^*(t, x, y|\hat{w}) = \e^{-(T-t)K} (x- \hat{w})^2 + \beta(t, y).
\end{align*}
The associated exploratory SDE for $X^* = (X^*_r)_{r \in [0, T]}$ is given by $X_0 = x_0$ and
\begin{align}\label{eq:exam:exploratory-SDE}
	\od X^{*}_r & = - (X^{*}_{r-} - \hat{w})\bb( K\od r + (b^\tran \Sigma^{-1} a)(Y_{r-}) \od W_r + (b^\tran \Sigma^{-1} \gamma)(Y_{r-}) \int_{E \times \bR^D} e \wt N^\psi_L(\od r, \od e, \od u)\bb) \notag\\
	& \quad  + \sqrt{\frac{\lambda}{2}}\e^{\frac{1}{2}(T-r) K}  \bb(\Tr[(\Sigma^{-\frac{1}{2}}a)(Y_{r-}) \od \cW_r^\tran] + \int_{E \times \bR^D}  \frac{u^\tran}{\psi(e)} (\Sigma^{-\frac{1}{2}} \gamma)(Y_{r-}) e \, \wt N^\psi_L(\od r, \od e, \od u)\bb),
\end{align}
whose explicit expression can be derived either from \cite{BT23S} or from \eqref{eq:explicit-solution-optimal-SDE} provided that $\Delta Z \neq 1$ on $[0, T]$.

Regarding the Lagrange multiplier $\hat{w}$, due to the condition $(b^\tran \Sigma^{-1}b)(y) = K$ in \eqref{eq:exam:dependent-coefficients}, we can simply calculate its value by taking the expectation of $X^*_r$ with noting that the martingale terms in the expression \eqref{eq:exam:exploratory-SDE} of $X^*_r$ are square integrable null at $0$, and then using Fubini's theorem to get
\begin{align*}
	\bE[X^*_r] = x_0 - K \int_0^r (\bE[X^*_s] - \hat{w}) \od s,
\end{align*}
which then gives
\begin{align*}
	\bE[X^*_r] = \hat{w} + (x_0 - \hat{w}) \e^{- K r}, \quad r \in [0, T].
\end{align*}
By the constraint $\bE[X^*_T] = \hat{z}$, we arrive at
$$\hat{w} = \frac{\hat{z}\e^{K T} - x_0}{\e^{K T} - 1}.$$
\end{exam}

\begin{exam}[Constant coefficients]\label{exam:constant-coefficients} Let $b, a, \gamma$ be constants on $\bR^D$ with $b \neq 0$ and $\Sigma \in \bS^D_{++}$, where $a$ might be degenerate.
  In this situation we can find solutions $\alpha, \beta$ of the PIDEs \eqref{system:PDEs} which do not depend on $y$. Namely, by letting $\alpha(t, \cdot) = \alpha(t)$, $\beta(t, \cdot) = \beta(t)$ and plugging them into \eqref{system:PDEs} we obtain a system of ODEs for $\alpha, \beta$ which possesses the following solutions on $[0, T]$,
	\begin{equation*}
		\left\{ \begin{aligned} 
			&\alpha(t) = \e^{-(T-t) K}, \\
			& \beta(t) =  -(T-t)^2 \frac{\lambda D}{4} K  -  (T-t) \frac{\lambda}{2} \log\bb( \frac{(\lambda \pi)^D}{\det(\Sigma)}\bb),
		\end{aligned} \right.
	\end{equation*}
	where $K: = b^\tran \Sigma^{-1} b>0$. It is also easy to check that the assumptions of \cref{prop:optimal-value-function} and \cref{thm:verification} are fulfilled for $\alpha, \beta$. Then the value function is explicitly given by 
	\begin{align*}
		V^*(t, x, y|\hat{w}) = V^*(t, x|\hat{w}) := \e^{-(T-t) K} (x-\hat{w})^2 -(T-t)^2 \frac{\lambda D}{4} K  -  (T-t) \frac{\lambda}{2} \log\bb(\frac{(\lambda \pi)^D}{\det(\Sigma)}\bb).
	\end{align*}
	The SDE for the optimal wealth $X^*$ and the Lagrange multiplier $\hat{w}$ are respectively the same as those in \cref{exam:propotional-coefficients} where one notices here that coefficients $b, a, \gamma, \Sigma$ are constant\footnote{If $\nu(\{e \in E : b^\tran \Sigma^{-1}\gamma e =1\}) = 0$, then applying the same argument as in  \cref{remark:jump-size-Z}\eqref{remark:jump-size-Z:AC} yields $\Delta Z \neq 1$ on $[0, T]$, and hence, \eqref{eq:explicit-solution-optimal-SDE} is usable.}.
\end{exam}

In the following we continue to specialize \cref{exam:constant-coefficients} to the case of no jumps.
\begin{exam}[Constant coefficients, $\nu \equiv 0$ and $D =1$]\label{exam:constant-coeff-D=1} This is the setting considered by Wang and Zhou \cite{WZ20}. For $a>0$, $b \neq 0$, letting $\sigma := a$ and $\rho := \frac{b}{a}$ we get the value function
	\begin{align*}
		V^*(t, x|\hat{w}) = \e^{- \rho^2(T- t)} (x - \hat{w})^2 - \frac{\lambda}{2} \bb(\frac{\rho^2}{2}(T-t)^2 + (T-t) \log\bb(\frac{\lambda \pi}{\sigma^2}\bb)\bb)
	\end{align*}
	which coincides with that in \cite[Theorem 3.1]{WZ20}. The associated SDE for the optimal wealth $X^*$ in our setting is
	\begin{align}\label{eq:explicit-solution-SDE-D=1}
		\od X^*_s = - \rho^2 (X^*_s - \hat{w}) \od s -  \rho (X^*_s - \hat{w}) \od W_s + \sqrt{\frac{\lambda}{2}} \e^{\frac{\rho^2}{2} (T-t)} \od \cW_s, \quad X^*_0 = x_0,
	\end{align}
	whose explicit representation is given, according to \cref{prop:explicit-SDE-multiplier}, by
	\begin{align*}
		X^*_r = \hat{w} + \bb[x_0 - \hat{w} + \sqrt{\frac{\lambda}{2}} \int_0^r \e^{\rho W_s + \frac{3}{2}\rho^2 s} \e^{\frac{1}{2} \rho^2 (T-s)} \od \cW_s \bb] \e^{-\rho W_r - \frac{3}{2}\rho^2 r}, \quad r \in [0, T].
	\end{align*}
We emphasize that the optimal exploratory SDE \eqref{eq:explicit-solution-SDE-D=1} is different from that in \cite[Eq. (27)]{WZ20} which is formulated in our notation as
\begin{align}\label{eq:explicit-solution-SDE-WZ20}
	\od \wt X^*_s = - \rho^2 (\wt X^*_s - \hat{w}) \od s + \sqrt{\rho^2 (\wt X^*_s - \hat{w})^2 + \frac{\lambda}{2} \e^{\rho^2 (T-s)}} \od W_s, \quad \wt X^*_0 = x_0.
\end{align}
However, solutions $X^*$ of \eqref{eq:explicit-solution-SDE-D=1} and $\wt X^*$ of \eqref{eq:explicit-solution-SDE-WZ20} have the same (finite-dimensional) distribution because of the uniqueness in law of  \eqref{eq:explicit-solution-SDE-WZ20}.
\end{exam}

\subsection{Relation to the sample state process}\label{sec:sample_state_process}

The continuous-time RL algorithms designed in \cite{JZ22,JZ23} rely on the sample state process for the actual learning task, which is the solution of an SDE that models the state dynamics evaluated along a randomized control. In this subsection, we explain the relation between  our exploratory dynamics and the sample state process. To avoid technicalities and to highlight the key ideas, we focus on the situation in  \cref{exam:constant-coeff-D=1}. 

The construction of the sample state process in \cite{JZ22,JZ23} starts with a feedback control  $\pi(\cdot|t,x)$ with values in the space of probability density functions. If the portfolio value is in state $X_t=x$ at time $t$, the portfolio position $H_t$ is randomly drawn from the probability density $\pi(\cdot|t,x)$. The actual drawing of the portfolio position is performed based on a family $(Z_t)_{t\in [0,T]}$ of independent random variables in \cite{JZ23}, which are uniformly distributed on $[0,1]$.  This family is supposed to be independent of the stochastic processes driving the stock price, hence, of the Brownian motion $W$  in the context of  \cref{exam:constant-coeff-D=1}. Denoting by $h_\pi(t,x;\cdot)$ the quantile function of the distribution with density $\pi(\cdot|t,x)$, the random drawing of the portfolio position can be made explicit by letting $H_t=h_\pi(t, X_t, Z_t)$, which formally leads to the SDE
\begin{equation}\label{eq:sample_1}
	  \od X^{h_\pi}_t= h_\pi(t, X^{h_\pi}_t, Z_t)(b \od t +\sigma  \od W_t),\quad X^{h_\pi}_0=x.
\end{equation}
In the terminology of \cite{JZ23}, the solution of this SDE is the \emph{sample state process}  corresponding to the \emph{action process} $h_\pi(t, X_t, Z_t)$, which is sampled from the given density $\pi$.\footnote{The authors in \cite{JZ23} do not give an explicit construction of the action process, whereas we use the construction based on the quantile function in this subsection. It is, however, clear from the presentation in \cite{JZ23} that iid uniform random variables $(Z_t)_{t\in [0,T]}$ independent of $W$ are applied for the control randomization mechanism in \cite{JZ23}.}

Given the optimality of Gaussian randomization, which in the context of \cref{exam:constant-coeff-D=1} has first been derived in \cite{WZ20}, we now specialize to the case that $\pi(\cdot|t,x)$ is a Gaussian density with mean $\mu(t,x)$ and standard deviation $\vartheta(t, x) : = \theta(t,x)^{\frac{1}{2}}$. Then,  $h_\pi(t, X_t,Z_t)=\mu(t,X_t)+ \vartheta(t,X_t)\xi_t$, where $(\xi_t)_{t\in [0,T]}$ is an independent family of standard Gaussians constructed from $(Z_t)_{t\in [0,T]}$ via $\xi_t=\Phi^{-1}(Z_t)$ ($\Phi$ denoting the cumulative distribution function of the standard normal distribution). Thus, \eqref{eq:sample_1} becomes (suppressing the superscript $h_\pi$)
\begin{equation}\label{eq:sample_2}
	\od X_t=b \mu(t,X_t)\od t + b \vartheta(t,X_t)\xi_t \od t+\sigma \mu(t,X_t) \od W_t+ \sigma \vartheta(t,X_t)\xi_t \od W_t ,\quad X_0=x.
\end{equation}

If we write $\mathbb{G}=(\mathcal{G}_t)_{t\in [0,T]}$ for the filtration generated by $W$ and $\xi$, then $\xi$ becomes an adapted process, but it is well-known that non-constant families of independent, identically distributed random variables indexed by continuous time cannot be measurable with respect to the standard product $\sigma$-field, see, e.g., \cite[Proposition 2.1]{Su06}. Hence, $\xi$ fails to be progressively measurable in the usual sense of stochastic calculus, and so the SDE \eqref{eq:sample_2} cannot be studied in the classical SDE framework. To deal with this problem, the authors in \cite{JZ23} refer to the framework of rich Fubini extensions developed in \cite{Su06, SZ09}. Roughly speaking, one can construct an extension $\bar \Leb$ of the Lebesgue measure on $[0,T]$ beyond the $\sigma$-field of Lebesgue-measurable sets and a suitable probability space such that the process $\xi$ becomes measurable with respect to some appropriate Fubini extension of the classical product measure space, see \cite[Theorem 2]{Po10} for a precise statement. Here the notion of a \emph{Fubini extension} refers to the property that a suitable reformulation of Fubini's theorem on iterated integration is still valid, see \cite{Su06}. Then, the Lebesgue integrals in \eqref{eq:sample_2} can be replaced by integrals with respect to the extension $\bar \Leb$ of the Lebesgue measure (but we will write $\od t$ in place of $\bar \Leb(\od t)$ below to simplify the notation). However, with this construction, it is still not clear to us, how to extend the It\^o integral to integrands which only satisfy this weaker measurability property ensured by the rich Fubini construction. In the following informal discussion, we make the conjecture that  It\^o's integral can be properly extended such that the standard results of stochastic calculus are still in force. 

Under this conjecture, we may consider
$$
A^\xi_t=\int_0^t \xi_s \od s,\quad W^\xi_t=\int_0^t \xi_s \od W_s.
$$
Then, \eqref{eq:sample_2} can be rewritten as
$$
\od X_t=b \mu(t,X_t)\od t + b \vartheta(t,X_t) \od A^\xi_t+\sigma \mu(t,X_t) \od W_t+ \sigma \vartheta(t,X_t) \od W^\xi_t ,\quad X_0=x.
$$
By the above conjecture, $W^\xi$ is a continuous martingale with
$$
\langle W^\xi \rangle_t= \int_0^t \xi^2_s \od s,\quad \langle W^\xi, W \rangle_t= \int_0^t \xi_s \od s.
$$
Sun's exact law of large numbers \cite[Theorem 2.6]{Su06}  developed in the framework of rich Fubini extensions  now implies
$$
\int_0^t \xi^q_s \od s =\int_0^t \bE[\xi^q_s]\od s = \begin{cases} 0, & q=1, \\ t,& q=2.\end{cases}
$$ 
Hence $A^\xi \equiv 0$ and, by L\'evy's characterization, $W^\xi$ is a Brownian motion independent of $W$. Thus, the SDE \eqref{eq:sample_2} for the sample state process with Gaussian randomization  becomes
\begin{equation}\label{eq:sample_3}
	\od X_t=b \mu(t,X_t)\od t +\sigma \mu(t,X_t) \od W_t+ \sigma \vartheta(t,X_t)\od W^\xi_t ,\quad X_0=x.
\end{equation}
This is exactly our form of the exploratory SDE \eqref{eq:exploratory-SDE-jump} (with a different notation for the additional independent Brownian motion), when specialized to the setting of  \cref{exam:constant-coeff-D=1} and applied to feedback controls with Gaussian randomization. 

Let us now look at the special case, when the control randomization is performed according to the standard Gaussian distribution independent of time and state, i.e., $\mu(t,x)=0$ and $\vartheta(t,x)=1$. Then, the corresponding portfolio wealth process $X_t=x+\sigma W^\xi_t$ in \eqref{eq:sample_3} is independent of $W$, and, hence, independent of the stock price dynamics. This appears to be counter-intuitive and illustrates that the SDE   for the sample state process may not be able properly describe the portfolio wealth along a randomized portfolio (with continuous re-sampling in the randomization mechanism). In order  justify the use of this SDE, we discretize the portfolio  process $H_t=\xi_t$ on a time grid $0=t_0<t_1<\cdots<t_n=T$ via $H^n_t=\xi_{t_j}$, for $t\in (t_j, t_{j+1}]$. Then, $H^n$ is $\mathbb{G}$-predictable (since it is adapted and left-continuous), and, hence, its wealth process with initial endowment $x$
$$
X^n_t=x+\int_0^t H^n_s(b \od s + \sigma \od W_s)=x+b\sum_{j=0}^{n-1} \xi_{t_j} (t_{j+1}\wedge t -t_j\wedge t)+ \sigma\sum_{j=0}^{n-1} \xi_{t_j} (W_{t_{j+1}\wedge t} -W_{t_j\wedge t})
$$  
is well-defined in the framework of the classical stochastic integration theory. It is not difficult to check that the first sum converges to zero by the law of large numbers (cp. \cref{sec:dift-part}) and the second sum weakly converges to a Brownian motion independent of $W$ by Donsker's invariance principle.  Hence, the  ``wealth process'' $X_t=x+\sigma W^\xi_t$ for the non-predictable portfolio position process $H_t=\xi_t$ suggested by the sample state process in the continuous-time RL literature can be properly interpreted as the weak limit of the wealth processes of the approximating sequence of predictable randomized portfolio positions $H^n$. 
The limit result in this illustrative example is, of course, the simplest special case of our general result, \cref{thm:limit-distribution-integrator}, which motivates our formulation of the exploratory SDE.

Summarizing, the above discussion suggests that our formulation \eqref{eq:exploratory-SDE-jump} of the exploratory SDE is one way to give a mathematically rigorous meaning to the SDE which models the sample state process in the recent RL literature. Moreover, \cref{thm:limit-distribution-integrator} provides a justification for the use of this SDE formulation as a limit of a natural control randomization mechanism in discrete time. While the results in this paper are presented for the mean-variance portfolio selection problem, it is obvious, how to transfer the derivation of our exploratory SDE based on \cref{thm:limit-distribution-integrator} to more general problems with controlled diffusion and jumps, provided the control enters the diffusion part linearly. The case of general dependence of the diffusion coefficient on the control requires more advanced tools from the theory of random measures and is discussed in our follow-up work \cite{BT24}. 

\begin{rema}
	In our general setting with $D$ stocks, the Gaussian randomization leads to an action process of the form
	$$
	\mu(t,X_t,Y_t)+ \vartheta(t,X_t,Y_t)\xi_t
	$$
	where the mean $\mu(t,x,y)$ takes values in $\mathbb{R}^D$, $\vartheta(t,x,y)$ is the positive definite root of the positive definite $D\times D$ covariance matrix $\theta(t,x,y)$ and each $\xi_t$ is a vector of $D$ independent standard Gaussians. Following the same argument as above, we will consider the ``processes''
	$$
	W^{\xi,(d,d')}_t=\int_0^t \xi^{(d)}_s \od  W^{(d')}_s,
	$$
	which additionally drive the SDE for the sample state process. Then, by the L\'evy characterization as above, $(W^{d'},W^{\xi,(d,d')}; d=1,\ldots D,\; d'=1,\ldots D)$ is a $(D^2+D)$-dimensional Brownian motion. Thus, making the sample state process rigorous by the same reasoning as above, the diffusion part  is driven by a $(D^2+D)$-dimensional Brownian motion as in our formulation \eqref{eq:exploratory-SDE-jump} of the exploratory SDE.  
	\end{rema}

\section{Weak convergence of discrete-time integrators} \label{sec:proof-weak-convergence}

This section provides the proof of  \Cref{thm:limit-distribution-integrator}. Throughout this part, let $c_D$ denote a positive constant depending only on $D$, and its value might vary in each appearance. The time-change $\sigma^n$ is extended constantly over  $t \in [T, \infty)$. To cover necessary test functions for the proof of \cref{thm:limit-distribution-integrator}, we use the following function space.

\begin{defi} For $\bfD = D^2 + 3D$, we let $g \in C^2_*(\bR^{\bfD})$ if the following conditions hold:
	\begin{enumerate}[\quad ($\rmG$1):]
		\itemsep2pt
		\item\label{item:G-hessian} $g \in C^2(\bR^{\bfD})$ with $g(0) = 0$ and $\|\scrD^2 g\|_\infty <\infty$;
		
		\item \label{item:G-brownian-part} for $1 \le d \vee d' \le D^2 +D$, the function  $\pd^2_{d, d'} g$ takes value $0$ in a neighborhood of $0$;
		
		\item \label{item:G-control-brownian-part}$c_{(\rmG\ref{item:G-control-brownian-part})} : = \max_{1 \le d \le D^2 +D} \|\pd_d g(0_{D^2 +D}, \cdot)\|_\infty <\infty$, where $0_{D^2 +D}$ is the vector $0$ in $\bR^{D^2+D}$;
		
		\item \label{item:G-jump-part} $c_{(\rmG\ref{item:G-jump-part})}: = \max_{D^2+D+1 \le d \le \bfD} \|\pd_{d} g\|_\infty <\infty$ and $\pd_d g(0) = 0$ for any $D^2+D+1 \le d \le \bfD$.
	\end{enumerate}
\end{defi}

\begin{prop}\label{lemm:convergence-jump-part}
For any $g \in C^2_*(\bR^{\bfD})$, one has when $n \to \infty$ that
	\begin{align}\label{eq:lemm:convergence-jump-part}
		\sum_{i=1}^{n} \bb| \bE[g(\Delta_{n, i} \bmZ^n)|\cF_{n, i-1}] - (t^n_i - t^n_{i-1}) \int_{\bR^{2D}} g(0, e, u) \nu^\psi_{L}(\od e, \od u)\bb| \xrightarrow{\bfL_1(\bP)} 0,
	\end{align}
	where $\bmZ^n$ is given in \cref{sec:distributional-limit}.
\end{prop}

\begin{proof} With a slight abuse of notation, in the sequel we use symbols $\eta, \xi$ without any sub-indices to denote \textit{deterministic} vectors in $\bR^D$, whereas $\eta^H_{n, i}$ and $\xi_{n, i}$ are random vectors introduced in \cref{sec:exploration-procedure}. Recall that 
		$$\Delta_{n, i} \bmZ^n = \vect(\Delta_{n, i} W^n, \Delta_{n, i} M^n,	\Delta_{n, i} L^{n, \psi}) = \vect(\Delta_{n, i} W,\,\eta^H_{n,i} \pr \Delta_{n, i} W,\, \Delta_{n, i} J,\, \psi(\Delta_{n,i} J) \xi_{n,i}).$$
	\textbf{\textit{Step 1.}} Since $g(0) = 0$ by $(\rmG\ref{item:G-hessian})$ and $\pd_dg(0) = 0$ for $D^2+D+1 \le d \le \bfD$ by $(\rmG\ref{item:G-jump-part})$, an argument using Taylor expansion shows
	\begin{align}\label{eq:estimate-g}
		|g(0, e, u)| \le c_D \|\scrD^2 g\|_\infty(\|e\|^2 + \|u\|^2), \quad e, u \in \bR^D.
	\end{align}
Since $\nu_{L}^\psi$ is a square integrable L\'evy measure, it ensures that $\int_{\bR^{2D}} |g(0, e, u)| \nu_{L}^\psi(\od e, \od u) <\infty$. Moreover, for any $n, i$, since
\begin{align*}
	\bE[\|\Delta_{n, i} \bmZ^n\|^2] & = \bE[\|\Delta_{n, i} W\|^2 + \|\eta^H_{n, i} \pr \Delta_{n, i} W\|^2 + \|\Delta_{n, i} J\|^2 + \psi(\Delta_{n, i} J)^2\|\xi_{n, i}\|^2]\\
	& \le (t^n_i - t^n_{i-1})\bb(D + D^2 + \int_{\RO}\|e\|^2 \nu(\od e) + D\|\scrD\psi\|_\infty^2  \int_{\RO}\|e\|^2 \nu(\od e)\bb) <\infty,
\end{align*}
together with the fact that $g$ has at most quadratic growth at infinity as $\|\scrD^2 g\|_\infty<\infty$ by $(\rmG\ref{item:G-hessian})$, it implies that $\bE[|g(\Delta_{n, i} \bmZ^n)|] <\infty$.

	\smallskip

\noindent\textbf{\textit{Step 2.}} To shorten the notation, for each $\eta, \xi \in \bR^D$, we define  $g_{\eta, \xi} \colon \bR^{2D} \to \bR$ by 
	\begin{align*}
		g_{\eta, \xi}(w, j) : = g(w, \eta\pr w, j, \psi(j) \xi), \quad w, j \in \bR^D.
	\end{align*}
	Then, $g_{\eta, \xi} \in C^2(\bR^{2D})$. Furthermore, for any $d, d' = 1, \ldots, D$, the partial derivatives of $g_{\eta, \xi}$ are given, with the convention $\eta^{(0)}:=1$ and $z : = (w, \eta\pr w, j, \psi(j) \xi) \in \bR^{\bfD}$, by
	\begin{align}
		&\pd_d g_{\eta, \xi}(w, j)  = \sum_{k=0}^D \eta^{(k)} \pd_{d + kD}  g(z), \quad  \pd^2_{d,d'} g_{\eta, \xi}(w, j)   = \sum_{k, l =0}^D \eta^{(k)} \eta^{(l)} \pd^2_{d + k D, d'+l D} g(z),\label{eq:partial-derivative-1}\\
		 &\pd_{D+d} g_{\eta, \xi}(w, j)  = \pd_{D^2 +D+d} g(z) + \pd_d\psi(j) \sum_{k =1}^D \xi^{(k)} \pd_{D^2 + 2D + k} g(z), \label{eq:partial-derivative-2} \\
		 &\pd^2_{D+d', D+d} g_{\eta, \xi}(w, j)  = \pd^2_{D^2+D+d', D^2+D+d}g(z)  + \pd^2_{d', d} \psi(j) \sum_{k=1}^D \xi^{(k)} \pd_{D^2+2D+k}g(z) \notag \\
		& \hspace{50pt} + \pd_d \psi(j)  \sum_{k=1}^D   \xi^{(k)} \bb[\pd^2_{D^2+D+d', D^2+2D+k}g(z) + \pd_{d'} \psi(j)\sum_{l=1}^D \xi^{(l)} \pd^2_{D^2+2D+l, D^2+2D+k} g(z)\bb].\notag 
	\end{align}
	Hence, there exists a constant $c_{\eqref{eq:hessian-g-estimate}}: = c(D, \|\scrD \psi\|_\infty, \|\scrD^2\psi\|_\infty, \|\scrD^2 g\|_\infty, c_{(\rmG\ref{item:G-jump-part})})>0$ such that
	\begin{align}\label{eq:hessian-g-estimate}
		\max_{1\le d, d' \le D} \|\pd^2_{D+d', D+d}g_{\eta, \xi}\|_\infty \le c_{\eqref{eq:hessian-g-estimate}}  (1+ \|\xi\|^2).
	\end{align}
	We also define the function $R^g_1$, which represents the remainder term in a Taylor expansion of $g_{\eta, \xi}$, by setting for $w, j, \eta, \xi, e \in \bR^D$ that
	\begin{align*}
		R^g_1(w, j; \eta, \xi; e) & := g_{\eta, \xi}(w, j+e) - g_{\eta, \xi}(w, j) - \sum_{d=1}^D e^{(d)} \pd_{D+d}g_{\eta, \xi}(w, j).
	\end{align*}
	Due to condition $(\rmG\ref{item:G-control-brownian-part})$, Taylor expansion implies for any $a\in \bR^{D^2+D}$, $a' \in \bR^{2D}$ that
	\begin{align}\label{eq:estimate-g-component-1}
		|g(a, a') -  g(0, a')| \le c_D ( c_{(\rmG\ref{item:G-control-brownian-part})} \|a\| + \|\scrD^2 g\|_\infty \|a\|^2) \le c_{\eqref{eq:estimate-g-component-1}} (\|a\| + \|a\|^2)
	\end{align}
for some constant $c_{\eqref{eq:estimate-g-component-1}}: = c_{\eqref{eq:estimate-g-component-1}}(D, \|\scrD^2 g\|_\infty,  c_{(\rmG\ref{item:G-control-brownian-part})})>0$. Hence,
	\begin{align}\label{eq:remainder-lipschitz}
		& |R^g_1(w, j; \eta, \xi;e) - R^g_1(w,j;0,\xi;e)| \notag \\
		&  \le  |g(w, \eta\pr w, j+e, \psi(j+e)\xi) -g(0, j+e, \psi(j+e)\xi)| \notag \\
		& \quad + |g(w, 0, j+e, \psi(j+e)\xi) - g(0, j+e, \psi(j+e)\xi)| \notag \\
		& \quad + |g(w, \eta \pr w, j, \psi(j)\xi) -g(0, j, \psi(j)\xi)| + |g(w, 0, j, \psi(j)\xi) -g(0, j, \psi(j)\xi)| \notag \\
		& \quad  + \sum_{d=1}^D |e^{(d)}| \B[\B| \pd_{D^2+D+d} g(w, \eta\pr w, j, \psi(j) \xi) - \pd_{D^2+D+d} g(w, 0, j, \psi(j) \xi)\B| \notag \\
		& \qquad + |\pd_d \psi(j)| \sum_{k=1}^D |\xi^{(k)}| \B| \pd_{D^2 +2D+k} g(w, \eta\pr w, j, \psi(j)\xi) - \pd_{D^2 +2D+k} g(w, 0, j, \psi(j)\xi)\B|\B] \notag \\
		& \le 2 c_{\eqref{eq:estimate-g-component-1}}(\|(w, \eta\pr w)\| + \|(w, \eta\pr w)\|^2 + \|w\| + \|w\|^2) + c_{\eqref{eq:remainder-lipschitz}} \|e\| (1+  \|\xi\|)\|\eta\pr w\|\\
		& \le 4 c_{\eqref{eq:estimate-g-component-1}}(\|w\| + \|w\|^2 + \|\eta \pr w\| + \|\eta \pr w\|^2) + c_{\eqref{eq:remainder-lipschitz}} \|e\| (1+  \|\xi\|)\|\eta\pr w\|, \notag
	\end{align}
	where  $c_{\eqref{eq:remainder-lipschitz}}: = c(D,\|\scrD \psi\|_\infty,\|\scrD^2 g\|_\infty)>0$. Moreover, the Taylor remainder $R^g_1$ is estimated by
	\begin{align}\label{eq:estimate:remainder}
		\sup_{(w, j) \in \bR^{2D}} |R^g_1(w, j; \eta, \xi;e)| & \le c_D\max_{1\le d, d' \le D} \|\pd^2_{D+d, D+d'} g_{\eta, \xi}\|_\infty \|e\|^2  \le c_{\eqref{eq:estimate:remainder}} (1+\|\xi\|^2) \|e\|^2,
	\end{align}
	where $c_{\eqref{eq:estimate:remainder}} : = c_Dc_{\eqref{eq:hessian-g-estimate}}$.

	\smallskip

\noindent	\textbf{\textit{Step 3.}}
	For $n \ge 1$ and $1\le  i\le n$, since $\eta^H_{n, i}$ is $\cF_{n, i-1} \vee \sigma\{\xi_{n, i}\}$-measurable and $(\Delta_{n, i} W$, $\Delta_{n, i} J)$ is independent of $\cF_{n, i-1} \vee \sigma\{\xi_{n, i}\}$, we get, a.s., 
\begin{align*}
	\bE[g(\Delta_{n, i} \bmZ^n) |\cF_{n, i-1}] &= \bE \B[\, \bE \B[g(\Delta_{n, i} W, \eta^H_{n, i}\pr \Delta_{n, i} W, \Delta_{n, i} J, \psi(\Delta_{n, i} J) \xi_{n, i})  \, \B| \cF_{n, i-1} \vee \sigma\{\xi_{n, i}\}\B]\,\B| \cF_{n, i-1}\B]\\
	& = \bE [G_{n, i}(\eta^H_{n, i}, \xi_{n, i})\,| \cF_{n, i-1}],
\end{align*}
where $G_{n, i}$ is a non-random and measurable function defined as 
\begin{align*}
	G_{n, i}(\eta, \xi) &: = \bE [g(\Delta_{n, i} W, \eta \pr \Delta_{n, i} W, \Delta_{n, i} J, \psi(\Delta_{n, i} J) \xi)], \quad \eta, \xi \in \bR^D.
\end{align*}
Given $\eta, \xi \in \bR^D$, applying  It\^o's formula for $g_{\eta, \xi} \in C^2(\bR^{2D})$ yields, a.s., 
\begin{align}\label{eq:ito-formula}
	& g(\Delta_{n, i} W, \eta \pr \Delta_{n, i} W, \Delta_{n, i} J, \psi(\Delta_{n, i} J)\xi)  = g_{\eta, \xi}(W_{t^n_i} - W_{t^n_{i-1}},  J_{t^n_i} - J_{t^n_{i-1}}) \notag\\
	& = \sum_{d = 1}^D \int_{t^n_{i-1}}^{t^n_i}  \pd_d g_{\eta, \xi}(W_s-W_{t^n_{i-1}}, J_{s-} - J_{t^n_{i-1}}) \od W^{(d)}_s \notag\\
	& \quad  + \frac{1}{2} \sum_{d,d' = 1}^D \int_{t^n_{i-1}}^{t^n_i} \pd_{d, d'}^2 g_{\eta, \xi}(W_s-W_{t^n_{i-1}}, J_{s-} - J_{t^n_{i-1}}) \od s \notag\\
	& \quad  + \int_{t^n_{i-1}}^{t^n_i} \int_{\RO}\B(g_{\eta, \xi}(W_s - W_{t^n_{i-1}}, J_{s-} - J_{t^n_{i-1}} + e) - g_{\eta, \xi}(W_s - W_{t^n_{i-1}}, J_{s-} - J_{t^n_{i-1}})\B) \wt N(\od e, \od s) \notag \\
	& \quad + \int_{t^n_{i-1}}^{t^n_i} \int_{\RO} R^g_1 (W_s - W_{t^n_{i-1}}, J_{s-} - J_{t^n_{i-1}}; \eta, \xi; e) \nu(\od e) \od s.
\end{align}
For $d = 1, \ldots, D$, we derive from \eqref{eq:partial-derivative-1} that $(w, j) \mapsto \pd_d g_{\eta, \xi}(w, j)$ has at most linear growth at infinity which hence implies that the stochastic integrals with respect to the Brownian motions are square integrable martingales. Moreover, for $w, j, j' \in \bR^D$, due to \eqref{eq:partial-derivative-2} and $(\rmG\ref{item:G-jump-part})$ one has \begin{align*}
|g_{\eta, \xi}(w, j) - g_{\eta, \xi}(w, j')| \le c_D \max_{1 \le d \le D} \|\pd_d g_{\eta, \xi}(w, \cdot)\|_\infty \|j - j'\|	\le c_D c_{(\rmG\ref{item:G-jump-part})} (1+ \|\scrD \psi\|_\infty \|\xi\|) \|j - j'\|.
\end{align*}
Then, due to the assumption $\int_{\RO} \|e\|^2 \nu(\od e) <\infty$, the stochastic integral with respect to the compensated Poisson random measure $\wt N$ in \eqref{eq:ito-formula} is also a square integrable martingale which then vanishes after taking the expectation. Hence,
\begin{align*}
G_{n, i}(\eta, \xi) =  G^W_{n, i}(\eta, \xi) + G^{J}_{n, i}(\eta, \xi),
\end{align*}
where the integrability condition is satisfied so that Fubini's theorem enables us to define 
\begin{align*}
	G^W_{n, i}(\eta, \xi) & : = \frac{1}{2} \sum_{d,d' = 1}^D \int_{t^n_{i-1}}^{t^n_i} \bE \B[\pd_{d, d'}^2 g_{\eta, \xi}(W_s-W_{t^n_{i-1}}, J_{s-} - J_{t^n_{i-1}})\B] \od s,\\
	G^J_{n, i}(\eta, \xi) &:= \int_{t^n_{i-1}}^{t^n_i} \int_{\RO} \bE \B[R^g_1(W_s -  W_{t^n_{i-1}}, J_{s-} -  J_{t^n_{i-1}} ; \eta, \xi; e)\B] \nu(\od e) \od s.
\end{align*}
To derive \eqref{eq:lemm:convergence-jump-part} it suffices to prove that the following three convergences hold:
\begin{align}
	G^n_{\eqref{eq:convergence-L1-GW}} & := \sum_{i =1}^{n}	\bE[|G^W_{n, i}(\eta^H_{n, i}, \xi_{n, i})|] \to 0, \label{eq:convergence-L1-GW}\\
	G^n_{\eqref{eq:convergence-L1-GJ-part-1}} & :=\sum_{i =1}^{n}	\bE [|G^J_{n, i}(\eta^H_{n, i}, \xi_{n, i}) - G^J_{n, i}(0, \xi_{n, i})|] \to 0,\label{eq:convergence-L1-GJ-part-1}\\
	G^{n}_{\eqref{eq:convergence-L1-GJ-part-2}} & :=\sum_{i =1}^{n}	\bE \bb[\bb|G^J_{n, i}(0, \xi_{n, i}) - (t^n_i - t^n_{i-1})\int_{\RO \times \bR^D} g(0, e, \psi(e)u) \nu(\od e) \varphi_D(u) \od u  \bb|\bb] \to 0. \label{eq:convergence-L1-GJ-part-2}
\end{align}

\smallskip

\noindent\textbf{\textit{Step 4.}} We show $G^n_{\eqref{eq:convergence-L1-GW}} \to 0$. For $1\le d, d' \le D$, by \eqref{eq:partial-derivative-1} one has
\begin{align*}
	G^W_{n, i}(\eta, \xi) & = \frac{1}{2} \sum_{d, d' = 1}^D \sum_{k, l =0}^D \eta^{(k)} \eta^{(l)} \int_{t^n_{i-1}}^{t^n_i}  \bE\B[ \pd^2_{d + kD, d'+ lD} g(W_{s} - W_{t^n_{i-1}}, \eta \pr (W_{s} -  W_{t^n_{i-1}}), \notag\\
	& \hspace{230pt}  J_{s-} - 
	J_{t^n_{i-1}}, \psi(J_{s-} - 
	J_{t^n_{i-1}})\xi)\B] \od s.
\end{align*}
Let $(\bar W, \bar J)$ be an independent copy of $(W, J)$ with the corresponding expectation $\bar \bE$. Applying Fubini's theorem we get
\begin{align}\label{eq:convergence-L1-GW-estimate}
	G^n_{\eqref{eq:convergence-L1-GW}} & \le \frac{1}{2} \sum_{d, d' = 1}^D \sum_{k, l = 0}^D \sum_{i =1}^{n}  \bE\bb[|\eta^{H, (k)}_{n, i} \eta^{H, (l)}_{n, i}| \int_{t^n_{i-1}}^{t^n_i}  \bar\bE\B[\B| \pd^2_{d + kD, d'+ lD} g(\bar W_{s} - \bar W_{t^n_{i-1}}, \eta^H_{n, i} \pr (\bar W_{s} - \bar W_{t^n_{i-1}}), \notag\\
	& \hspace{240pt} \bar J_{s-} - 
	\bar J_{t^n_{i-1}}, \psi(\bar J_{s-} - 
	\bar J_{t^n_{i-1}})\xi_{n, i})\B|\B] \od s\bb] \notag \\
	& = \frac{1}{2} \sum_{d, d' = 1}^D \sum_{k, l =0}^D \int_0^{T}   \bE \bb[ \sum_{i =1}^{n} |\eta^{H, (k)}_{n, i} \eta^{H, (l)}_{n, i}| \B| \pd^2_{d + kD, d'+ lD} g(W_{s} - W_{t^n_{i-1}}, \eta^H_{n, i} \pr (W_{s} - W_{t^n_{i-1}}), \notag\\
	& \hspace{200pt} J_{s-} - 
	J_{t^n_{i-1}}, \psi(J_{s-} - 
	J_{t^n_{i-1}})\xi_{n, i})\B| \1_{(t^n_{i-1}, t^n_i]}(s) \bb] \od s \notag\\
	& =: \frac{1}{2} \sum_{d, d' = 1}^D \sum_{k, l =0}^D \int_0^{T} \bE \b[G^n_{\eqref{eq:convergence-L1-GW-estimate}}(s)\b] \od s.
\end{align}
In order to derive \eqref{eq:convergence-L1-GW}, we prove  for any $1 \le d, d' \le D$, $0\le k, l \le D$ that
\begin{align*}
	\int_0^{T} \bE \b[G^n_{\eqref{eq:convergence-L1-GW-estimate}}(s)\b] \od s \to 0 \quad \mbox{as } n \to \infty.
\end{align*}
By the dominated convergence theorem, it is sufficient to show that
\begin{align}\label{eq:convergence-L1-GW-2-conditions}
	\lim_{n \to \infty} \bE[G^n_{\eqref{eq:convergence-L1-GW-estimate}}(s)] = 0 \quad \mbox{for all } s \in (0, T), \quad \mbox{and} \quad \int_0^{T} \sup_{n \ge 1} \bE[G^n_{\eqref{eq:convergence-L1-GW-estimate}}(s)] \od s <\infty.
\end{align}
Indeed, for each fixed $s \in (0, T)$ one has 
\begin{align*}
	\eta^H_{n, i} \pr (W_s - W_{t^n_{i-1}}) \xrightarrow{\bfL_2(\bP)} 0 \quad \mbox{and} \quad \psi(J_{s-} - J_{t^n_{i-1}}) \xi_{n, i} \xrightarrow{\bfL_2(\bP)} 0
\end{align*}
when $n \to \infty$ because of the independence, $t^n_{i-1} \to s$, and
\begin{align*}
	& \bE[\|\eta^H_{n, i} \pr (W_s - W_{t^n_{i-1}})\|^2] = \bE[\|\eta^H_{n, i}\|^2] \, \bE[\|W_s - W_{t^n_{i-1}}\|^2] = D^2 (s - t^n_{i-1}),\\
	&  \bE[\|\psi(J_{s-} - J_{t^n_{i-1}}) \xi_{n, i}\|^2] \le D \|\scrD \psi\|_\infty^2 \bE[\|J_{s-} - J_{t^n_{i-1}}\|^2] = (s - t^n_{i-1}) D \|\scrD \psi\|_\infty^2  \int_{\RO}\|e\|^2 \nu(\od e).
\end{align*}
Since $\pd^2_{d+kD, d'+lD} g$ is continuous and is equal to 0 in a neighborhood of $0$ by $(\rmG\ref{item:G-brownian-part})$, we get
\begin{align*}
	G^n_{\eqref{eq:convergence-L1-GW-estimate}}(s) \xrightarrow{\bP} 0 \quad \mbox{as } n \to \infty,
\end{align*}
where the convergence in probability can be asserted by showing that any subsequence has a further subsequence converging a.s. to $0$. Moreover, since $g$ has bounded second-order partial derivatives by $(\rmG\ref{item:G-hessian})$ and $\{\|\eta^H_{n, i}\|^2\}_{1\le i \le n, n\ge 1}$ is uniformly integrable by \cref{assumption-eta}, it implies that $\{G^n_{\eqref{eq:convergence-L1-GW-estimate}}(s)\}_{n \ge 1}$ is also uniformly integrable. Hence, the dominated convergence theorem is applicable to obtain the first assertion in \eqref{eq:convergence-L1-GW-2-conditions}. The integrability condition in \eqref{eq:convergence-L1-GW-2-conditions} is easily verified by noting that
\begin{align*}
	\sup_{n \ge 1} \bE[ G^n_{\eqref{eq:convergence-L1-GW-estimate}}(s)] & \le \|\scrD^2 g\|_\infty \sup_{1\le i \le n, n\ge 1} \bE[|\eta^{H, (k)}_{n, i} \eta^{H, (l)}_{n, i}|] \\
	& \le \frac{1}{2}\|\scrD^2 g\|_\infty \sup_{1\le i \le n, n\ge 1} \bE[|\eta^{H, (k)}_{n, i}|^2 + |\eta^{H, (l)}_{n, i}|^2] = \|\scrD^2 g\|_\infty.
\end{align*} Hence, \eqref{eq:convergence-L1-GW} is proved.

\smallskip

\noindent\textbf{\textit{Step 5.}} We prove $G^n_{\eqref{eq:convergence-L1-GJ-part-1}}\to 0$. By the independence and Fubini's theorem we obtain
\begin{align}\label{eq:convergence-L1-GJ-estimate-1}
	 G^n_{\eqref{eq:convergence-L1-GJ-part-1}} & \le \int_{\RO} \int_0^{T} \bE\bb[  \sum_{i=1}^{n} \B| R^g_1( W_s -  W_{t^n_{i-1}}, J_{s-} - J_{t^n_{i-1}}; \eta^H_{n, i}, \xi_{n, i}; e) \notag \\
	& \hspace{130pt} - R^g_1(W_s - W_{t^n_{i-1}}, J_{s-} - J_{t^n_{i-1}}; 0, \xi_{n, i}; e)\B| \1_{(t^n_{i-1}, t^n_i]}(s) \bb] \od s \nu(\od e) \notag \\
	& =: \int_{\RO} \int_0^{T} \bE \b[ G^n_{\eqref{eq:convergence-L1-GJ-estimate-1}}(s; e)\b] \od s \nu(\od e).
\end{align}
By dominated convergence, it suffices to show that 
\begin{align}
	&\forall (s, e) \in (0, T) \times \RO: \lim_{n \to \infty} \bE[ G^n_{\eqref{eq:convergence-L1-GJ-estimate-1}}(s; e)] = 0, \label{eq:verify-DCT-GJ-1}\\
	&\mbox{and} \quad \int_{\RO}\int_0^{T} \sup_{n \ge 1} \bE[ G^n_{\eqref{eq:convergence-L1-GJ-estimate-1}}(s; e)] \od s \nu(\od e) <\infty. \label{eq:verify-DCT-GJ-2}
\end{align}
Indeed, for each $(s, e) \in (0, T) \times \RO$, using \eqref{eq:remainder-lipschitz} yields
\begin{align*}
	G^n_{\eqref{eq:convergence-L1-GJ-estimate-1}}(s; e) & \le \sum_{i=1}^{n} \B[4 c_{\eqref{eq:estimate-g-component-1}} \B(\|W_s - W_{t^n_{i-1}}\| + \|W_s - W_{t^n_{i-1}}\|^2 + \|\eta^H_{n, i} \pr (W_s - W_{t^n_{i-1}})\| \\
	& \quad + \|\eta^H_{n, i} \pr (W_s - W_{t^n_{i-1}})\|^2\B) + c_{\eqref{eq:remainder-lipschitz}} \|e\|(1+ \|\xi_{n, i}\|)\|\eta^H_{n, i} \pr (W_s - W_{t^n_{i-1}})\|\B] \1_{(t^n_{i-1}, t^n_i]}(s).
\end{align*}
Then, by H\"older's inequality we get
\begin{align*}
	\bE[G^n_{\eqref{eq:convergence-L1-GJ-estimate-1}}(s; e)] & \le  \sum_{i=1}^{n} \B[4 c_{\eqref{eq:estimate-g-component-1}} \B(\sqrt{D}\sqrt{s - t^n_{i-1}} + D(s - t^n_{i-1}) + D\sqrt{s - t^n_{i-1}} + D^2(s-t^n_{i-1}) \B) \\
	& \qquad + c_{\eqref{eq:remainder-lipschitz}}\|e\|\sqrt{\bE[|1+\|\xi_{n, i}\||^2]} \, \sqrt{\bE[\|\eta^H_{n, i} \pr (W_s - W_{t^n_{i-1}})\|^2]}\,\B] \1_{(t^n_{i-1}, t^n_i]}(s)\\
	& \to 0 \quad \mbox{as } n\to \infty,
\end{align*}
which then verifies \eqref{eq:verify-DCT-GJ-1}. To show \eqref{eq:verify-DCT-GJ-2}, we use the estimate \eqref{eq:estimate:remainder} to get
\begin{align*}
	\sup_{n \ge 1} \bE[G^n_{\eqref{eq:convergence-L1-GJ-estimate-1}}(s; e)] 	& \le 2 c_{\eqref{eq:estimate:remainder}}  \sup_{n\ge 1, 1\le i \le n} \bE[ (1+\|\xi_{n, i}\|^2)\|e\|^2] = 2 c_{\eqref{eq:estimate:remainder}}  (D+1)\|e\|^2.
\end{align*}
Since $\int_{\RO}\|e\|^2\nu(\od e) <\infty$ by assumption, \eqref{eq:verify-DCT-GJ-2} follows.

\smallskip

\noindent\textbf{\textit{Step 6.}} We show $G^{n}_{\eqref{eq:convergence-L1-GJ-part-2}} \to 0$. By the independence and Fubini's theorem one has
\begin{align}\label{eq:estimate-GJ-part-2}
	& G^{n}_{\eqref{eq:convergence-L1-GJ-part-2}} \notag\\
	&  \le \sum_{i=1}^{n} \int_{\RO} \int_{t^n_{i-1}}^{t^n_i} \bE \bb[ \bb|R^g_1(W_s - W_{t^n_{i-1}}, J_{s-} - J_{t^n_{i-1}}; 0, \xi_{n, i}; e) - \int_{\bR^D} g(0, e, \psi(e)u) \varphi_D(u) \od u\bb|\bb]\od s \nu(\od e) \notag \\
	& \le \int_{\bR^D} \int_{\RO} \int_0^{T} \bE\bb[\sum_{i=1}^{n} \bb|R^g_1(W_s - W_{t^n_{i-1}}, J_{s-} - J_{t^n_{i-1}}; 0, u; e) -  g(0, e, \psi(e)u) \bb| \1_{(t^n_{i-1}, t^n_i]}(s) \bb]  \notag \\
	& \hspace{360pt} \times \od s \nu(\od e) \varphi_D(u) \od u \notag \\
	& =: \int_{\bR^D} \int_{\RO} \int_0^{T} \bE\B[G^n_{\eqref{eq:estimate-GJ-part-2}}(s; e,u)\B] \od s \nu(\od e) \varphi_D(u) \od u.
\end{align}
For any $(s, e, u) \in (0, T] \times \RO \times \bR^D$, since the first two arguments in $R^g_1$ converge to $0$ a.s. as $n \to \infty$, we obtain that $G^n_{\eqref{eq:estimate-GJ-part-2}}(s;e,u) \to 0$ a.s. Moreover, one has
\begin{align*}
	 \bE\bb[\sup_{n \ge 1} |G^n_{\eqref{eq:estimate-GJ-part-2}}(s;e,u)|\bb] & \le \bE\bb[ \sup_{n \ge 1, 1\le i \le n}|R^g_1(W_s - W_{t^n_{i-1}}, J_{s-} - J_{t^n_{i-1}}; 0, u; e)|\bb] + |g(0, e, \psi(e)u)|\\
	 & \le c_{\eqref{eq:estimate:remainder}}  (1+\|u\|^2)\|e\|^2 + |g(0, e, \psi(e)u)|.
\end{align*}
Since, 
by \eqref{eq:estimate-g},
\begin{align*}
	\int_{\bR^D} \int_{\RO} \B((1+\|u\|^2)\|e\|^2 + |g(0, e, \psi(e)u)|\B) \nu(\od e) \varphi_D(u)\od u < \infty,
\end{align*}
the dominated convergence theorem implies that $G^{n}_{\eqref{eq:convergence-L1-GJ-part-2}} \to 0$ as $n \to \infty$.
\end{proof}

We first deal with the jump part of the  limit of $(\bmZ^n)_{n \ge 1}$. To do this, we recall from \cite[p.395]{JS03} the function space $C_2(\bR^\bfD)$, which consists of all  continuous bounded functions $g \colon \bR^\bfD \to \bR$ with $0 \notin \supp(g)$.

\begin{lemm}\label{limit-jump-part}
	The assertion \eqref{eq:lemm:convergence-jump-part} holds true for $g \in C_2(\bR^{\bfD})$. Consequently, for any $t \in [0, \infty)$ one has when $n \to \infty$ that 
	\begin{align*}
		\sum_{i=1}^{\sigma^n_t} \bE[g(\Delta_{n, i} \bmZ^n)|\cF_{n, i-1}] \xrightarrow{\bfL_1(\bP)} (t\wedge T) \int_{\bR^{2D}} g(0, e, u) \nu_L^\psi(\od e, \od u).
	\end{align*}
\end{lemm}

\begin{proof} It suffices to show the convergence for $t \in [0, T]$.
	Let $g \in C_2(\bR^{\bfD})$ and assume that $\supp(g) \cap B_{\bfD}(r_g) = \emptyset$ for some $r_g >0$. Let $\ep>0$ be arbitrarily small and $K >r_g$ a sufficiently large constant which is specified later. Since $g$ is continuous and bounded, there is a continuous function $g_{K}$ with compact support such that $\|g_{K}\|_\infty \le \|g\|_\infty$ and $g_{K} = g$ on $B_{\bfD}(K)$. Moreover, by convolution approximation, there is a function $\hat g_{\ep, K} \in C_2(\bR^{\bfD}) \cap C_c^2(\bR^{\bfD})$ such that $\supp(g_{K} - \hat g_{\ep, K}) \cap B_{\bfD}(r_g/2) = \emptyset$ and $\|g_{K} - \hat g_{\ep, K}\|_\infty \le \ep$. For $t \in (0, T]$, we denote
	\begin{align}\label{eq:jump-part-difference}
		I^g_{\eqref{eq:jump-part-difference}} : = \sum_{i=1}^{\sigma_t^n} \bb| \bE[g(\Delta_{n, i} \bmZ^n)|\cF_{n, i-1}] - (t^n_i - t^n_{i-1}) \int_{\bR^{2D}} g(0, e, u) \nu^\psi_{L} (\od e, \od u)\bb|
	\end{align}
	and then get by the triangle inequality that
	\begin{align*}
		I^g_{\eqref{eq:jump-part-difference}} \le I^{g - g_{K}}_{\eqref{eq:jump-part-difference}} + I^{g_{K} - \hat g_{\ep, K}}_{\eqref{eq:jump-part-difference}} + I^{\hat g_{\ep, K}}_{\eqref{eq:jump-part-difference}}.
	\end{align*}
	Since $\hat g_{\ep, K} \in C_2(\bR^{\bfD}) \cap C_c^2(\bR^{\bfD}) \subset C^2_*(\bR^{\bfD})$, according to  \cref{lemm:convergence-jump-part} one has
	\begin{align*}
		I^{\hat g_{\ep, K}}_{\eqref{eq:jump-part-difference}} \xrightarrow{\bfL_1(\bP)} 0.
	\end{align*}
	For the stochastic term in $I^{g - g_{K}}_{\eqref{eq:jump-part-difference}}$, we have, a.s.,
	\begin{align*}
		&\sum_{i=1}^{\sigma_t^n} \bE[|(g - g_{K})(\Delta_{n, i} \bmZ^n)|\,|\cF_{n, i-1}]  \le \|g - g_{K}\|_\infty \sum_{i=1}^n \bE[\1_{\{\|\Delta_{n, i} \bmZ^n\| \ge K\}}|\cF_{n, i-1}]\\
		&  \le \frac{2\|g\|_\infty}{K^2} \sum_{i=1}^n \bE[\|\Delta_{n, i} \bmZ^n\|^2 |\cF_{n, i-1}] \\
		& =  \frac{2\|g\|_\infty}{K^2} \sum_{i=1}^n \bE \B[\|\Delta_{n, i} W\|^2 + \|\eta^H_{n, i} \pr \Delta_{n, i} W\|^2 + \|\Delta_{n, i} J\|^2 + \psi(\Delta_{n, i} J)^2 \|\xi_{n, i}\|^2 \,\B|\cF_{n, i-1}\B] \\
		& \le \frac{2\|g\|_\infty}{K^2} \sum_{i=1}^n \bb[(t^n_i - t^n_{i-1})(D+ D^2) + (1+ D\|\scrD\psi\|_\infty^2 )(t^n_i - t^n_{i-1}) \int_{\RO}\|e\|^2 \nu(\od e)\bb]\\
		& = \frac{2T\|g\|_\infty}{K^2}\bb[D + D^2 + (1+ D\|\scrD\psi\|_\infty^2) \int_{\RO}\|e\|^2 \nu(\od e)\bb].
	\end{align*}
	For the stochastic term in $I^{g_{K} - \hat g_{\ep, K}}_{\eqref{eq:jump-part-difference}}$, we use the same arguments as for $I^{g - g_{K}}_{\eqref{eq:jump-part-difference}}$  to obtain, a.s., 
	\begin{align*}
		\sum_{i=1}^{\sigma_t^n} \bE[|(g_{K} - \hat g_{\ep, K})(\Delta_{n, i} \bmZ^n)|\,|\cF_{n, i-1}]  & \le \|g_{K} - \hat g_{\ep, K}\|_\infty \sum_{i=1}^n \bE[\1_{\{\|\Delta_{n, i} \bmZ^n\| \ge r_g/2\}}|\cF_{n, i-1}] \\
		& \le \frac{4T \ep}{r_g^2} \bb[D + D^2 + (1+ D\|\scrD\psi\|_\infty^2) \int_{\RO}\|e\|^2 \nu(\od e)\bb].
	\end{align*}
	Then, by the triangle inequality,
	\begin{align*}
		I^{g - g_{K}}_{\eqref{eq:jump-part-difference}} \le \frac{2T\|g\|_\infty}{K^2}\bb[D + D^2 + (1+ D\|\scrD\psi\|_\infty^2) \int_{\RO}\|e\|^2 \nu(\od e)\bb] + 2T\|g\|_\infty \int_{B^c_{2D}(K)} \nu_{L}^\psi(\od e, \od u)
	\end{align*}
	which can be made arbitrarily small as long as we choose a sufficiently large $K>0$. Analogously,
	\begin{align*}
		I^{g_{K} - \hat g_{\ep, K}}_{\eqref{eq:jump-part-difference}} \le	\ep\bb[\frac{4T}{r_g^2} \bb(D + D^2 + (1+ D\|\scrD\psi\|_\infty^2) \int_{\RO}\|e\|^2 \nu(\od e)\bb) + T \int_{B^c_{2D}(r_g/2) } \nu^\psi_{L}(\od e, \od u)\bb].
	\end{align*}
	Eventually, since $\ep>0$ is arbitrarily small, it implies that $I^g_{\eqref{eq:jump-part-difference}} \xrightarrow{\bfL_1(\bP)} 0$.
\end{proof}

We continue to investigate the continuous and the drift components of the limit of $(\bmZ^n)_{n \ge 1}$. To this end, let us fix a truncation function $h \colon \bR^{\bfD} \to \bR^{\bfD}$ in the sense of \cite[Definition II.2.3]{JS03}, i.e. $h$ is bounded and $h(z) = z$ in a neighborhood of 0. As we will see later that the limit of $(\bmZ^n)_{n \ge 1}$ does not depend on the particular form of truncation function, we assume that $h = (h^{(d)})_{d=1}^{\bfD}$ with $h^{(d)} \in C^2_b(\bR^{\bfD})$.

\begin{lemm}\label{lemm:convergence-drift-part}
	For any $t \in [0, \infty)$, one has when $n \to \infty$ that
	\begin{align*}
		\sup_{s \le t} \bb\| \sum_{i=1}^{\sigma^n_s} \bE[h(\Delta_{n, i} \bmZ^n)|\cF_{n,i-1}] - B_s\bb\| \xrightarrow{\bfL_1(\bP)} 0,
	\end{align*}
where $B: = B(h)$ given by
\begin{align*}
	B_t: =  (t \wedge T) \int_{\bR^{2D}}(h(0, e, u) - (0, e, u)^\tran) \nu^\psi_{L}(\od e, \od u).
\end{align*}
\end{lemm}

\begin{proof}
It is sufficient to consider $t \in [0, T]$ and prove that for any $d = 1, \ldots, \bfD$ one has
\begin{align}\label{drift-term-convergence}
I^{(d)}_{\eqref{drift-term-convergence}}:=	\sup_{s \le t} \bb| \sum_{i=1}^{\sigma^n_s} \bE[h^{(d)}(\Delta_{n, i} \bmZ^n)|\cF_{n, i-1}] - s B_1^{(d)}\bb| \xrightarrow{\bfL_1(\bP)} 0,
\end{align}
Let $\tilde h^{(d)}(z) : = h^{(d)}(z)  - z^{(d)}$ for $z = (z^{(1)}, \ldots, z^{(\bfD)})\in \bR^{\bfD}$. It follows from  $\bE[\Delta_{n, i} \bmZ^n|\cF_{n, i-1}] = 0$ a.s. that
\begin{align}\label{eq:equality-h-tilde}
	\bE[h^{(d)}(\Delta_{n, i} \bmZ^n)|\cF_{n, i-1}] = \bE[\tilde h^{(d)}(\Delta_{n, i} \bmZ^n)|\cF_{n, i-1}]\quad \mbox{a.s.}
\end{align}
Hence we now prove \eqref{drift-term-convergence} for $\tilde h^{(d)}$ in place of $h^{(d)}$. We remark that there is no problem regarding $\bP$-null sets for that replacement as only countably many random variables are considered in \eqref{drift-term-convergence}. On the other hand, since $h^{(d)} \in C^2_b(\bR^{\bfD})$ and $h^{(d)}(z) = z^{(d)}$ in a neighborhood of $0$, it is straightforward to check that $\tilde h^{(d)} \in C^2_*(\bR^{\bfD})$. By the triangle inequality, a.s.,
\begin{align*}
	I^{(d)}_{\eqref{drift-term-convergence}} & \le \sup_{s \le t} \bb| \sum_{i=1}^{\sigma^n_s} \bE[\tilde h^{(d)}(\Delta_{n, i} \bmZ^n)|\cF_{n, i-1}] - t^n_{\sigma^n_s} B_{1}^{(d)}\bb| + \sup_{s \le t} \B| t^n_{\sigma^n_s} B_{1}^{(d)} - s B^{(d)}_1 \B|\\
	& \le \sum_{i=1}^{n} \B| \bE[\tilde h^{(d)}(\Delta_{n, i} \bmZ^n) |\cF_{n, i-1}] - (t^n_i - t^n_{i-1})B_1^{(d)} \B|  + \max_{1 \le i \le n} (t^n_i - t^n_{i-1})|B^{(d)}_1|.
\end{align*}
According to \cref{lemm:convergence-jump-part}, the first term on the right-hand side converges to $0$ in $\bfL_1(\bP)$. The second term $\max_{1 \le i \le n} (t^n_i - t^n_{i-1})|B^{(d)}_1|$ obviously tends to $0$ as $n \to \infty$. Hence, \eqref{drift-term-convergence} follows.
\end{proof}

We now investigate the continuous part of the limit of $(\bmZ^n)_{n \ge 1}$. For $t \in [0, \infty)$, we define the matrices	$\bmC_t = (C_t^{(k, l)}) \in \bR^{\bfD} \times \bR^{\bfD}$ and its modification $\wt \bmC_t = (\wt C_t^{(k, l)}) \in \bR^{\bfD} \times \bR^{\bfD}$ by
\begin{align}\label{eq:definition-coeff-matrix}
	C_t^{(k,l)} := \begin{cases}
	t \wedge T & \mbox{if } 1\le k = l \le D^2 +D\\
	0 & \mbox{otherwise,}
\end{cases}
\end{align}
and
\begin{align*}
	\wt C_t^{(k,l)}: = C_t^{(k,l)} + (t \wedge T) \int_{\bR^{2D}} (h^{(k)} h^{(l)})(0, e, u) \nu^\psi_{L}(\od e, \od u).
\end{align*}

\begin{lemm}\label{lemm:convergence-continuous-part} For any $t \in [0, \infty)$ and $1 \le k, l \le \bfD$, one has when $n \to \infty$ that
	\begin{align}
		I_{\eqref{mixture-term-convergence-1}}&:= \sum_{i=1}^{\sigma^n_t} \bE[h^{(k)}(\Delta_{n, i} \bmZ^n)|\cF_{n, i-1}]\, \bE[h^{(l)}(\Delta_{n, i} \bmZ^n)|\cF_{n,i-1}] \xrightarrow{\bfL_1(\bP)} 0, \label{mixture-term-convergence-1}\\
		I_{\eqref{mixture-term-convergence-2}}&:=\sum_{i=1}^{\sigma^n_t} \bE[(h^{(k)} h^{(l)})(\Delta_{n, i} \bmZ^n)|\cF_{n, i-1}] \xrightarrow{\bfL_1(\bP)} \wt C_{t}^{(k,l)}. \label{mixture-term-convergence-2}
 	\end{align}
\end{lemm}

\begin{proof} It suffices to prove for $t \in [0, T]$. We first show that $I_{\eqref{mixture-term-convergence-1}} \xrightarrow{\bfL_1(\bP)} 0$ as $n \to \infty$.
In the sequel we employ the notation as in the proof of \cref{lemm:convergence-drift-part}. According to \eqref{eq:equality-h-tilde} one has, a.s., 
\begin{align*}
	I_{\eqref{mixture-term-convergence-1}} & = \sum_{i=1}^{\sigma^n_t} \bE\B[\tilde h^{(k)}(\Delta_{n, i} \bmZ^n) - (t^n_i - t^n_{i-1}) B_1^{(k)}\B|\cF_{n, i-1}\B]  \bE[h^{(l)}(\Delta_{n, i} \bmZ^n)|\cF_{n,i-1}]\\
	& \quad  +  B_1^{(k)} \sum_{i=1}^{\sigma^n_t}  (t^n_i - t^n_{i-1}) \bE\B[\tilde h^{(l)}(\Delta_{n, i} \bmZ^n) - (t^n_i - t^n_{i-1}) B_1^{(l)}\B|\cF_{n,i-1}\B] + B_1^{(k)} B_1^{(l)} \sum_{i=1}^{\sigma^n_t}(t^n_i- t^n_{i-1})^2.
\end{align*}
Then, a.s.,
\begin{align*}
	|I_{\eqref{mixture-term-convergence-1}}| & \le \|h^{(l)}\|_\infty \sum_{i=1}^{n} \B|\bE\B[\tilde h^{(k)}(\Delta_{n, i} \bmZ^n) - (t^n_i - t^n_{i-1}) B_1^{(k)}\B|\cF_{n, i-1}\B]\B|\\
	& \quad + |B_1^{(k)}|\max_{1\le i \le n}(t^n_i - t^n_{i-1})  \sum_{i=1}^{n} \B| \bE\B[\tilde h^{(l)}(\Delta_{n, i} \bmZ^n) - (t^n_i - t^n_{i-1}) B_1^{(l)}\B|\cF_{n,i-1}\B]\B|\\
	& \quad + t |B_1^{(k)} B_1^{(l)}| \max_{1\le i \le n}(t^n_i - t^n_{i-1}).
\end{align*}
Since $\max_{1\le i \le n}(t^n_i - t^n_{i-1}) \to 0$, applying \cref{lemm:convergence-jump-part} yields \eqref{mixture-term-convergence-1}.

We next show that $I_{\eqref{mixture-term-convergence-2}} \to \wt C^{(k, l)}_t$ in $\bfL_1(\bP)$. For $z = (z^{(1)}, \ldots, z^{(\bfD)}) \in \bR^{\bfD}$, we define
$$q^{(k, l)}(z): = z^{(k)} z^{(l)} \quad \mbox{and} \quad  \hat h^{(k, l)}(z) : = \begin{cases}
	 (h^{(k)}h^{(l)})(z) - q^{(k, l)}(z) & \mbox{if } 1 \le k \vee l \le D^2 +D\\
	 (h^{(k)}h^{(l)})(z) & \mbox{otherwise.}
	 \end{cases}$$
We now verify that $\hat h^{(k, l)} \in C^2_*(\bR^{\bfD})$ for any $k, l = 1, \ldots, \bfD$:
\begin{itemize}
	\itemsep3pt
	\item $\hat h^{(k, l)}$ obviously satisfies $(\rmG\ref{item:G-hessian})$.
	
	\item Let $1 \le d\vee d' \le D^2 +D$. If $k \vee l \le D^2 +D$, then $\hat h^{(k, l)}$, and thus $\pd^2_{d, d'} \hat h^{(k, l)}$, are $0$ in a neighborhood of $0$. If $k \vee l \ge D^2 +D+1$, then $\pd^2_{d, d'} \hat h^{(k, l)} = \pd^2_{d, d'}(h^{(k)} h^{(l)} - q^{(k, l)})$, which also shows that $\pd^2_{d, d'} \hat h^{(k, l)}$ is $0$ around $0$. Hence, $(\rmG\ref{item:G-brownian-part})$ is satisfied.
	
	\item For $d = 1, \ldots, D^2+D$ and for any $j \in \bR^{2D}$, one has 
	$$\pd_d \hat h^{(k, l)}(0, j) = \begin{cases}
		\pd_d(h^{(k)}h^{(l)})(0, j) - \pd_d q^{(k, l)}(0, j) & \mbox{if } 1 \le k \vee l \le D^2 +D\\
		\pd_d(h^{(k)}h^{(l)})(0, j) & \mbox{otherwise}
	\end{cases} = \pd_d(h^{(k)}h^{(l)})(0, j).$$ Hence, $\max_{1 \le d \le D^2 +D} \|\pd_d \hat h^{(k, l)}(0_{D^2 +D}, \cdot)\|_\infty \le \|\scrD(h^{(k)}h^{(l)})\|_\infty <\infty$, which verifies $(\rmG\ref{item:G-control-brownian-part})$.
	
	\item For $d = D^2+D+1, \ldots, \bfD$, since $\pd_d q^{(k, l)} = 0$ if $k \vee l \le D^2 +D$ we infer that $\pd_d \hat h^{(k, l)} = \pd_d(h^{(k)}h^{(l)})$ and $\pd_d \hat h^{(k, l)}(0) = h^{(l)}(0) \pd_d h^{(k)}(0)  + h^{(k)}(0) \pd_d h^{(l)}(0) = 0$. Thus, $(\rmG\ref{item:G-jump-part})$ is satisfied.
\end{itemize}
Applying \cref{lemm:convergence-jump-part} and noticing that, for any $1\le k, l \le \bfD$,
$$\int_{\bR^{2D}} \hat h^{(k, l)}(0, e, u) \nu_L^{\psi}(\od e, \od u) = \int_{\bR^{2D}} (h^{(k)} h^{(l)})(0, e, u) \nu_L^{\psi}(\od e, \od u)$$
 we obtain 
\begin{align}\label{eq:mixture-term-convergence-2-limit}
	\sum_{i=1}^{\sigma^n_t} \bb| \bE[\hat h^{(k, l)}(\Delta_{n, i} \bmZ^n)|\cF_{n, i-1}] - (t^n_i - t^n_{i-1})\int_{\bR^{2D}} (h^{(k)}h^{(l)})(0, e, u) \nu_L^{\psi}(\od e, \od u)\bb| \xrightarrow{\bfL_1(\bP)} 0.
\end{align}
On the other hand, for $1\le k\vee l \le D^2 +D$, a direct calculation exploiting the independence and \eqref{property-eta} gives the following convergence as $n \to \infty$, particularly in $\bfL_1(\bP)$,
\begin{align*}
	\sum_{i=1}^{\sigma^n_t} \bE[q^{(k, l)}(\Delta_{n, i} \bmZ^n) |\cF_{n, i-1}] & = \sum_{i=1}^{\sigma^n_t} \bE[\Delta_{n, i} \bmZ^{n,(k)} \Delta_{n, i} \bmZ^{n,(l)} |\cF_{n, i-1}] \\
	& = \begin{cases}
		t^n_{\sigma^n_t} & \mbox{if } 1\le k = l \le D^2 +D\\
		0 & \mbox{otherwise}
	\end{cases} \to t C^{(k, l)}_1.
\end{align*}
Therefore, \eqref{mixture-term-convergence-2} follows from \eqref{eq:mixture-term-convergence-2-limit}, and the proof is completed.
\end{proof}

\subsection*{Proof of \Cref{thm:limit-distribution-integrator}} We combine \cite[Theorem VIII.2.29]{JS03} with \cref{limit-jump-part,lemm:convergence-drift-part,lemm:convergence-continuous-part} to obtain that
$(\bmZ^n_{t \wedge T})_{t \in [0, \infty)} \to \bmZ$ weakly in the Skorokhod topology on the space of c\`adl\`ag functions $ \colon [0, \infty) \to \bR^{\bfD}$. Here, $\bmZ$ is a semimartingale with the predictable characteristic\footnote{in the sense of \cite[Definition II.2.6]{JS03}.} $(B, C, m_{\bmZ})$ associated with the truncation function $h$, where
	\begin{itemize}
		\item $\bmZ_0 = 0$ as $\bmZ^n_0 = 0$ for all $n$;
		
		\item $h$ is taken as in the paragraph right before \cref{lemm:convergence-drift-part};
		
		\item $B$ is provided in \cref{lemm:convergence-drift-part};
		
		\item $C$ is defined in \eqref{eq:definition-coeff-matrix};
		
		\item $m_{\bmZ}(\od t, \od z) = \nu_{\bmZ}(\od z)  \Leb_{[0, T]}(\od t) $, where $\Leb_{[0, T]}$ is the restriction of the Lebesgue measure on $[0, T]$, $\nu_{\bmZ}$ is a L\'evy measure on $\bR^{\bfD}_0: = \bR^{\bfD}\backslash\{0\}$ with support on $\{0\}\times \bR^{2D}_0$, i.e. $\nu_{\bmZ}(\bR^{D^2+D}_0 \times \bR^{2D}_0) =0$, and such that $\nu_{\bmZ}(\{0\} \times B) = \nu_L^\psi(B)$ for $B \in \cB(\bR^{2D}_0)$.
	\end{itemize} 
	Note that $((W_{t \wedge T}, \cW_{t \wedge T}))_{t \in [0, \infty)}$ and $(L^\psi_{t \wedge T})_{t \in [0, \infty)}$ are independent due to \cref{lemm:independence-BM-Levy}. Then a standard calculation using L\'evy--Khintchine formula shows that $(\vect(W_{t \wedge T}, \cW_{t \wedge T}, L^\psi_{t \wedge T}))_{t \in [0, \infty)}$ is a (time-inhomogeneous) L\'evy process with characteristic triplet $(B, C, m_{\bmZ})$ with respect to the truncation function $h$. Hence, we derive from \cite[Theorem VIII.2.29]{JS03} the weak convergence $$(\bmZ^n_{t \wedge T})_{t \in [0, \infty)} \to (\vect(W_{t \wedge T}, \cW_{t \wedge T}, L^\psi_{t \wedge T}))_{t \in [0, \infty)}.$$ Eventually, since the limit process has no fixed time of discontinuity, we apply \cite[Theorem 16.7]{Bi99} to obtain the weak convergence on the time interval $[0, T]$ as desired. \qed

\begin{appendix}

\section{Proof of \cref{prop:decomposition-H}}\label{app:proof:prop:decomposition-H}
Condition $\bE\b[\int_{\bR^D} \|H_{n, i-1}(u)\|^2 \varphi_D(u) \od u\b] <\infty$ allows us to define
	\begin{align*}
		\mu^H_{n, i-1} &: = \int_{\bR^D} H_{n, i-1}(u) \varphi_D(u) \od u, \quad \wt H_{n, i-1}(u): = H_{n, i-1}(u) - \mu^H_{n, i-1}, \notag\\
		\Theta^H_{n, i-1} &: = \int_{\bR^D} \wt H_{n, i-1}(u)  \wt H_{n, i-1}(u)^\tran \varphi_D(u) \od u.
	\end{align*}
	Obviously $\mu^H_{n, i-1} \in \bfL_2(\bP)$. In addition, the finiteness of accumulative entropy implies that $\det(\Theta^H_{n, i-1})>0$ a.s. for all $n, i$.	Since $\Theta^H_{n, i-1} \in \bS^D_{++}$, we apply the spectral theorem for symmetric matrices to obtain a real diagonal matrix $\Lambda^H_{n, i-1} = \diag(\lambda_{1}(\Theta^H_{n, i-1}), \ldots, \lambda_{D}(\Theta^H_{n, i-1}))$ with $\lambda_{1}(\Theta^H_{n, i-1}) \ge \cdots \ge \lambda_{D}(\Theta^H_{n, i-1}) > 0$ and a $U^H_{n, i-1} \in \cO_D$, such that 
	\begin{align*}
		\Theta^H_{n, i-1} = U^H_{n, i-1} \Lambda^H_{n, i-1} (U^H_{n, i-1})^\tran.
	\end{align*}
	One remarks that $U^H_{n, i-1}$ and $\Lambda^H_{n, i-1}$ are matrices whose entries are $\cF_{n, i-1}$-measurable random variables. Now, by adjusting on a $\bP$-null set, we define
	\begin{align*}
		\vartheta^H_{n, i-1} & := (\Theta^H_{n, i-1})^{\frac{1}{2}} = U^H_{n, i-1} (\Lambda^H_{n, i-1})^{\frac{1}{2}} (U^H_{n, i-1})^\tran,\\
		\eta^H_{n, i} & : = U^H_{n, i-1} \hat \eta^H_{n, i}, \quad \mbox{where} \quad \hat \eta_{n, i}^{H, (d)} : = \frac{1}{\sqrt{\lambda_{d}  (\Theta^H_{n, i-1})}} (U^{H}_{n, i-1} \bme_d)^\tran \wt H_{n, i-1}(\xi_{n, i}),\quad  d = 1, \ldots, D.
	\end{align*}
	Then it is easy to check that $\vartheta^H_{n, i-1} \in \bfL_2(\bP)$. Moreover, for $d = 1, \ldots, D$, one has, a.s.,
	\begin{align*}
		[\vartheta^H_{n, i-1} \eta^H_{n, i}]^{(d)} & = \sum_{k=1}^D U^{H, (d, k)}_{n, i-1} \sqrt{\lambda_k(\Theta^H_{n, i-1})} \hat \eta^{H, (k)}_{n, i} = \sum_{k, l = 1}^D U^{H, (d, k)}_{n, i-1} U^{H, (l, k)}_{n, i-1} \wt H^{(l)}_{n, i-1}(\xi_{n, i}) \\
		& = \sum_{l=1}^D [U^H_{n, i-1} (U^H_{n, i-1})^\tran]^{(d, l)} \wt H^{(l)}_{n, i-1}(\xi_{n, i}) = \wt H^{(d)}_{n, i-1}(\xi_{n, i}),
	\end{align*}
	which shows $\vartheta^H_{n, i-1} \eta^H_{n, i} = \wt H_{n, i-1}(\xi_{n, i})$ a.s. For any $d = 1, \ldots, D$, we let $\hat \eta^{H, (d)}_{n, i}(\ep)$ be the random variable obtained by adding $\ep>0$ to $\lambda_d(\Theta^H_{n, i-1})$ in the definition of $\hat \eta^{H, (d)}_{n, i}$. Then one has, a.s., 
	\begin{align*}
		\bE\B[\b|\hat \eta^{H, (d)}_{n, i}(\ep)\b|^2\B| \cF_{n, i-1}\B] & = \frac{1}{\lambda_{d}(\Theta^H_{n, i-1})+ \ep}\bme_d^\tran (U^H_{n, i-1})^\tran \bE\B[\wt H_{n, i-1}(\xi_{n, i}) (\wt H_{n, i-1}(\xi_{n, i}))^\tran \B| \cF_{n, i-1}\B] U^H_{n, i-1} \bme_{d}\\
		& = \frac{1}{\lambda_{d}(\Theta^H_{n, i-1})+ \ep} \bme_d^\tran (U^H_{n, i-1})^\tran \Theta^H_{n, i-1}  U^H_{n, i-1} \bme_{d}  = \frac{\bme_d^\tran \Lambda^H_{n, i-1} \bme_{d}}{\lambda_{d}(\Theta^H_{n, i-1})+ \ep}  \\ 
		& = \frac{\lambda_{d}(\Theta^H_{n, i-1})}{\lambda_{d}(\Theta^H_{n, i-1})+ \ep}.
	\end{align*}
	Letting $\ep \downarrow 0$ yields  $\bE[|\hat \eta^{H, (d)}_{n, i}|^2 |\cF_{n, i-1}] = 1$ a.s. by the monotone convergence theorem, and thus, $\|\hat \eta^H_{n, i}\| \in \bfL_2(\bP)$ as a by-product. Analogously, we can show that $\bE[\hat \eta^{H, (d)}_{n, i} \hat \eta^{H, (d')}_{n, i} |\cF_{n, i-1}] = \1_{\{d = d'\}}$ a.s., which means that $\bE[\hat \eta^H_{n, i} (\hat \eta^H_{n, i})^\tran |\cF_{n, i-1}] = I_D$. Then we get
	$\bE[\eta^{H}_{n, i} (\eta^{H}_{n, i})^\tran | \cF_{n, i-1}] = I_D$ a.s., and hence, \eqref{eq:decomposition-H} follows. The uniqueness is straightforward.

\section{Some auxiliary results}

\subsection{Positive semidefinite matrices}
For matrices $A, B \in \bS^D$ we write $A \preceq B$ if $B - A \in \bS^D_+$.

\begin{lemm}[\cite{Ha74}, Sec.82, Exercises 12 and 13]\label{lem:matrix-inequality} \mbox{ }
	 \begin{enumerate}[\quad \rm(1)]
		\itemsep2pt
		\item For $A, B \in \bS^D_+$ with $A \preceq B$ one has $\det(A) \le \det(B)$.
		
		\item Let $A, B \in \bS^D_{++}$ with $A \preceq B$. Then $B^{-1} \preceq A^{-1}$ and $\Tr[A C] \le \Tr[BC]$ for any $C \in \bS^D_{+}$.
	\end{enumerate}
\end{lemm}

\subsection{Integrability for solutions of SDEs with jumps}

Although the following fact can be easily extended to a multidimensional setting, however, we formulate it in the one-dimensional case for the sake of simplicity. We refer to \cite{BT23S} for its proof. 

\begin{lemm}\label{lem:solution-SDE}
	Let $\xi = (\xi_t)_{t \in [0, T]}$ be c\`adl\`ag and adapted with $\|\xi\|_{\cS_2([0, T])}^2: = \bE[\sup_{0 \le t \le T} \xi_t^2] <\infty$. Assume that $\od Z_t = \phi_t \od t  + \od K_t$, where $K = (K_t)_{t \in [0, T]}$ is a c\`adl\`ag $\bfL_2(\bP)$-martingale satisfying $\od \<K, K\>_t = \eta_t^2 \od t$, where $\eta$ and $\phi$ are progressively measurable with $\sup_{0 < t < T} \eta_t^2 + \int_0^T \phi_t^2 \od t \le C$ a.s. for some (non-random) constant $C>0$. Then, for a Lipschitz function $\sigma\colon \bR \to \bR$, the SDE
	\begin{align*}
		X_t = \xi_t + \int_0^t \sigma(X_{u-}) \od Z_u, \quad X_0 = \xi_0 = x_0 \in \bR,
	\end{align*}
	has a unique c\`adl\`ag strong solution $X = (X_t)_{t \in [0, T]}$ satisfying  $\bE[\sup_{0 \le t \le T} X_t^2] \le C' <\infty$ for some constant $C' = C'(\|\xi\|_{\cS_2([0, T])}, T, \sigma, C)>0$.
\end{lemm}

\subsection{Independence of Gaussian and purely non-Gaussian L\'evy processes}
L\'evy processes in the following assertion are considered with the canonical truncation function $h(x) = x\1_{\{\|x\|\le 1\}}$. We refer to \cite{BT23S} for its proof.

\begin{lemm}\label{lemm:independence-BM-Levy}
Let $D, D' \in \bN$. Assume that $W$ is a $D$-dimensional Gaussian L\'evy process and $L$ is a $D'$-dimensional purely non-Gaussian L\'evy process, both defined on the same probability space. Then $W$ and $L$ are independent.
\end{lemm}

\end{appendix}


\bibliographystyle{amsplain}

\end{document}